\documentclass[10pt,reqno,english]{amsart}
\usepackage{amsfonts,amsmath,latexsym,verbatim,amscd,mathrsfs,color,array}
\usepackage[hmargin=3cm, vmargin=3cm]{geometry}

\usepackage{amsmath,amssymb,amsthm,amsfonts,graphicx,color}
\usepackage{amssymb}
\usepackage{pdfsync}
\usepackage{epstopdf}

\usepackage{color} 
\usepackage{enumerate} 
\theoremstyle{plain}  
\newtheorem{theorem}{Theorem}[section] 
\newtheorem{proposition}[theorem]{Proposition}
\newtheorem{lemma}[theorem]{Lemma}

\theoremstyle{definition}

\theoremstyle{remark}
\newtheorem{remark}[theorem]{Remark}

\newcommand{\ass}{\quad\mbox{as}\quad}
 \usepackage{graphicx}

\newcommand{\inn}{{\quad\hbox{in } }}
\newcommand{\onn}{{\quad\hbox{on } }}
\newcommand{\ttt}{\tilde }

\newcommand{\pp}{ {\partial} }

\newcommand{\R} {\mathbb R}

\newcommand{\cuad}{{\sqcap\kern-.68em\sqcup}}

\newcommand{\foral}{\quad\mbox{for all}\quad}
\newcommand{\ve}{\varepsilon}

\newcommand{\be}{\begin{equation}}
\newcommand{\ee}{\end{equation}}

\newcommand{\la}{\lambda}

\newcommand{\equ}[1]{(\ref{#1})}

\newcommand{\uin}{u_{\mbox {in}}}
\newcommand{\uout}{u_{\mbox {out}}}

\numberwithin{equation}{section}

\title[Infinite time blow-up for the 3d critical heat equation]{Infinite time blow-up   for the 3-dimensional energy  critical heat equation}

\author[M. del Pino]{Manuel del Pino}
\address{\noindent   Department of Mathematical Sciences University of Bath,
Bath BA2 7AY, United Kingdom \\
and  Departamento de
Ingenier\'{\i}a  Matem\'atica-CMM   Universidad de Chile,
Santiago 837-0456, Chile}
\email{m.delpino@bath.ac.uk}

\author[M. Musso]{Monica Musso}
\address{\noindent   Department of Mathematical Sciences University of Bath,
Bath BA2 7AY, United Kingdom \\
and Departamento de Matem\'aticas, Universidad Cat\'olica de Chile, Macul 782-0436, Chile}
\email{m.musso@bath.ac.uk}

\author[J. Wei]{Juncheng Wei}
\address{\noindent  Department of Mathematics University of British Columbia, Vancouver, BC V6T 1Z2, Canada
}  \email{jcwei@math.ubc.ca}

\begin{document}


\begin{abstract}
We construct  globally defined in time,  unbounded  positive solutions to the energy-critical heat equation in dimension three
$$
u_t = \Delta u + u^5 , \quad {\mbox {in}} \quad \R^3 \times (0,\infty), \ \ u(x, 0)= u_0 (x)\inn \R^3.
$$
For each $\gamma>1$ we find  initial data (not necessarily radially symmetric)  with
 $\lim\limits_{|x| \to \infty} |x|^\gamma u_0 (x)  >0$ such that  as $t \to \infty$
$$
 \| u(\cdot ,t ) \|_\infty \sim t^{\gamma-1 \over 2} , \quad {\mbox {if}} \quad 1<\gamma <2, \quad  \| u(\cdot ,t ) \|_\infty \sim  \sqrt{t}, \quad {\mbox {if}} \quad \gamma >2, \quad
$$
and
$$
\| u(\cdot , t)\|_\infty \sim \sqrt{t}\, (\ln t )^{-1}  , \quad {\mbox {if}} \quad \gamma = 2.
$$
Furthermore we show that this infinite time blow-up is co-dimensional one stable.  The existence of such solutions was  conjectured by Fila and King \cite{FK}.

\end{abstract}


\maketitle

\setcounter{equation}{0}
\section{Introduction}\label{intro}

Let  $n\ge 3$. The {\em energy critical heat equation} in $\R^n$  is the parabolic Cauchy problem
\begin{equation}
\label{p000}
\left\{\begin{matrix}
u_t = \Delta u + |u|^{\frac{4}{n-2}}u &  \inn & \R^n& \times (0,\infty) , \\ \\
u(\cdot,0) = u_0 &  \inn & \R^n.&
\end{matrix}
\right.
\end{equation}
The {\em energy }
$$
E(u) = \frac 12 \int_{\R^n} |\nabla u|^2 - \frac {n-2}{2n} \int_{\R^n} |u|^{\frac {2n}{n-2}}
$$
 defines a Lyapunov functional for Problem \equ{p000}. In fact for classical solutions $u(x,t)$ with sufficient decay in space variable we have that
$$
\frac{d}{d t} \, E(u(\cdot ,t))  =  - \int_{\R^n} |u_t|^2 .
$$
Classical parabolic theory yields that the Cauchy problem \equ{p000} is well-posed in its natural
finite-energy space for short time intervals.

\medskip
In this paper we are interested in {\bf positive finite-energy solutions} of \equ{p000} which are { global in time}, namely defined and smooth in the entire time interval $ (0,\infty)$. The presence of the Lyapunov functional implies that limits of bounded solutions along sequences $t=t_n\to +\infty$ can only be steady states, namely solutions of the Yamabe equation
\be\label{talenti}
\Delta u + |u|^{\frac 4{n-2}}u=0 \inn \R^n.
\ee
 All {\bf positive} solutions of \equ{talenti} are given by the {\em Aubin-Talenti  bubbles}
$$
U_{\mu,\xi} (x) =    \mu^{-\frac{n-2}2} w \left ( \frac {x-\xi}{\mu}\right) ,
$$
where $\mu>0$, $\xi\in \R^n$ and
$$
w(x)\ = \  (n(n-2))^{\frac{n-2}4} \left ( \frac 1 { 1+ |x|^2}    \right )^{\frac{n-2}2}  .
$$
They are precisely the extremals of Sobolev's embedding.  The {\em criticality} of Problem \equ{p000}
refers to the presence of this continuum of steady states which become singular as $\mu\to 0$, in addition to energy invariance. In fact we immediately see that
$$
E(U_{\mu,\xi})  = E(U)  \foral \xi\in \R^n,\  \mu>0.
$$
A solution $u(x,t)$ of \equ{p000} which looks around one or more points of space like
$ u(x,t) \approx  U_{\mu(t),\xi(t)} (x)     $  with $\mu(t) \to 0$ is called a bubbling blow-up solution. Bubbling phenomena is present in many important time-dependent and stationary setting, usually carrying deep meaning in the global structure of their solutions.  Notable examples include the Yamabe  and harmonic map flows and the Keller-Segel chemotaxis system.  (See \cite{CFM, DDS, rs1, DDW, GM} and the references therein.) In the last decade or so it has been extensively studied in energy-critical wave equations, Schrodinger maps and other dispersive settings.

\medskip
Problem \equ{p000} is a simple looking model which contains much of the complexity of the bubbling blow-up issue.  Basic questions have remain unanswered until today.  Existence or nonexistence of infinite time bubbling positive solutions in Problem \equ{p000} is not known. This question has been explicitly stated for instance in
\cite{PY2} and in \cite{QS}, Remark 22.10. Detecting such solutions rigorously is not easy. Usual behaviors in the flow \equ{p000} are either asymptotic vanishing
 or blow-up in finite time. Global solutions with nontrivial asymptotic patterns are typically unstable objects and hence harder to be detected.

 \medskip
 In a very interesting paper Fila and King \cite{FK} provided insight on the question in the case of
 a radially symmetric, positive  initial condition with an exact power decay rate.
 Using formal matching asymptotic analysis, they demonstrated that the power decay determines the blow-up rate in a precise manner.
 Intriguingly enough, their analysis leads them to conjecture that infinite time blow-up {\bf should only happen} in low
 dimensions 3 and 4, see  Conjecture 1.1 in \cite{FK}.

\medskip
 In this paper we rigorously establish the existence of solutions with infinite time blow-up
in dimension 3, confirming the conjecture in \cite{FK}.
Thus we consider the Cauchy problem
\begin{equation}
\label{p001}
\left\{\begin{matrix}
u_t = \Delta u + u^5 &  \inn & \R^3& \times (0,\infty) , \\ \\
u(\cdot,0) = u_0 &  \inn & \R^3,&
\end{matrix}
\right.
\end{equation}
for an initial datum $u_0$ which we assume first  radially symmetric with an exact
 power decay of the form
\be\label{gam}
\lim_{|x|\to \infty }  |x|^\gamma u_0(x )\ =:\ A >0. \quad
\ee
As in \cite{FK} we assume that $\gamma >1$ which means that $u_0$ decays faster than
the bubble
\begin{equation}
\label{bubble1}
w(x) = 3^{1\over 4} \, \left( {1\over 1+|x|^2} \right)^{1\over 2}.
\end{equation}

\medskip

\begin{theorem}\label{teo1}
Given $\gamma >1$, there exists a positive, radially symmetric global solution $u(x,t)$ to problem $\equ{p001}$
whose initial condition $u_0(|x|)$ satisfies $\equ{gam}$
and  as $ t \to +\infty$
\begin{equation}\label{alglin}
 \| u(\cdot ,t ) \|_{ L^\infty(\R^3) }   \sim   \left \{ \begin{matrix}   t^{\frac{\gamma-1}2} & \quad {\mbox {if}} \quad & \quad\  1<\gamma <2, \\  \frac {\sqrt{t}}{ \ln t}   &\quad {\hbox {if}} \quad  &\gamma = 2,
 \\  \sqrt{t} & \quad {\mbox {if}} \quad & \gamma >2.
  \end{matrix} \right.
\end{equation}

\end{theorem}

More precisely, the blow-up takes place by bubbling near the origin.
The solution of Theorem \ref{teo1} is in  the inner self-similar region,
$ |x| \ll  \sqrt{t}$,  in leading order of the {\em bubbling blow-up form}
$$
u(x,t) \sim  \frac 1{\mu(t)^{\frac 12} } w\left ( \frac {x}{\mu(t)} \right ),
$$
where
\begin{equation}\label{bath1}
\mu (t) \ \sim \
  \left \{ \begin{matrix}   t^{{1-\gamma}} & \quad {\mbox {if}} \quad & \quad\  1<\gamma <2, \\ t^{-1}{ \ln^2 t}   &\quad {\hbox {if}} \quad  &\gamma = 2,
 \\   t^{-1}  & \quad {\mbox {if}} \quad & \gamma >2
  \end{matrix} \right.
\end{equation}
and $w$ is given by \equ{bubble1}.
In the {outer self-similar region} $ |x| \gg  \sqrt{t}$, the solution dissipates in the form of a  self-similar solution of heat equation
$
u_t=\Delta u
$
in $\R^3 \times (0,\infty).$

\medskip
 A  surprising  feature of the construction is the dynamics discovered for the scaling parameter $ \mu (t)$. It has a highly non-local character governed by a equation involving a perturbation of the fractional $\frac{1}{2}$-Caputo derivative. In fact, in order to find the precise lower order corrections needed for the scaling parameter $\mu(t)$ we will need to solve linear equations of the type
 $$
\int_{0}^t {\beta'  (s) \over \sqrt{t-s}} \, \left( 1- e^{-{M^2 \over  (t-s)}}  \right) ds = h(t),
$$
for suitably decaying right hand sides $h(t)$. See (\ref{mp1}) and (\ref{mp11}) below.

\medskip
Problem \equ{p000} is a special  case of the {\em Fujita equation}
\begin{equation}
\label{1.1}
\left\{\begin{matrix}
u_t = \Delta u + u^p &  \inn & \R^n& \times (0,\infty) , \\ \\
u(\cdot,0)= u_0 &  \inn & \R^n&
\end{matrix}
\right.
\end{equation}
with $ p>1$. 
Blow-up phenomena in Problem \equ{1.1} is extremely sensitive to the values of the exponent $p$.
 A vast literature has been devoted to this problem after Fujita's seminal work \cite{F}. We refer the reader for instance to the book \cite{QS} for background and a comprehensive account of results until 2007 and to the more recent works \cite{MM1,MM2,MM3} and references therein. The case $p=\frac{n+2}{n-2}$ is special in many ways. Positive steady states do not exist when $p< \frac{n+2}{n-2}$. Positive radial global solutions must be bounded and go to zero, see
\cite{PQ1, PQ2, QS}. They exist when $p> \frac{n+2}{n-2}$ but they have infinite energy, see
\cite{GNW1}. Infinite time blow-up exists in that case but it has an entirely different nature, see \cite{PY1, PY2}.

\medskip
The study of energy critical problems has attracted much attention in the last decade. For energy-critical wave equations,  blow-up solutions have been characterized and constructed in   \cite{DKM1, DKM2, DKM3, DKM4, KST}.  In \cite{Sch} Type-II sign changing, finite time blow-up for \equ{p000} is constructed, first formally predicted in \cite{FHV}. Threshold dynamics around the steady states of \equ{p000}  has been  characterized in large dimensions $n\ge 7$ in \cite{CMR}.
Also in large dimensions $n\ge 5$ in \cite{CDM} infinite time bubbling solutions of \equ{p000} in a bounded domain under Dirichlet boundary conditions are constructed for $n\ge 5$. The cases $n=3,4$ are indeed considerably more delicate and not treated there. The solutions in Theorem \ref{teo1} are specially meaningful for the full dynamics since they are {\em threshold solutions} in the sense that
the solution of \equ{p001} with initial condition $\la u_0$ goes to zero as $t\to \infty$ if $\la<1$
while it blows-up in finite time if $\la >1$. Radial threshold solutions for various ranges of exponents in \equ{p001} are analyzed in \cite{QS}.

\medskip
We recall that from \cite{FK}, it is not expected
to have this blow-up in entire space in dimensions $n\ge 5$. Our approach is entirely different from that in \cite{Sch} for $n=4$ in which a finite-time type II blow-up solution of \equ{p000} is constructed on the basis of the modulation equation methods developed for critical dispersive equations in  \cite{DK,O,MRR, RR1, rs1}.

 \medskip
  Our approach has a parabolic-elliptic flavor, in line with the recent works  \cite{CDM, DDW}. Since our proofs only rely on elliptic and parabolic estimates, we can easily modify the proof to deal with nonradial and general initial data, in particular establishing {\em codimension 1 stability} of the solution built. This is concordant with a result on \cite{KNS}  on the corresponding wave
analogue. In Section \ref{nr} we prove the following

  \begin{theorem}\label{teonr} Let $\bar v_0 = \bar v_0 (x) $ be a positive continuous function, uniformly bounded for $x \in \R^3$. Let $\gamma >1$ and $\kappa > \max \{{\gamma + 3 \over 2} , \gamma \}$.
 Then, there exists a positive  global solution $u(x,t)$ to problem $\equ{p001}$
with initial condition
$$
u(x,0) = u_0 (|x| ) + \frac{\bar v_0 (x)}{ |x|^\kappa} \left[1-\eta \left(\frac{|x|}{t_0}\right)\right]
$$
where $u_0$ is positive, radially symmetric, satisfies \eqref{gam}, $ t_0>0$ is a fixed large number and $\eta $ is a smooth cut-off function with $ \eta (s)=1$ for $s<1$ and $ \eta (s)=0$ for $s>2$. As $ t \to +\infty$, $u (x,t)$ satisfies \eqref{alglin}.

  \medskip
  \noindent
  Furthermore, there exists a codimension $1$ manifold of functions in  $C^1 (\R^3 )$  converging to $0$ at infinity with a sufficiently fast decay, that contains $u_0 (|x| ) + {\bar v_0 (x) \over |x|^\kappa}  (1-\eta (\frac{|x|}{t_0})) $ such that if
  $\bar u_0$ lies in that manifold and it is sufficiently close to $u_0 (|x| ) + {\bar v_0 (x) \over |x|^\kappa }(1-\eta (\frac{|x|}{t_0} ))$ in the sense that $ \bar u_0= u_0 (|x| ) + {\bar v_0 (x) \over |x|^\kappa } (1-\eta (\frac{|x|}{t_0})) + {\mathcal O} (|x| e^{-b |x|})$ for some $b>0$,  then the solution
  $\bar u (x, t) $ to \eqref{p001} with $\bar u (x,0 ) = \bar u_0 (x)$ is global in time and satisfies \eqref{alglin}.
   \end{theorem}

\medskip
In the non-radial setting, the profile of the solution in the inner self-similar regime is
$$
u(x,t) \sim  \frac 1{\mu(t)^{\frac 12} } w\left ( \frac {x - p(t) }{\mu(t)} \right ), \quad {|p(t) | \over \mu(t)} \to 0, \quad {\mbox {as}} \quad t \to \infty
$$
where $w$ is given by \equ{bubble1} and $\mu$ satisfies the asymptotics \eqref{bath1}. Precise description of the dynamics of the center $p=p(t)$ is provided.

\medskip
 A  surprising  feature of the construction is the dynamics discovered for the scaling parameter $ \mu (t)$. It has a highly non-local character governed by a equation involving a perturbation of the fractional $\frac{1}{2}$-Caputo derivative. In fact, in order to find the precise lower order corrections needed for the scaling parameter $\mu(t)$ we will need to solve linear equations of the type
 $$
\int_{0}^t {\beta'  (s) \over \sqrt{t-s}} \, \left( 1- e^{-{M^2 \over  (t-s)}}  \right) ds = h(t),
$$
for suitably decaying right hand sides $h(t)$. See (\ref{mp1}) and (\ref{mp11}) below.

  \medskip
  We believe that an approach similar to that in this paper could be used to prove the existence of global unbounded solution when $N=4, p=3 $ as conjectured in \cite{FK}. We will undertake that issue in a future work.

\medskip

The proof of Theorem \ref{teo1} starts with the construction of an approximate solution to Problem \eqref{p001} with the asymptotic behavior described in
 \eqref{alglin}. This is done in full details in Section \ref{ansatz1}. We then show the existence of an actual solution to Problem \eqref{p001} deforming the approximation, by means of a {\it inner-outer gluing} procedure. This scheme  is described in Section \ref{iogluing}, and its proof is addressed in
 Sections \ref{outerP} to \ref{inner0}. In Section \ref{nr} we prove Theorem \ref{teonr}. Sections \ref{appeA} to \ref{appeC} gather some technical results needed to prove the Theorems.

\medskip
In the rest of the paper, we shall denote by $C$ a generic positive constant, whose value may change from line to line, and within the same line.
We shall use the notation ${\bf c}$ to indicate a positive constant, with ${\bf c} <1$, whose explicit value may change from line to line.
Furthermore, $t_0$ will denote a large fixed positive number and
\begin{equation}\label{defeta}
\eta : \R \to \R,
\end{equation}
a smooth cut-off function with $\eta(s) =1$ for $s<1$ and $=0$ for $s>2$.

\medskip

\medskip
\noindent
{\bf Acknowledgements:}
We are indebted to Marek Fila for introducing this problem to us and for many useful discussions.
 M.~del Pino and M. Musso have been  partly supported by grants
 Fondecyt  1160135,  1150066, Fondo Basal CMM and
Millenium Nucleus CAPDE NC130017. The  research  of J.~Wei is partially supported by NSERC of Canada.

\setcounter{equation}{0}
\section{Construction of an approximate solution and estimate of the associated error}\label{ansatz1}

After shifting the initial time to $t_0 >0$, Problem \eqref{p001} takes the form
\begin{equation}
\label{p}
u_t = \Delta u + u^5 , \quad {\mbox {in}} \quad \R^3 \times (t_0 ,\infty) ,
\end{equation}
with initial condition $u_0 (r) = u (r, t_0 )$ satisfying
\be \label{IC}
\lim_{r \to \infty} r^\gamma u_0 (r)  = A >0, \quad {\mbox {for some}} \quad \gamma >1.
\ee
This section is devoted to the construction of a first approximation for a solution to \eqref{p}-\eqref{IC}, and to the description of the associated error.

\medskip

The first approximation is build by  matching an inner profile, made upon solving the elliptic problem
\begin{equation} \label{yamabe}
\Delta u + u^5 = 0 \quad {\mbox {in}} \quad \R^3,
\end{equation}
and an outer profile, made upon solving the heat equation in the whole space
\begin{equation} \label{heat}
u_t=\Delta u  \quad {\mbox {in}} \quad \R^3,
\end{equation}
in the set of functions satisfying the decaying conditions \eqref{IC}.
It is constructed in Subsections \ref{sc1} (for the inner profile), \ref{sc2} (for the outer profile), and in Subsection \ref{sc3} we derive a precise description of the {\it error} of approximation. In \cite{FK}, this approximate solution was already derived. We realize though that, for our rigorous construction to work, we need a further improvement of the approximation.  This is done in Subsection \ref{sc4}, where we introduce a next correction term, and describe the associated error. It turns out that this next correction term gives the right dynamics for the blow-up rate which turns out to be governed by a  nonlocal differential equation with a fractional time-derivative closely related to the so-called $1/2$-Caputo derivative. See (\ref{mp11}).

\medskip

\subsection{Construction of the first inner profile.} \label{sc1}
We recall that
all positive radially symmetric  solutions to \eqref{yamabe} constitute a one-parameter family of functions, which are given explicitly by
\begin{equation}
\label{bubble}
w(r) = 3^{1\over 4} \, \left( {1\over 1+r^2} \right)^{1\over 2}, \quad w_\mu (r)=
\mu^{-{1\over 2}} w ({r\over \mu} ),
\end{equation}
for any positive number $\mu >0$. (See \cite{aubin,CGS}.)
We denote by  $Z_0$  the only bounded and radial function belonging to  the kernel of the linear operator
\begin{equation}\label{defL0}
L_0 (\phi ) = \Delta \phi + 5 w^4 \phi .
\end{equation}
See \cite{rey}. The function $Z_0$ is explicitly defined by
\begin{equation}\label{Z0def}
Z_0 (r) =  -\left[ {w\over 2} + w'(r) r  \right]= {3^{1\over 4} \over 2} \,  {r^2 -1 \over (1+ r^2)^{3\over 2}}.
\end{equation}
Given $Z_0$, we denote by  $\Phi_1 (r) $ the solution to
\begin{equation}
\label{ubc1}
\Delta \Phi_1 + 5 w^4 \Phi_1 = Z_0,
\end{equation}
defined as
\begin{equation}
\label{expaPhi1}
\Phi_1 (r) = \Phi_0 (r) + \pi_0 + \bar \Phi_1 (r) , \quad {\mbox {where}} \quad \Phi_0 (r) = {3^{1\over 4} \over 4 \, } \, r,
\end{equation}
$$
\left( 5 \int_0^\infty w^4 Z_0 r^2 \, dr \right) \, \pi_0 =
\int_0^\infty (Z_0 - { 3^{1\over 4} \over 2r} ) Z_0 r^2 \, dr  -5 \int_0^\infty w^4 \Phi_0 Z_0 r^2 \, dr
$$
and $\bar \Phi_1 $ being the unique solution to
$$
\Delta \phi + 5 w^{p-1} \phi = \underbrace{(Z_0 -  { 3^{1\over 4} \over 2r} )  - 5 w^4 (\Phi_0 + \pi_0 )}_{:= \Pi_0 (r) } ,
$$
explicitly given by
$$
\bar \Phi_1 (r) =
\tilde Z (r) \int_0^r \Pi_0 (s) Z_0 (s) s^2 \, ds - Z_0 (r) \int_0^r \Pi_0 (s) \tilde Z(s) s^2 \, ds.
$$
In the above expression,  $\tilde Z$ denoted another solution to  $\Delta \phi + 5 w^4 \phi = 0$, linearly independent to $Z_0$. $\tilde Z$ satisfies the asymptotic behavior $\tilde Z(s) \sim s^{-1}$, as $s \to 0$, and
$\tilde Z(s) \sim 1$, as $s \to \infty$.

\medskip

A closer look at the expression of $\bar \Phi_1$ gives that,
$$
\| r^{2-\sigma} \bar \Phi_1 (r) \|_\infty < C,
$$
for some fixed positive constant $C$, and any $\sigma >0$ small.

\medskip

\begin{remark} The solution to (\ref{ubc1}) is not unique. (In fact one can add any multiple of $Z_0$.) The choice we made in (\ref{expaPhi1}) is used to match the outer solution in the next section.

\end{remark}

\medskip
We have now the elements to define the first inner profile.
We introduce a  smooth positive function $\mu (t)$ of the form
\begin{equation}
\label{rho1}
\mu (t) = \mu_0 (t) \left(1+ \Lambda (t) \right)^2, \quad {\mbox {where}} \quad \mu_0 (t) >0, \quad \lim_{t \to \infty} \mu_0 (t) = 0.
\end{equation}
The function $\mu_0$ will be defined below, (see (\ref{rho0def1}), (\ref{rho0def1gamma1}), (\ref{rho0def1gamma2})),  as an explicit function of $t$ depending on the decay rate $\gamma$. On the other hand, the
function $\Lambda = \Lambda (t)$ will be left as a parameter in the construction, and it will be determined in the final argument to get an actual solution to the problem. In the meanwhile, we shall assume that $\Lambda = \Lambda (t)$ is a smooth function in $(t_0 , \infty)$, defined by
\begin{eqnarray}\label{lambdabound}
\Lambda (t) &:&= \int_t^\infty \lambda (s) ds , \quad {\mbox {where}} \quad \la \quad {\mbox {satisfies}} \nonumber \\
\| \la \|_\sharp &:&= \sup_{t > t_0 } \,  \mu_0 (t)^{-1} t \, \left[ \| \la \|_{\infty , [t, t+1]}   + [\la ]_{0, \sigma , [t,t+1]} \right] \,  \leq \ell,
\end{eqnarray}
for $\sigma = {1\over 2} + \sigma'$, with $\sigma' >0$ small, and for some fixed constant $\ell$. Here we intend
$$
\| f \|_{\infty , [t, t+1]} = \sup_{s\in [t, t+1] } |f(s) |, \quad [ f ]_{0, \sigma, [t, t+1]}=  \sup_{s_1 \not= s_2 \in [t, t+1] } { |f(s_1 ) - f(s_2 ) |
\over |s_1 - s_2|^\sigma}.
$$
For later purpose we introduce the space
\begin{equation}\label{Xsharp}
X_\sharp =\{ \la \in C (t_0, \infty) \, : \, \| \la \|_\sharp \quad {\mbox {is bounded}} \}.
\end{equation}
With this in mind, we define the inner approximation to be
\begin{equation}\label{maso0}
u_{{\mbox {in}} }(r,t ) = w_\mu ( r ) + \mu_0' \psi_1 (r, t), \quad \psi_1 (r , t) = \mu^{1\over 2} \Phi_1 ({r\over \mu} ).
\end{equation}
A direct computation gives that
$$
\Delta \psi_1 + 5 w_\mu^4 \psi_1 =- \mu^{-{3\over 2}} Z_0 ({r\over \mu} ) =  {\partial
w_{\mu} \over \partial \mu } (r).
$$
In the region $\{ r \, : \, r >R \mu_0 \} $, where $R$ is any large but fixed positive number, the inner approximation looks like
\begin{equation}
\label{expain}
\uin (r,t ) = 3^{1\over 4} {\mu^{1\over 2} \over r} -  {3^{1\over 4} \over 4} \, \mu_0' \mu^{-{1\over 2}} \, r + \mu_0^{1\over 2} \mu_0'  \, \Theta [\mu ] (r,t) + {\mu_0^{1\over 2} \over r}
\left( {\mu_0 \over r } \right)^2  \, \Theta [\mu] (r,t)
\end{equation}
where $\Theta [\mu ] (r,t)$ denotes a generic function, which depends smoothly on $\mu$, and on $(r,t)$, and which is uniformly bounded, for parameters $\mu$ satisfying \eqref{rho1},
for $r$ in the considered region, and any $t$ large.

\medskip

\subsection{Construction of the first  outer profile and choice of $\mu_0 (t)$.} \label{sc2}

The outer profile is chosen to satisfy  the heat equation
$
u_t = \Delta u $, in the whole space $\R^3$, and to fit the requested decaying property for the initial condition \eqref{IC}. Its properties and exact definitions change depending on the value of the decay rate $\gamma $  of the initial condition $u_0$, see \eqref{IC}. We consider three different situations: $1< \gamma < 2$, $\gamma =2$ and $\gamma >2$.

\medskip
\noindent
{\bf Case $1<\gamma <2$}. \ \ In this case we define $\uout$ as
\be \label{IC2}
\uout (r, t ) = t^{-{\gamma \over 2}} g({r \over \sqrt{t}})
\ee
with $g$ the positive solution to
\begin{equation}\label{ubc2}
g'' (s) + \left( {2\over s}  + {s\over 2} \right) g'(s) + {\gamma \over 2} g(s)  =0 \quad s \in (0, \infty)
\end{equation}
that satisfies the properties
\begin{enumerate}
\item
 $\lim_{s \to \infty} s^\gamma g(s) = A$,
\item  $\lim_{s \to 0^+} s g(s) = d$, for a certain positive constant $d $ for which $\lim_{s \to 0^+} \left[ g(s) - {d \over s} \right]= 0$.
\end{enumerate}
Such a function $g$ indeed exists. Let
$$
L_\nu (g) = g'' + ( {2\over s} + {s\over 2} ) g'+ \nu g, \quad s \in (0, \infty).
$$
In Section \ref{appeA}, we prove the following

\begin{lemma}\label{uno}
If ${1\over 2}< \nu <1$,
there exist two positive linearly independent solutions $y_1 = y_1 (s) $ and $y_2 = y_2 (s)$  to
\begin{equation}\label{app0}
L_\nu (g) = 0, \quad s \in (0, \infty)
\end{equation}
that satisfy respectively
\begin{equation}\label{app11}
y_1 (s)  = {1\over s} + (\nu -1 )  \left(\int_0^\infty s y_1 (s) \, ds \right) + {1 - 2 \nu \over 4}  s + O(s^2),    \quad {\mbox {if}} \quad s \to 0^+,
\end{equation}
\begin{equation}\label{app111}
  y_2 (s)  = c_2 + o(s)   \quad {\mbox {if}} \quad s \to 0^+,
\end{equation}
\begin{equation}\label{app12}
y_1 (s) =   c_1 e^{-{s^2 \over 4}} \, s^{4\nu -3 } , \quad y_2 (s) =   {1\over s^{2\nu}} (1+ o({1\over s} ) ) \quad {\mbox {if}} \quad s \to \infty, \end{equation}
for some positive constants $c_1$, $c_2$.
\end{lemma}

\medskip
Thanks to the Lemma, which we apply to solve \eqref{ubc2} when $\nu = {\gamma \over 2}$, we get that
the function $g$ we are looking for in \eqref{IC2} is thus given by
\begin{equation}
\label{defg1}
g(s ) = d y_1 (s) + A y_2 (s), \quad {\mbox {with}} \quad
d = {2 A y_2 (0) \over (2-\gamma) \left(\int_0^\infty s y_1 (s) \, ds \right)} >0.
\end{equation}
We observe that, in a region like $ r < R^{-1} \sqrt{t}$, for some large but fixed $R$, we get
\begin{equation}
\label{expaout1}
\uout (r,t ) = d {t^{-{\gamma -1 \over 2}} \over r} + t^{-{\gamma +1 \over 2}} \,
A  \, {(1-\gamma )   y_2(0) \over 2 (2-\gamma ) \int_0^\infty z y_1 (z) \, dz  }  \, r + t^{-{\gamma \over 2}} O({r^2 \over t} ).
\end{equation}

\medskip
\noindent
We next choose the function $\mu_0 (t) $ in the definition of $\mu (t)$, \equ{rho1},
in such a way that the functions $\uin$ and $\uout$ automatically match in the whole region $R \mu_0 < r < R^{-1} \sqrt{t}$, for some $R$ large, but fixed independent of $t$.
This is possible if
\begin{equation}\label{rho0def1}
\mu_0 (t)  =  {d^2 \over \sqrt{3} } \, t^{1-\gamma}.
\end{equation}
Indeed, with this choice for $\mu_0 (t)$, and given the bound \eqref{lambdabound},  there exists a constant $C$ so that
\begin{equation}
\label{diffues}
\left| \uin (r,t) - \uout (r,t) \right|  \leq C   { \mu_0^{1\over 2}\over r}, \quad
\left| \nabla \uin (r,t) -\nabla  \uout (r,t) \right|  \leq C   { \mu_0^{1\over 2} \over r^2}
\end{equation}
for any $R \mu_0  < r < R^{-1} \sqrt{t}$, and $t$ large enough.

\medskip
\noindent
{\bf Case $\gamma =2$.} \ \ In this case, we define $\uout$ as
\begin{equation}\label{uout2}
\uout (r, t ) = t^{-1} (\log t ) k A g_0 ({r \over \sqrt{t}}) + t^{-1} h({r\over \sqrt{t}} )
\end{equation}
where $g_0 (s ) = s^{-1} e^{-{s^2 \over 4}}$ is a solution to
\begin{equation}
\label{ubc5}
g'' (s) + \left( {2\over s}  + {s\over 2} \right) g'(s) + g(s)  =0
\end{equation}
and $h$ solves
\begin{equation}\label{ubc22}
h'' (s) + \left( {2\over s}  + {s\over 2} \right) h'(s) + h(s)  =k A g_0 (s)
\end{equation}
with $\lim_{s \to \infty} s^\gamma h(s) = A$, and $\lim_{s \to 0^+} s h(s) = d$,  so that $\lim_{s \to 0^+} \left[ h(s) - {d \over s} \right] = 0$. The function $h$ can be described explicitly. Let $g_1 (s) =s^{-1} e^{-{s^2 \over 4}} \int_0^s e^{z^2 \over 4} dz$. This function  solves \equ{ubc5}. Since $g_1$ and $g_0$ are linearly independent, the variation of parameters formula gives that, for any constants $d $ and $b$
\begin{equation}\label{mitrano}
h(s) = g_0 (s) \left[ d - k A \int_0^s z g_1 (z) \, dz \right] + g_1 (s) \left[ b +
k A \int_0^s z g_0 (z) \, dz \right]
\end{equation} solves \equ{ubc22}. In order to have $\lim_{s \to \infty} s^\gamma h(s) = A$, we need $ 2 \left[ b +
k A \int_0^\infty z g_0 (z) \, dz \right] = A$. Furthermore, to have $\lim_{s \to 0^+} \left[ h(s) - {d \over s} \right] = 0$, we need $b= 0$. Thus we select
\begin{equation}
\label{betak}
b = 0 , \quad k = {1\over 2 \int_0^\infty z g_0 (z) \, dz }.
\end{equation}
Observe that, up to this moment, the constant $d $ is arbitrary. Nevertheless, we remind that $\uout$ wants to be a solution to $u_t = \Delta u = u_{rr} + {2\over r} u_r$. Multiplying this equation by $r$, and integrating in $(0, R)$, for some fixed, large $R$, we get
$$
{d \over dt} \left( \int_0^R r u(r,t) \, dr \right) = R u_r (R,t ) + u(R,t),
$$
where we use the fact that $\lim_{r \to 0} [ r u_r (r,t ) + u(r,t)] =0.$ Next, we integrate the above equation in $t$, from $0$ to $\infty$, and using the fact that
$\lim_{t \to \infty} \int_0^R r u(r,t) \, dt = 0$, we get
\begin{equation}
\label{ubc6}
-\int_0^R r u(r,0) \, dr = \int_0^\infty [ R u_r (R,t) + u(R,t) ] \, dt.
\end{equation}
Take now $u= \uout $ and compute the right hand side of \eqref{ubc6}
\begin{align*}
\int_0^\infty [ R u_r (R,t) + u(R,t) ] \, dt \ =&\ A k
	\int_0^\infty  t^{-1} (\log t ) [ {R \over \sqrt{t}} g_0'( {R \over \sqrt{t}}) + g_0 ({R \over \sqrt{t}}) ] \, dt
	\\     +  &  \int_0^\infty  t^{-1} [ {R \over \sqrt{t}} h'( {R \over \sqrt{t}}) + h ( {R \over \sqrt{t}} ) ] \, dt \quad s:= {R \over \sqrt{t}} \\
=	&\ \left( 4 A k  \int_0^\infty s^{-1} [ s g_0' (s) + g_0 (s) ] \, ds \right) \log R  \\
+ & \bar d +
  \left( 2 \int_0^\infty s^{-1} [ s h' (s) + h (s) ] \, ds \right)
\end{align*}
where $\bar d$ is  the constant defined by
$$\bar d=- \left( 4 A k \int_0^\infty s^{-1} (\log s) [ s g_0' (s) + g_0 (s) ] \, ds \right) .$$ We can simplify the  expression of the constant in front of $\log R$. Indeed, multiplying \equ{ubc5} against $s$, we get that $(s g'(s) + g + {s^2 \over 2} g)'=0$. For $g=g_0$, and using the fact that $g_0$ decays very fast as $s \to \infty$, we get that $ s g_0'(s) + g_0 (s) = -{s^2\over 2} g_0 (s) $ for any $s$, thus
$$
4 A k  \int_0^\infty s^{-1} [ s g_0' (s) + g_0 (s) ] \, ds = A k \left( - 2 \int_0^\infty s g_0 (s) ds \right) =- A
$$
since \equ{betak}.
On the other hand, the decaying condition $\lim_{r \to \infty} r^2 u(r,0) = A$ gives
$$
-\int_0^R r u(r,0)\, dr = - A \log R + B (R),
$$
with $\lim_{R\to \infty} B (R) = B$, being $B$ a real constant. Plugging this information in \equ{ubc6}, we get that
$$
\bar d+  \left( 2 \int_0^\infty s^{-1} [ s h' (s) + h (s) ] \, ds \right) = B.
$$
This last relation defines in a unique way the constant $d >0$ in the definition of $h$, \eqref{mitrano}.
Indeed, a direct computation gives that
$$
\int_0^\infty s^{-1} [ s h' (s) + h (s) ] \, ds = - {d \over 2}  \left( \int_0^\infty s g_0 (s) \, ds \right) + \omega,
$$
with
$$
\omega= {kA\over 2} \int_0^\infty s g_0 (s) (\int_0^s z g_1 (z) \, dz) \, ds +
\int_0^s s^{-1} [sg_1'  + g_1] ( kA \int_0^s z g_0 (z) \, dz ) \, ds,
$$
from which we deduce that
$$
d = {\bar d-2\omega-B \over \int_0^\infty s g_0 (s) \, ds }.
$$
With this choice for the function $h$ in \eqref{uout2}, we get
$$
h(s) = {d \over s} - {s\over 4} [d + 10 k A] + O(s^3 ) , \quad {\mbox {as}} \quad s\to 0^+
$$
and
\begin{eqnarray}\label{expaout2}
\uout (r,t )  &= & {t^{-{1\over 2}} \over r} \, \left[kA (\log t) + d \right]    \\
& + & t^{-1} \left[ -{kA (\log t) \over 4} - {d + 10 k A \over 4} \right] \, {r \over \sqrt{t}} + O \left((\log t)  \,  {r^3 \over t^3 \sqrt{t}} ) \right)\nonumber
\end{eqnarray}
 in the region  $r < R^{-1} \sqrt{t}$, for some large but fixed $R$, as $t \to \infty$.

\medskip
In this case, namely when $\gamma =2$, we  choose $\mu_0$ in \eqref{rho1} as
\begin{equation}\label{rho0def1gamma1}
\mu_0 (t)  =  {[d + k A (\log t)]^2 \over \sqrt{3} } \, t^{-1},
\end{equation}
and thanks to this choice, and to the bound \eqref{lambdabound} on $\la$,  we find a constant $C$ so that
\begin{equation}
\label{diffuesI}
\left| \uin (r,t) - \uout (r,t) \right|  \leq C  { \mu_0^{1\over 2} \over r}, \quad
\left| \nabla \uin (r,t) -\nabla  \uout (r,t) \right|  \leq C { \mu_0^{1\over 2} \over r^2}
\end{equation}
for any $R \mu_0  < r < R^{-1} \sqrt{t}$, for some fixed and large $R$, and for all $t$ large enough.

\medskip
\noindent
{\bf Case $\gamma>2 $.} \ \ In this case, we define $\uout^1 $ as
$$
\uout^1  (r, t ) = t^{-1} \, d  g_0 ({r \over \sqrt{t}}), \quad d = \left( {\int_0^\infty r u_0 (r)  dr \over \int_0^\infty s g_0 (s) \, ds } \right)
$$
where $g_0 (s ) = s^{-1} e^{-{s^2 \over 4}}$ solves \equ{ubc5}, and $u_0(r)$ is the initial condition for \eqref{p}-\eqref{IC}.
Observe that, in a region like $ r < R^{-1}  \sqrt{t}$, for some large but fixed $R$, we get
\begin{equation}
\label{expaout3}
\uout^1 (r,t ) = d\,  {t^{-{1 \over 2}} \over r} - t^{-1} \,
{d \over 4} \, {r \over \sqrt{t}}  + t^{-1} O({r^2 \over t^{3\over 2}} ).
\end{equation}

\medskip
\noindent
For a given time $t$, the function $\uout^1$ is decaying very fast as $r\to \infty$. For this reason, we modify
$\uout^1$ with a function that has the right decay to match the initial condition $u_0(r)$,  for $r$ large. Define
\begin{equation}
\label{vava}
\uout (r,t) = \eta ({r \over t } ) \uout^1 (r,t) + (1-  \eta ({r \over t } )) \uout^2 (r) , \quad {\mbox {with}} \quad \uout^2 (r) = {A \over r^\gamma},
\end{equation}
where $\eta $ is the cut off function defined in \eqref{defeta}.

\medskip
\noindent
In this case, $\gamma >2$, we choose $\mu_0$ in \eqref{rho1} as
\begin{equation}\label{rho0def1gamma2}
\mu_0 (t)  =  {d^2 \over \sqrt{3} } \, t^{-1}.
\end{equation}
With this choice for $\mu_0 (t)$, and thanks to \eqref{lambdabound}, given any large but fixed number $R>0$, there exists a constant $C$ so that
\begin{equation}
\label{diffues2}
\left| \uin (r,t) - \uout (r,t) \right|  \leq C { \mu_0^{1\over 2}  \over r}, \quad
\left| \nabla \uin (r,t) -\nabla  \uout (r,t) \right|  \leq C { \mu_0^{1\over 2}  \over r^2}
\end{equation}
for any $R \mu_0  < r < R^{-1} \sqrt{t}$, and for all $t$ large.

\medskip

\subsection{Construction of the first global approximation and estimate of the error.} \label{sc3}

\medskip
\noindent
Let $r_0 >0$ be a small and fixed number, 
define
\begin{equation}
\label{DefU}
U_1 (r,t) = \eta ({r\over r_0 \sqrt{t}} ) \uin (r,t) + \left(1-\eta ({r\over r_0 \sqrt{t} } ) \right) \uout (r,t)
\end{equation}
where $\eta$ is given by \eqref{defeta}.
For any smooth function $u = u(r,t)$, we define the Error Function as
\begin{equation}
\label{error}
{\mathcal E} [u] (r,t ) = \Delta u + u^5 - u_t .
\end{equation}

Our next purpose is to describe
\be \label{error1}
{\mathcal E}_1 (r,t) = {\mathcal E} [U_1 ] (r,t)
\ee
with $U_1$ given by \equ{DefU}. To this end, we introduce the function $\alpha = \alpha (t)$, $t >t_0$,
\begin{equation}\label{defaaa}
\alpha (t)  = 3^{1\over 4}  \, \mu_0^{-{1\over 2}} \, \left( \mu_0 \Lambda \right)'.
\end{equation}
Since $\Lambda$ satisfies \eqref{lambdabound}, definition \eqref{defaaa} defines a linear homeomorphism
${\mathcal A} : X_\sharp \to X_\flat$, ${\mathcal A} (\la ) = \alpha$, where
\begin{equation}\label{Xflat}
X_\flat =\{ \alpha \in C (t_0, \infty) \, : \, \| \alpha \|_\flat \quad {\mbox {is bounded}} \},
\end{equation}
and
\begin{equation}\label{alphabound}
\| \alpha \|_\flat := \sup_{t >t_0 } \mu_0^{-{3\over 2}} (t) \, t \, \left[ \|\alpha \|_{\infty , [t, t+1]} + |\alpha |_{0, \sigma , [t, t+1]} \right].
\end{equation}
Here $\sigma$ is the number introduced in \eqref{lambdabound}.
Let us denote by   $h_0 : (0, \infty ) \to (0, \infty)$
 a smooth function with the properties that
\begin{equation}\label{defh}
h_0 (s ) =\left\{ \begin{matrix}
{1\over s} & \quad {\mbox {for}} \quad s\to 0 \\
{1\over s^3} & \quad {\mbox {for}} \quad s\to \infty,
\end{matrix}\right.
\end{equation}
and define the following norm for any function $f: \R^3 \times (t_0 , \infty) \to \R$
\begin{eqnarray}\label{norm1}
\| f \|_{*} :=  \sup_{x \in \R^3, t>t_0}  \mu_0^{-{1\over 2}} \, t^{3\over 2}  \, h_0^{-1}  ({r\over \sqrt{t}})  && \Biggl[  \| f \|_{\infty , B(x,1)\times [t, t+1]} \nonumber \\
&+&  \left[ f \right]_{0, \sigma , B(x,1)\times [t, t+1]} \Biggl], \quad r=|x|.
\end{eqnarray}
Here $\sigma$ is defined in \eqref{lambdabound},
\begin{equation}\label{anita1}
\| f \|_{\infty , B(x,1)\times [t, t+1]} = \sup_{y \in B(x,1), \quad  s\in [t,t+1]}  |f(y,s)|
\end{equation}
and
\begin{equation}\label{anita2}
[f]_{0, \sigma , B(x,1)\times [t, t+1]} = \sup_{ y_1\not= y_2 \in B(x,1) , \quad s_1 \not= s_2 \in [t,t+1]}
{|f(y_1 , s_1 ) - f(y_2 , s_2) | \over |y_1 - y_2|^{2\sigma} + |s_1 - s_2|^\sigma}.
\end{equation}

\medskip
We have the validity of the following estimates, whose proof is quite technical and delayed to Section \ref{appeB}.

\begin{lemma}\label{des1}
Assume $\lambda = \lambda (t) $ satisfies \eqref{lambdabound}.
The error function defined in \eqref{error1} can be described as follows
\begin{equation}\label{error1def}
{\mathcal E}_1 (r,t) =
{\alpha (t) \over \mu+ r} \eta ({r\over r_0 \sqrt{t} } ) + {\mathcal E}_{1,*} [\lambda ] (r,t)  ,
\end{equation}
where $\eta$ is the smooth cut off function defined in \eqref{defeta}, $\alpha $ is the function defined in \eqref{defaaa}, and $r_0$ is a given fixed small number. The function ${\mathcal E}_{1,*} [\lambda] (r,t) $ depends smoothly on $\lambda$. Furthermore, there exists $C>0$ such that
\begin{equation}\label{marte0}
 \|{\mathcal E}_{1,*} \|_{*} \leq C.
\end{equation}
If the initial time $t_0$ in Problem \eqref{p} is large enough, there exist ${\bf c} \in (0,1)$ so that,
 for any $\lambda_1$, $\lambda_2$ satisfying \eqref{lambdabound}, we have
\begin{equation}\label{marte0001}
\| {\mathcal E}_{1,*} [ \lambda_1 ]  - {\mathcal E}_{1,*} [ \lambda_2 ]  \|_{\infty , B(x,1)\times [t, t+1]}   \leq {\bf c} \mu_0^{1\over 2} t^{-{3\over 2}}  h_0 ({r\over \sqrt{t}}) \, \| \lambda_1 - \lambda_2 \|_\sharp
\end{equation}
and
\begin{equation}\label{marte0002}
  \left[  {\mathcal E}_{1,*} [ \lambda_1 ]  - {\mathcal E}_{1,*} [ \lambda_2 ]  \right]_{0, \sigma , B(x,1)\times [t, t+1]}  \leq {\bf c} \mu_0^{1\over 2} t^{-{3\over 2}} \,  h_0 ({r\over \sqrt{t}}) \, \| \lambda_1 - \lambda_2 \|_\sharp ,
\end{equation}
for any $r=|x|$ and any $t$.
The definition of the function $h_0$ and of the norm $\| \cdot \|_*$ are given respectively in \eqref{defh} and in \eqref{norm1}.  Furthermore the constant ${\bf c}$ in \eqref{marte0001} and \eqref{marte0002} can be made as small as one needs, provided that the initial time $t_0$ is chosen large enough.
\end{lemma}

\subsection{Construction of the second global approximation and estimate of the new error.}\label{sc4}

\medskip
\noindent
\medskip
Taking into account the expression of the error function given in \eqref{error1def}, we
 introduce a correction function $\phi_0$ to partially get rid of the term ${\alpha (t) \over \mu + r}$. More precisely, let
 \begin{equation}\label{alphabar}
 \bar \alpha (t) = \left\{ \begin{matrix}
\alpha (t_0)  & \quad {\mbox {for}} \quad t< t_0 \\
\alpha (t)  & \quad {\mbox {for}} \quad t \geq t_0
\end{matrix}\right. ,
 \end{equation}
and introduce the function $\phi_0$ solution to
\begin{equation}
\label{defphi0}
\partial_t \phi_0 = \Delta \phi_0 + {\bar \alpha (t) \over \mu + r} \, {\bf 1}_{\{ r<  M\}}, \quad {\mbox {in}} \quad \R^3 \times (0, \infty), \quad
\phi_0 (x, t_0 -1) = 0, \quad {\mbox {in}} \quad \R^3 , \quad M^2= t_0.
\end{equation}
Here, for a set $K$, we mean
$$
{\mbox 1}_K (x) = 1, \quad {\mbox {if}} \quad x\in K, \quad =0 , \quad {\mbox {if}} \quad x \not\in K.
$$
Duhamel's formula provides an explicit expression for  $\phi_0$
\begin{equation}\label{duha}
\phi_0 (x,t) =\int_{t_0 -1}^t {1\over (4\pi (t-s) )^{3\over 2} } \, \int_{\R^3} e^{-{|x-y|^2 \over 4 (t-s) }} \, {\bar \alpha (s) \over \mu + |y|} \, {\bf 1}_{\{ r< M\}} \,  dy \, ds.
\end{equation}
Since $\lambda$ satisfies \eqref{lambdabound}, classical parabolic estimates give that $\phi_0 $ is locally  $C^{2+2\sigma , 1+\sigma}$, where $\sigma$ is the H\"{o}lder exponent in \eqref{lambdabound}.  In the interval $(t_0 , \infty )$, the function $\phi_0$ solves
\begin{equation}
\label{defphi01}
\partial_t \phi_0 = \Delta \phi_0 + {\alpha (t) \over \mu + r} \, {\bf 1}_{\{ r<  M\}}, \quad {\mbox {in}} \quad \R^3 \times (t_0, \infty),
\end{equation}
and at time $t=t_0$, the function $\phi_0 (x,t_0)$ is radial in $x$ and decays fast as $|x| \to \infty$, that is
\begin{equation}\label{defphi02}
| \phi_0 (x,t_0 ) | \leq c e^{-a |x|^2} , \quad {\mbox {as}} \quad |x| \to \infty
\end{equation}
for some positive, fixed constants $a$ and $c$. Indeed,
let $x=\ell \, e$, with $\| e \| = 1$, and assume that $\ell >\max \{ 1, 2M \}$. Thus $|x-y|^2 > {\ell^2 \over 4}$, for any $|y|<M$, and
$$
|\phi_0 (x , t_0)| \leq C |\alpha (t_0 ) | \left( \int_{t_0 -1}^{t_0} {e^{-{\ell^2 \over 16 (t_0 -s)}} \over (t_0 -s )^{3\over 2}} \, ds
\right) \, \left( \int_{|y|<M} {dy \over |y|} \right) \leq C  |\alpha (t_0 ) |  M^2  e^{-{\ell^2 \over 16}}.
$$
Taking $\ell \to \infty$, estimate \eqref{defphi02} thus follows from \eqref{defaaa}.

\medskip
\noindent
The second approximation is given by
\begin{equation}\label{defU2}
U_2 [\lambda ] (r,t) = U_1 (r,t) + \phi_0 (r,t)
\end{equation}
where $U_1$ is in  \eqref{DefU}. Observe that $U_2$ satisfies the decaying conditions \eqref{IC} at the initial time $t_0$ as consequence of \eqref{defphi02}. The new Error Function
$$
{\mathcal E}_2 [\lambda] (r,t) = {\mathcal E} [U_2] (r,t)
$$
is thus
\begin{equation}\label{defE2}
{\mathcal E}_2 [\lambda  ] (r,t) = \underbrace{{\mathcal E}_{1,*} + {\alpha (t) \over r} \left( \eta ({r\over r_0 \sqrt{t} } ) - {\bf 1}_{\{ r < 2M\} }\right) }_{:={\mathcal E}_{21} }+ \underbrace{\left( U_1 + \phi_0 \right)^5 - U_1^5}_{{\mathcal E}_{22}}.
\end{equation}
The function ${\mathcal E}_{1,*}$ is defined in \eqref{error1def}.
For later purpose, it is useful to estimate, in the $\| \cdot \|_*$-norm introduced in \eqref{norm1}, the function
\begin{equation}\label{cut1}
\bar {\mathcal E}_{2} := {\mathcal E}_{21} + (1- \eta_R (x,t) ) {\mathcal E}_{22}
\quad {\mbox {where}} \quad
\eta_{R} (x,t) = \eta \left (\frac {x} {R \mu_{0}}    \right ) .
\end{equation}
Here $\eta(s)$ is given by \eqref{defeta}, while the number $R$ is a large number, whose definition will depend on $t_0$, but it will not dependent on $t$.

\medskip
\medskip
We have the validity of the following lemma, whose proof is given in Section \ref{appeC}.

\begin{lemma}\label{des2}
Assume $\lambda = \lambda (t) $ satisfies \eqref{lambdabound}.
The error function defined in \eqref{defE2}  depends smoothly on $\lambda$ and it satisfies the following estimates: there exists $C>0$
\begin{equation}\label{marte0new}
\| \bar {\mathcal E}_2 \|_{*} \leq C .
\end{equation}
If the initial time $t_0$ is large enough,  there exist  small positive number   ${\bf c} \in (0,1)$ such that, for any $\lambda_1$, $\lambda_2$ satisfying \eqref{lambdabound}, we have
\begin{equation}\label{marte0001new}
\|\bar  {\mathcal E}_2 [ \lambda_1 ]  - \bar {\mathcal E}_2 [ \lambda_2 ]  \|_{\infty , B(x,1)\times [t, t+1]}    \leq {\bf c}  \mu_0^{1\over 2} t^{-{3\over 2}}  h_0 ({r\over \sqrt{t}}) \, \| \lambda_1 - \lambda_2 \|_\sharp , \quad r=|x|,
\end{equation}
and
\begin{equation}\label{marte0002new}
  \left[  \bar {\mathcal E}_2 [ \lambda_1 ] (r,t) - \bar {\mathcal E}_2 [ \lambda_2 ] (r,t) \right]_{0,\sigma , [t,t+1]}  \leq {\bf c} \mu_0^{1\over 2} t^{-{3\over 2}} \,  h_0 ({r\over \sqrt{t}}) \, \| \lambda_1 - \lambda_2 \|_\sharp ,
\end{equation}
for any $x$ and $t>t_0$, provided the initial time $t_0$ in Problem \eqref{p} is chosen large enough.
The definition of the function $h_0$ is given in \eqref{defh}, and the definition of the $\| \cdot \|_*$-norm is given in \eqref{norm1}.
\end{lemma}

\medskip
\begin{remark}\label{rrr1new} From the proof of the result, we also get that the constant ${\bf c}$ in \eqref{marte0001new} and \eqref{marte0002new} can be made as small as one needs, provided that the initial time $t_0$ is chosen large enough.
\end{remark}

\setcounter{equation}{0}
\section{The inner-outer gluing}\label{iogluing}

We recall the reader that our ultimate purpose is to construct a global unbounded solution $u$ to \equ{p}-\eqref{IC}
 of the form
\be\label{solll}
u= U_2 [\lambda ] (r,t) + \tilde \phi, \quad t > t_0
\ee
where $U_2$ is defined in \eqref{defU2}, while $\tilde \phi (x,t )$ is  a smaller perturbation. The rest of the paper is thus devoted to
 find $\tilde \phi (x,t )$. The construction of $\tilde \phi (x,t )$ is done
by means of
a {\em inner-outer gluing} procedure.
This procedure consists in writing
\begin{equation}\label{deftildephi}
\tilde \phi(x,t) =  \psi (x,t) +    \phi^{in} (x,t)
\quad
{\mbox {where}} \quad
 \phi^{in}(x,t) : =   \eta_{R} (x,t) \hat \phi  (x,t)
 \end{equation}
with
\begin{equation}\label{martin1}
\hat \phi (x,t) := \mu_{0}^{-\frac{1} 2} \phi \left (\frac {x} {\mu_{0}} , t     \right ) , \quad \eta_{R} (x,t) = \eta \left (\frac {x} {R \mu_{0}}    \right ) ,
\end{equation}
where $\eta(s)$ is given in \eqref{defeta}.

\medskip
In terms of $\tilde \phi$, Problem \equ{p}-\eqref{IC} reads as
\be\label{equ2prima}
\partial_t \tilde \phi =  \Delta \tilde \phi + 5 U_2^4\tilde \phi +  N(\tilde \phi ) + {\mathcal E}_2 \inn \R^3 \times [t_0,\infty) ,
\ee
where ${\mathcal E}_2$ is defined in \eqref{defE2} and
$$
 N(\tilde \phi )  =  (U_2 + \tilde \phi)^5 - U_2^5 - 5\, U_2^4 \, \tilde\phi.
$$
Recalling that $w_\mu = \mu^{-{1\over 2}} w ({r\over \mu})$, we let
\begin{equation}\label{defVmu}
V[\lambda ] (r,t) =
5 \left( U_2^4 - w_\mu^4 \right) \eta_R + 5 U_2^4 \left(1-\eta_R \right)
\end{equation}
and write
$
5 U_2^4  = 5 w_\mu^4  \eta_R + V[\lambda ] (r,t).
$
A main observation we make is that $\tilde \phi $ solves Problem \equ{equ2prima} if the tuple $(\psi, \phi)$  solves the following coupled system of nonlinear equations
\begin{align} \label{equpsi}
\partial_t \psi &=  \Delta \psi +  V [\la] \psi  +
  [2\nabla \eta_{R} \nabla_x \hat \phi +  \hat \phi (\Delta_x-\partial_t)\eta_{R} ]\nonumber \\
&+   N [\lambda ]( \tilde \phi  ) +  {\mathcal E}_{21}  + {\mathcal E}_{22} (1-\eta_R ) \inn \R^3 \times [t_0,\infty),
\end{align}
and
\be\label{equ2}
\partial_t \hat \phi =  \Delta \hat \phi+ 5  w_\mu^4 \hat \phi + 5 w_\mu^4 \psi +  {\mathcal E}_{22} \inn  B_{2R\mu_{0}} (0)  \times [t_0,\infty).
\ee
We refer to \eqref{defE2} for the definition of ${\mathcal E}_{21}$ and ${\mathcal E}_{22}$.
In terms of $\phi$, see \eqref{martin1},  equation \eqref{equ2} becomes
\begin{align} \label{equ3}
\mu_{0}^2 \partial_t  \phi  =&  \Delta_y \phi +  5 w^4 \phi
+
 \mu_{0}^{\frac{5}2}  {\mathcal E}_{22}  (\mu_{0} y,t) +  5 \,  { \mu_{0}^{\frac{1}2} \over (1+\Lambda )^4 } \, w^4 ( {y \over (1+ \Lambda )^2}  ) \psi(\mu_{0} y,t)
\\
&+  B[\phi] + B^0 [\phi]  \inn  B_{2R} (0)  \times [t_0,\infty)   \nonumber
\end{align}
where
\begin{equation}\label{gajardo1}
 B[\phi]:= \mu_{0} \, ( \partial_t \mu_{0} ) \, \left(\frac {\phi }2  + y\cdot \nabla_y \phi \right)
 \end{equation}
and
\begin{equation}\label{gajardo2}
B^0 [\phi ] := 5 \left[ w^4 \left( {y \over (1+ \Lambda )^2}  \right) -
w^4 (y) \right] \, \phi +5 \, \left( {1-(1+ \Lambda)^4 \over (1+ \Lambda)^4 }  \right) \, w^4 \left( {y \over (1+ \Lambda )^2}   \right)
\phi.
\end{equation}
We call \equ{equpsi} the {\em outer problem} and  \equ{equ3} the {\em inner problem(s) }.

\medskip
We next describe precisely our strategy to solve \eqref{equpsi}-\eqref{equ3}. For given parameter $\lambda$ satisfying \eqref{lambdabound}, and function $\phi$ fixed in a suitable range, we first solve for $\psi$ the outer Problem \eqref{equpsi}, in the form of a (nonlocal) nonlinear operator $\psi = \Psi ( \lambda , \phi )$. This is done in full details in Section \ref{outerP}.

\medskip
We then replace this $\psi$ in equation \eqref{equ3}.
At this point we consider the  change of variable,
$$
t = t (\tau ) , \quad
{dt \over d\tau } = \mu_{0}^2  (t) ,
$$
that reduces \eqref{equ3} to
\begin{equation} \label{equ3331}
 \partial_\tau  \phi  =  \Delta_y \phi + 5w^4  \phi
+ H [ \psi , \lambda , \phi ] (y,t (\tau ) ) , \quad y \in B_{2R} (0), \quad \tau \geq \tau_0
\end{equation}
where $\tau_0 $ is such that $t (\tau_0 ) = t_0$, and
\begin{align}\label{defHH}
 H [\psi , \lambda , \phi ](y,t (\tau ) ) &=  \mu_{0}^{\frac{5}2}  {\mathcal E}_{22} (\mu_{0} y,t) +  5 \,  { \mu_{0}^{\frac{1}2} \over (1+ \Lambda)^4 } \, w^4 ( {y \over (1+ \Lambda)^2 }  ) \psi(\mu_{0} y,t) \nonumber \\
&
+  B[\phi] + B^0 [\phi]
\end{align}
Next step is to construct a solution  $\phi$ to Problem \eqref{equ3331}. We can do this for functions $\phi$ which furthermore satisfy
\begin{equation} \label{equ3332}
 \phi(y, \tau_0 ) = e_{0} Z(y) , \quad y \in B_{2R} (0),
\end{equation}
for some constant $e_{0}$. Here $Z$ is the positive radially symmetric bounded eigenfunction
associated to the only negative eigenvalue $\la_0$ to the problem
\begin{equation}
\label{eigen0}
L_0 (\phi ) + \lambda \phi = 0 , \quad  \phi \in L^\infty(\R^3).
\end{equation}
Here $L_0$ is the linear operator around the standard bubble $w$ in $\R^3$. We refer to \eqref{defL0} for the definition of $L_0$. Furthermore, it is known that $\la_0 $ is simple and  $Z$
 decays like
$$Z (y) \sim  |y|^{-1} e^{-\sqrt{|\la_0 |}\,  |y|}  \quad {\mbox {as}} \quad |y| \to \infty. $$

\medskip
To be more precise, we prove that Problem \eqref{equ3331}-\eqref{equ3332} is solvable in $\phi$, provided that in addition
 the parameter  $\lambda$ is chosen so that   $H [ \psi,  \lambda , \phi ] (y,t (\tau ) ) $ satisfies the orthogonality condition
\begin{align}\label{gajardo3}
\int_{B_{2R} } H [ \psi , \lambda , \phi ](y,t (\tau ) )  Z_0 (y)  dy &= 0, \quad {\mbox {for all}} \quad  t>t_0.
\end{align}
We recall that $Z_0 (y) $, defined in \eqref{Z0def},  is the only bounded radial element
in the kernel of the linear elliptic operator $L_0$.

\medskip
Equation \eqref{gajardo3} becomes a non-linear, non-local problem in $\la$, for any fixed $\phi$. We attack this problem in Sections \ref{secpar},
\ref{nonlocal}, \ref{secpar1}. In Section \ref{secpar}, we get the precise form of Equation \eqref{gajardo3} as a non local non linear operator in $\la$. The principal part of the operator in $\la$ defined by Equation \eqref{gajardo3} is a linear non-local operator which turns out to be a perturbation of the ${1\over 2}$-Caputo derivative. We refer to \cite{caputo} for the original definition  of Caputo derivatives. In Section \ref{nonlocal} we develop an invertibility theory for such linear operator. In Section \ref{secpar1} we fully solve Equation \eqref{gajardo3} in $\la$, by means of a Banach fixed point argument. The solution $\la = \la [\phi]$ is a non linear operator in $\phi$, and we also
 describe the Lipschitz dependence of $\la $ with respect to $\phi$, which is a key property for our final argument.

At this point, one realizes that a central point of our complete proof is to design a linear theory that allows us to solve in $\phi$ Problem \eqref{equ3331}-\eqref{equ3332}.
To this purpose,
we shall construct a solution to an initial value problem of the form
\be \label{p110nuovo}
\phi_\tau  =
\Delta \phi +5w^4  \phi + h(y,\tau )  \inn B_{2R} \times (\tau_0, \infty ) , \quad \phi(y,\tau_0) = e_0Z (y)  \inn B_{2R}.
\ee
And then we solve Problem \eqref{equ3331}-\eqref{equ3332} by means of a contraction mapping argument.

\medskip
Let $a$ be a fixed number with $a\in (0,2)$, and let $\nu >0$ so that, for $t$ large,
$$ \tau^{-\nu } \sim \mu_0^{3\over 2} t^{-1} , \quad {\mbox {if}}  \quad \gamma \not= 2, \quad {\mbox {and}}  \quad \tau^{-\nu } \sim \mu_0^{3\over 2} t^{-1+\nu'} , \quad
{\mbox {if}} \quad \gamma =2,
$$ for some $\nu' >0$ that can be fixed arbitrarily small.
We solve \eqref{p110nuovo} for functions $h$  with $\| h \|_{ \nu , 2+a}$-norm  bounded, where
 \be \label{minchia}
\|h\|_{ \nu, 2+a} := \sup_{ \tau >\tau_0 , y \in \R^3 }    \tau^{\nu} (1+ |y|^{2+a} ) \,\Biggl[  \|h\|_{\infty ,B(y,1) \times [\tau,\tau +1]}
+[ h]_{0,\sigma , B(y,1) \times [\tau,\tau +1]} \Biggl] ,
\ee
and we construct solutions $\phi$  in the class of functions with $\| \phi \|_{\nu, a}$-norm bounded, where
\begin{eqnarray}\label{normphi}
\| \phi \|_{\nu , a} &:= \sup_{ \tau>\tau_0 , y \in \R^3 }  \tau^\tau (1+ |y|^a )  \left[ \,  \|\phi \|_{\infty ,B(y,1) \times [\tau,\tau +1]} +   [ \phi]_{0,\sigma , B(y,1) \times [\tau,\tau +1]}\right] \nonumber \\
&+ \sup_{ \tau >\tau_0 , y \in \R^3 }  \tau^\nu (1+ |y|^{1+a} )  \left[ \,  \|\nabla \phi \|_{\infty ,B(y,1) \times [\tau,\tau +1]} +   [ \nabla \phi]_{0,\sigma , B(y,1) \times [\tau,\tau +1]}\right]
\end{eqnarray}

\medskip
We have the validity of the following result

\begin{proposition} \label{prop0}
Let $\nu,a$ be given positive numbers with $0<a <2$. Then, for all sufficiently large $R>0$ and  function $h=h(y,\tau)$,   with $h (y, \tau ) = h(|y| , \tau )$ and  $\|h\|_{\nu, 2+a} <+\infty$
that satisfies
\be
 \int_{B_{2R}} h(y ,\tau)\, Z_0 (y) \, dy\ =\ 0  \foral \tau\in (\tau_0, \infty)
\label{ortio}\ee
there exist  $\phi \in C^{2+2\sigma , 1+\sigma}$-loc., which is radial in $y$, and $e_0 $ which solve Problem $\equ{p110nuovo}$. Moreover, $\phi= \phi[h]$, and $e_0 = e_0 [h]$  define linear operators of $h$
that satisfy the estimates
\be
  |\phi(y,\tau) |  \ \leq C  \  \tau^{-\nu} \, \frac {R^{4-a}} { 1+ |y|^3} \,   \|h \|_{\nu, 2+a}     , \quad
  |\nabla_y \phi(y,\tau) |  \ \leq C   \  \tau^{-\nu} \, \frac {R^{4-a}} { 1+ |y|^4} \,   \|h\|_{\nu, 2+a}     ,
\label{cta1g}\ee
and
$$
 | e_0[h]| \, \leq C \,  \|h\|_{ \nu, 2+a},
$$
for some fixed constant $C$.
\end{proposition}

\medskip
We postpone the proof of this Proposition to Section \ref{inner0}.  Section \ref{final} is devoted to solve Problem \eqref{equ3331}-\eqref{equ3332} and this concludes the proof of Theorem \ref{teo1}.

\setcounter{equation}{0}
\section{Solving the outer problem}\label{outerP}

The aim of this section is to solve the {\it outer problem} \eqref{equpsi} for given parameter $\lambda  $ satisfying \eqref{lambdabound}, and for given small functions $\phi$,
in the form of a nonlinear nonlocal operator
$$
\psi (x,t) = \Psi [ \lambda, \phi] (x,t).
$$
We recall that $\phi^{in} (x,t)=    \eta_{R} (x,t) \hat \phi  (x,t)
 $
with
$$
\hat \phi (x,t) := \mu_{0}^{-\frac{1} 2} \phi \left (\frac {x} {\mu_{0}} , t     \right ), \quad {\mbox {and}} \quad
\eta_{R} (x,t) = \eta \left (\frac {x} {R \mu_{0}}    \right ) .
$$
Here $\eta(s)$ is defined in \eqref{defeta}, and number $R$ is a sufficiently large number, independent of $t$.
We assume that
\begin{equation}\label{phibound}
\| \phi \|_{\nu , a} \quad {\mbox {is bounded}}.
\end{equation}

\medskip
\noindent
Let $\varphi_0 : (0,\infty ) \to (0, \infty) $ be a  smooth and bounded  given function with the property that
\begin{equation}\label{defvarphi0}
 \varphi_0 (s) = \left\{ \begin{matrix}
s & \quad {\mbox {for}} \quad s\to 0^+ \\
{1\over s^3} & \quad {\mbox {for}} \quad s\to \infty
\end{matrix}\right. .
\end{equation}
We introduce the following $L^\infty$-weighted norms for functions $f = f(r,t)$
\begin{eqnarray}\label{norm88}
\| f \|_{** } :=  \| f \|_1 + \| D f \|_2
\end{eqnarray}
\begin{eqnarray}\label{norm2}
\| f \|_{1} :=  \sup_{x \in \R^3, t>t_0}  \mu_0^{-{1\over 2}} \, t^{1\over 2}  \, \varphi_0^{-1}  ({r\over \sqrt{t}})  && \Biggl[  \| f \|_{\infty , B(x,1)\times [t, t+1]} \nonumber \\
&+&  \left[ f \right]_{0, \sigma , B(x,1)\times [t, t+1]} \Biggl], \quad r=|x|.
\end{eqnarray}
\begin{eqnarray}\label{norm3}
\| f \|_{2} :=  \sup_{x \in \R^3, t>t_0}  \mu_0^{-{1\over 2}} \, t  \, (\varphi_0' )^{-1}  ({r\over \sqrt{t}})  && \Biggl[  \| f \|_{\infty , B(x,1)\times [t, t+1]} \nonumber \\
&+&  \left[ f \right]_{0, \sigma , B(x,1)\times [t, t+1]} \Biggl], \quad r=|x|.
\end{eqnarray}
Refer to \eqref{anita1} and \eqref{anita2} for the definitions of $ \| f \|_{\infty , B(x,1)\times [t, t+1]}$ and $\left[ f \right]_{0, \sigma , B(x,1)\times [t, t+1]}$.

\begin{proposition}\label{propext} Assume that $ \lambda $ satisfies \eqref{lambdabound}, and that the function $\phi$ satisfies the bound \eqref{phibound}. Let $ \psi_0 \in C^2 (\R^3 )$, radially symmetric so that
 \begin{equation}\label{gg1}
 |y|\,| \psi_0 (y) | + |y| \, |\nabla  \psi_0 (y) | \leq t_0^{-a} e^{-b |y|} , \quad
 \end{equation}
 for some positive constants $a$ and $b$. There exists $t_0$ large so that Problem \eqref{equpsi} has
 a unique solution $\psi = \Psi [ \lambda , \phi ]$
 so that
\begin{equation}\label{marte2}
\psi (r, t_0 ) = \psi_0 (r), \quad
\| \psi \|_1 + \| D \psi \|_2 \leq C.
\end{equation}
\end{proposition}

\begin{proof}
Let $f$ be a given function with  $\| f \|_{*}$-norm bounded. Classical parabolic estimates give that
any solution to $\partial_t \psi = \Delta \psi + f$ is locally $C^{2+2\sigma , 1+\sigma}$.
Furthermore, consequence  of Lemma \ref{lemacalor} is that  the function
$\bar \varphi_0 (r,t) = \mu_0^{1\over 2} t^{-{1\over 2}} \varphi_0 ({r\over \sqrt{t}})$ is a positive supersolution for
$
\partial_t \psi \geq \Delta \psi + f(r,t).
$ Observe also that $\bar \varphi_0 (r,t_0 ) \geq \psi_0 (r)$.
Combining these facts with the maximum principle, we see that, for a function $f$ with $\| f \|_{*}$-norm bounded, the unique  solution to $\partial_t \psi = \Delta \psi + f$,  with $\psi( r,t_0) =  \psi_0$, has
$\| \psi \|_{**}$-norm bounded.
We claim that a possibly large multiple of $\bar \varphi_0$ works as a supersolution also for Problem
\begin{equation}\label{marte8}
\partial_t \psi \geq \Delta \psi + V (r,t) \psi + f(r,t).
\end{equation}
Indeed, recalling the definition of $V$ in \eqref{defVmu}, we write
$$
V= V_1 + V_2, \quad V_1 =
5 \left( U_2^4 - w_\mu^4 \right) \eta_R , \quad V_2 =5 U_2^4 \left(1-\eta_R \right).
$$
In the region where $\eta_R \not= 0$, namely when $r < 2R \mu_0$, we expand in Taylor the function $V_1$ and we find $s^* \in (0,1)$ so that
$$
V_1 (r,t) = 20 \left(w_\mu + s^* (\mu_0' \Psi_1 (r, t) + \phi_0 (r,t) ) \right)^3 [  \mu_0' \Psi_1 (r , t) + \phi_0 (r,t) ) ] \eta_R.
$$
From here, we see that, in this region,
$
|V_1 (r,t) | \lesssim  R t^{-1}  \, \eta_R,
$
so that
\begin{equation}\label{marte6}
|V_1 (r,t)  \psi_0 (r,t) |  \lesssim  \mu_0^{1\over 2} t^{-{3\over 2}}\,  h_0 ({r\over \sqrt{t}} ) .
\end{equation}
Let us now consider $V_2$. This function is not zero only when $r > R \mu_0 $, and in this region we have that
$
| V_2 (r,t) | \lesssim {\mu_0^2 \over r^4} \, (1-\eta_R ),
$
so that
\begin{equation}\label{martes7}
|V_2 (r,t) \psi_0 (r,t)| \lesssim {\mu^2 \over r^4} \mu_0^{1\over 2} t^{-{1\over 2} } \varphi_0 \left( {r\over \sqrt{t}} \right) \, (1-\eta_R)
 \lesssim R^{-2} \mu_0^{1\over 2} t^{-{3\over 2} } h_0 ({r\over \sqrt{t}}) .
\end{equation}
Choosing $R$ large, but independent of $t$, we thus find that a multiple of $\bar \varphi_0$ is a supersolution for \eqref{marte8}.

\medskip
We call  $T_o: (f , \psi_0 ) \to \psi$ the linear operator that to any
$f$ with $\| f \|_{*}$-norm bounded  and any initial condition $\psi_0$ satisfying \eqref{gg1} associates the unique solution  to
\begin{equation}\label{vai}
\partial_t \psi = \Delta \psi + V[\lambda ] (r,t) \psi + f(r,t) , \quad \psi (r,t_0) =\psi_0 (r) ,
\end{equation}
which has bounded
 $\| \psi \|_{** }$-norm. Define $\bar \psi = T_o (0, \psi_0)$.
We observe that  $\psi + \bar \psi $ is a solution to \eqref{equpsi} if $\psi $ is a fixed point for the operator
\begin{equation}
\label{marte3}
{\mathcal A}_o (\psi ) = T_o \left(  [2\nabla \eta_{R} \nabla_x \hat \phi +  \hat \phi (\Delta_x-\partial_t)\eta_{R} ]
+   N [\lambda ]( \tilde \phi +\bar \psi  ) +  {\mathcal E}_{21}  + {\mathcal E}_{22} (1-\eta_R ) \right)
\end{equation}
We shall show the existence and uniqueness of such fixed point as consequence of the Contraction Mapping Theorem.
We perform a fixed point argument in  the set of functions $\psi$ in
\begin{equation}\label{marte4}
B_o = \{ \psi \in L^\infty \, : \, \| \psi \|_{** } < r \}
\end{equation}
for some  $r>0$.

\medskip
\noindent
From Lemma \ref{des1} we have that
 there exists a constant $c_1$ so that
\begin{equation}\label{marte00bu}
\| {\mathcal E}_{21} + {\mathcal E}_{22} (1-\eta_R )  \|_{*} \leq c_1.
\end{equation}
We now claim that there exists constant $c_2$ such that, if the parameter $\lambda $ satisfies \eqref{lambdabound}, and if the function $\phi$ satisfies the bound \eqref{phibound}, then

\begin{equation}
\label{marte02}
\left\| 2\nabla \eta_{R} \nabla_x \hat \phi +  \hat \phi (\Delta_x-\partial_t)\eta_{R}
\right\|_{*} + \left\|   N ( \tilde \phi +\bar \psi ) \right\|_{*} \leq c_2
\end{equation}

\medskip
Furthermore, we claim that there exists a constant ${\bf c}  \in (0,1)$ so that, for any $\psi_1$, $\psi_2 \in B_0$,
\begin{equation}\label{marte03}
\left\| {\mathcal A}_o (\psi_1 ) - {\mathcal A}_o (\psi_2 ) \right\|_{** } \leq {\bf c} \, \| \psi_1 - \psi_2 \|_{**}.
\end{equation}

\medskip
If we assume, for the moment, the validity of \eqref{marte00bu}, \eqref{marte02} and \eqref{marte03}, we get the existence of a fixed point for problem \eqref{marte3} in the set \eqref{marte4}, provided $r$ is chosen large enough.

\medskip
\noindent

\medskip
\noindent
Proof of \eqref{marte02}. \ \ We start with the estimate of the first term in \eqref{marte02}.
Since we assume the validity of the bound \eqref{phibound} on $\phi$, we write
$$
\left| \hat \phi \Delta_x \eta_R \right| \lesssim {|\eta'' ({|x| \over R \mu_0 }) |\over R^2 \mu_0^2} \, |\hat \phi |
\lesssim {|\eta'' ({|x| \over R \mu_0 })| \over R^2 \mu_0^2} {\mu_0^{3\over 2} t^{-1} \over (1+ |{x\over \mu_0}|^a)} \, \| \phi \|_{\nu , a}
$$
see \eqref{normphi} for the notation $\| \phi \|_{\nu , a}$. Thus, we get
\begin{equation*}
\left| \hat \phi \Delta_x \eta_R \right| \lesssim   {|\eta'' ({|x| \over R \mu_0 }) |\over R^{2+a}  } \mu_0^{-{1\over 2} } t^{-1} {r \over \sqrt{t}} h_0 ({r\over \sqrt{t}})  \, \| \phi \|_{\nu , a} \lesssim   {|\eta'' ({|x| \over R \mu_0 }) |\over R^{1+a}  } \mu_0^{1\over 2 } t^{-{3\over 2}}  h_0 ({r\over \sqrt{t}})  \, \| \phi \|_{\nu , a}
\end{equation*}
$$
\lesssim  \mu_0^{1\over 2 } t^{-{3\over 2}} h_0 ({r\over \sqrt{t}} )  {\| \phi \|_{\nu , a} \over R^{1+a} } .
$$
Arguing similarly, we get
\begin{equation*}\label{gg2}
\left| \hat \phi \partial_x \eta_R \right| \lesssim   \mu_0^{1\over 2 } t^{-{3\over 2}} h_0 ({r\over \sqrt{t}} )  {\| \phi \|_{\nu , a} \over R^{1+a} } , \quad {\mbox {and}} \quad
\left| \nabla \hat \phi \nabla  \eta_R \right| \lesssim   \mu_0^{1\over 2 } t^{-{3\over 2}} h_0 ({r\over \sqrt{t}} )  {\| \phi \|_{\nu , a} \over R^{1+a} },
\end{equation*}
which proves the $L^\infty$ bound in the first estimate in \eqref{marte02}. To check the H\"{o}lder bound for this term, we focus the analysis on the term $g(x,t):= \hat \phi \Delta_x \eta_R $. The others terms can be treated in a similar way.
We write
\begin{align*}
{|g(x_1 , t_1 ) - g(x_2 , t_2) | \over |x_1 - x_2|^{2\sigma} + |t_1 - t_2|^\sigma}& =
| \Delta_x \eta_R (x_1 , t_1 ) | {|\hat \phi (x_1 , t_1 ) - \hat \phi (x_2 , t_2) | \over |x_1 - x_2|^{2\sigma} + |t_1 - t_2|^\sigma}\\
&+| \hat \phi (x_2 , t_2) | {| \Delta_x \eta_R (x_1 , t_1 ) - \Delta_x \eta_R (x_2 , t_2) | \over |x_1 - x_2|^{2\sigma} + |t_1 - t_2|^\sigma}
\end{align*}
In order to control the first term, we use the definition in \eqref{normphi} of $\| \phi \|_{\nu , a}$ and we argue as before. The second term can be easily treated using the $L^\infty$-bound on $\hat \phi$ and the smoothness of the function $\Delta_x \eta_R$. This complete the analysis of the first estimate in \eqref{marte02}.

We continue with the proof of the second estimate in \eqref{marte02}. We recall that $
 N(\tilde \phi )  =  (U_2 + \tilde \phi)^5 - U_2^5 - 5\, U_2^4 \, \tilde\phi.
$ It is convenient to estimate this function in three different regions: where $r < {\bar M}^{-1} \mu_0$, where ${\bar M}^{-1} \mu_0 < r < \bar M \sqrt{t}$ and where
$r > \bar M \sqrt{t}$, with $\bar M$ a large positive number.

From the definition of $U_2$ in \eqref{defU2}, we see that, if $r < {\bar M}^{-1} \mu_0$, then
$$
|N(\tilde \phi ) | \lesssim \mu_0^{-{3\over 2}} |\tilde \phi|^2 \lesssim \mu_0^{-{3\over 2}}  \left[ |\psi |^2 + |\eta_R \hat \phi |^2\right].
$$
We recall that
\begin{equation}\label{nonso0}
\left| \psi \right| \lesssim  \| \psi \|_{**} \, \mu_0^{1\over 2} t^{-{1\over 2}} \varphi_0 ({r\over \sqrt{t}} )
, \quad
\left| \eta_R \hat \phi \right| \lesssim \mu_0^{3\over 2} t^{-1} |\eta_R |  \, \| \phi \|_{\mu , a}
\end{equation}
so that we get, for $r < {\bar M}^{-1} \mu_0$,
\begin{equation}\label{nonso1}
|N(\tilde \phi +\bar \psi ) | \lesssim \mu_0^2 t^{-1} \left[  \| \psi + \bar \psi \|_{**}^2  +  \| \phi \|_{\mu , a}^2 \right] \, \left( \mu_0^{1\over 2} t^{-{3\over 2}} h_0 ({r\over \sqrt{t}} ) \right)
\end{equation}
Let us now consider the region ${\bar M}^{-1} \mu_0 < r < \bar M \sqrt{t}$. Here, after a Taylor expansion, we get that
$$
\left| N(\tilde \phi +\bar \psi ) \right| \lesssim w_\mu^3 \left[ |\psi +\bar \psi |^2 + |\eta_R \hat \phi |^2\right]\lesssim {\mu_0^{3\over 2} \over r^3}
\left[ |\psi |^2 + |\eta_R \hat \phi |^2\right].
$$
Using again \eqref{nonso0}, we obtain, for ${\bar M}^{-1} \mu_0 < r < \bar M \sqrt{t}$,
\begin{equation}\label{nonso2}
|N(\tilde \phi + \bar \psi ) | \lesssim \mu_0^2 t^{-1} \left[  \| \psi + \bar \psi \|_{**}^2  +  \| \phi \|_{\mu , a}^2 \right] \, \left( \mu_0^{1\over 2} t^{-{3\over 2}} h_0 ({r\over \sqrt{t}} ) \right).
\end{equation}
Let us now consider $r > \bar M \sqrt{t}$. Observe that in this region $\eta_R = 0$, $|(\psi + \bar \psi )(r,t) | \lesssim \mu_0^{1\over 2} t^{-{1\over 2}} \varphi_0 ({r \over \sqrt{t} })$
and, from \eqref{maso1}, also $|U_2 (r,t) | \lesssim {\mu_0 \over r} $. Thus we have
\begin{equation}\label{nonso3}
\left| N(\tilde \phi +\bar \psi ) \right| \lesssim \left({\mu_0 \over r} \right)^5 \lesssim \mu_0^{9 \over 2} t^{-{1\over 2}} \, \left( \mu_0^{1\over 2} t^{-{3\over 2}} h_0 ({r\over \sqrt{t}} ) \right).
\end{equation}
From \eqref{nonso1}, \eqref{nonso2}, \eqref{nonso3}, we get the $L^\infty$ bound for the second estimate in \eqref{marte02}.

\medskip
\noindent
Proof of \eqref{marte03}. \ \ For any $\psi_1$, $\psi_2 \in B_o$, we have that
$$
{\mathcal A}_o (\psi_1 ) - {\mathcal A}_o (\psi_2 ) = T_0 \left( N(\psi_1 +  \bar \psi +\phi^{in} ) - N (\psi_2 +\bar \psi + \phi^{in} ) \right)
$$
thus
$$
\| {\mathcal A}_o (\psi_1 ) - {\mathcal A}_o (\psi_2 )  \|_{** } \leq C \|N(\psi_1 + \bar \psi + \phi^{in} ) - N (\psi_2 + \bar \psi +\phi^{in} )\|_{* }.
$$
We write
\begin{align*}
N(\psi_1 + \phi^{in} ) - N (\psi_2 + \phi^{in} ) &= (U_2 + \psi_1 +  g )^5 - (U_2 + \psi_2 + g )^5 - 5 U_2^4 (\psi_1 - \psi_2) \\
&= \underbrace{(U_2 + \psi_1 + g )^5 - (U_2 + \psi_2 + g )^5 - 5 (U_2 + g )^4 (\psi_1 - \psi_2)}_{:=N_1} \\
&  + \underbrace{ 5 [(U_2 + g )^4  -   U_2^4]  (\psi_1 - \psi_2) }_{:=N_2} , \quad g := \phi^{in} + \bar \psi
\end{align*}
In the region where $r< \bar M \sqrt{t}$, we have that
$$
|N_1 (x,t) | \lesssim w_\mu^3 |\psi_1 - \psi_2|^2
$$
which yields to
$$
\left| N_1 (x,t) \right| \lesssim \mu_0^2 t^{-1} \left[ \| \psi_1 - \psi_2 \|_{**}^2 \right] \, \left( \mu_0^{1\over 2} t^{-{3\over 2}} h_0 ({r\over \sqrt{t}} ) \right).
$$
while $N_2$ can be estimated as
$$
\left| N_2 (x,t) \right| \lesssim \left[ \mu_0^2 t^{-1} \| \bar \psi \|_{**} + \mu_0^2 t^{-1} \, \, \| \phi^{in} \|_{\nu, a} \right] \| \psi_1 - \psi_2 \|_{** }  \left( \mu_0^{1\over 2} t^{-{3\over 2}} h_0 ({r\over \sqrt{t}} ) \right).
$$
On the other hand, if $r > \bar M \sqrt{t}$, we have that $\phi^{in} \equiv 0$, so that
$$
\left| N_2 (x,t) \right| \lesssim  \mu_0^2 \| \bar \psi \|_{**}  \| \psi_1 - \psi_2 \|_{** }  \left( \mu_0^{1\over 2} t^{-{3\over 2}} h_0 ({r\over \sqrt{t}} ) \right).
$$
On the other hand $N_1$ can be estimates
as follows
$$
\left| N_1 (x,t) \right| \lesssim |\psi_1 - \psi_2 |^5, \quad {\mbox {from which}} \quad \left| N_1 (x,t) \right|  \lesssim \mu_0^2  \, \, \mu_0^{1\over 2} t^{-{1\over 2}} \, h_0 ({r\over \sqrt{t}} ) \| \psi_1 - \psi_2 \|_{**}.
$$
In summary, we get that
$$
\| N(\psi_1 + \phi^{in} +\bar \psi ) - N (\psi_2 + \phi^{in} +\bar \psi ) \|_{*, \beta} \leq C \mu_0^2  \, \| \psi_1 - \psi_2 \|_{**}
$$
where $C = \max \{ \| \psi_1 - \psi_2 \|_{**} , \| \phi^{in} \|_{\nu, a} \}$. Thus we get the validity of \eqref{marte03} provided that $t_0$ is large enough.

\end{proof}

\medskip
\begin{remark}\label{rmk1}
Proposition \ref{propext} defines the solution to Problem \eqref{equpsi} as a function of the initial condition $\psi_0$, in the form of an operator $ \psi = \bar \Psi [\psi_0]$, from a small neighborhood of $0$ in the Banach space $L^\infty (\Omega)$ equipped with the  norm
\begin{equation}\label{gg2}
 \sup_{y \in \R^3} \left[ |y| \,  | e^{b |y|} \psi_0 (y) | +|y| \,  | e^{b |y|} \nabla \psi_0 (y) |\right]
\end{equation}
 into the Banach space of functions $\psi \in L^\infty (\Omega)$ equipped with the norm
$ \| \psi \|_{**}$ , defined in \eqref{norm88}.
A closer look to the proof of Proposition \ref{propext}, and the Implicit Function Theorem give that $\psi_0 \to \bar \Psi [\psi_0]$ is a diffeomorphism, and that
$$
\| \bar \Psi [\psi_0^1 ] - \bar \Psi [\psi_0^2] \|_{**} \leq c \left[
\sup_{y \in \R^3} \left|  |y| \, e^{b |y|} [  \psi_0^1 - \psi_0^2 ] \right| +
\sup_{y \in \R^3} \left|  |y| \, e^{b |y|} [ \nabla \psi_0^1 - \nabla \psi_0^2 ] \right|  \right],
$$
for some positive constant $c$.
\end{remark}

\medskip

\begin{proposition}\label{propext1} Assume the validity of the  assumptions of Proposition \ref{propext}. Then the function $\psi = \Psi ( \lambda, \phi)$ depends smoothly on $ \lambda $ and $\phi$, and we have the validity of the following estimates: for any initial time $t_0$ in Problem \eqref{p} sufficiently large, and any sufficiently large radius $R$ in the cut off function $\eta_R$ introduced  in \eqref{deftildephi} and there exist ${\bf c} $ such that, given  $\lambda_1 $, $\lambda_2$ satisfying \eqref{lambdabound} one has
\begin{equation}\label{voto1}
\|  \Psi [\lambda_1, \phi ] - \Psi [\lambda_2, \phi ] \|_{**} \leq {\bf c} \| \lambda_1 - \lambda_2 \|_\sharp
\end{equation}
 and for any $\phi$ satisfying \eqref{phibound}. Moreover, given $\phi_1$, $\phi_2$ satisfying \eqref{phibound}, one has
\begin{equation}\label{voto2}
\|  \Psi [\lambda, \phi_1 ] - \Psi [\lambda, \phi_2 ] \|_{**} \leq {\bf c} \| \phi_1 - \phi_2 \|_{\nu,a}
\end{equation}
for any $\lambda$ satisfying \eqref{lambdabound}.
\end{proposition}

\begin{proof}
Fix $\phi $ and define
 $\bar \psi = \psi [\lambda_1 , \phi ] - \psi [\lambda_2 , \phi ]$, for $\lambda_1$ and $\lambda_2$ satisfying \eqref{lambdabound}. Then $\bar \psi $ solves
 $$
 \partial_t \bar \psi = \Delta \bar \psi + \left( V[\lambda_1]  + N'[\lambda] \right) (\bar \psi ) + F , \quad \R^3 \times (t_0 , \infty ) , \quad \bar \psi (r, t_0 ) = 0
 $$
for  $\lambda = s \lambda_1 + (1-s) \lambda_2$, $s \in (0,1), $ where
 \begin{align*}
 F&= {\mathcal E}_{21} [\lambda_1] - {\mathcal E}_{21} [\lambda_2] + (1-\eta_R) \left[ {\mathcal E}_{22} [\lambda_1] - {\mathcal E}_{22} [\lambda_2] \right]\\
 &+ \left[ V[\lambda_1] - V [\lambda_2] \right] \psi_2 + \left[ N[\lambda_1] - N[\lambda_1] \right]  (\psi_2+ \phi^{in} )
 \end{align*}
 where $\psi_j= \psi [\lambda_j , \phi ]$, $j=1, 2$. From Lemma \ref{des1}, estimates \eqref{marte0001new}-\eqref{marte0002new}, we get that
 $$
 \| {\mathcal E}_{21} [\lambda_1] - {\mathcal E}_{21} [\lambda_2]  \|_* \leq {\bf c} \| \lambda_1 - \lambda_2 \|_\sharp
 $$
 and
 $$
 \| (1-\eta_R) \left[ {\mathcal E}_{22} [\lambda_1] - {\mathcal E}_{22} [\lambda_2] \right] \|_*
\leq {\bf c} \| \lambda_1 - \lambda_2 \|_\sharp,
 $$
 provided $t_0$ is large enough.
 One also checks that, for some ${\bf c} \in (0,1)$
 $$
 \| \left[ V[\lambda_1] - V [\lambda_2] \right] \psi_2  \|_* \leq {\bf c} \| \lambda_1 - \lambda_2 \|_\sharp, \quad
\|\left[ N[\lambda_1] - N[\lambda_2] \right]  (\psi_2+ \phi^{in} )  \|_* \leq {\bf c} \| \lambda_1 - \lambda_2 \|_\sharp.
 $$
The constant $c_1$ can be made arbitrarily small  provided $t_0$ is large.
 Arguing as in \eqref{marte6} and \eqref{martes7}, one can show that a certain multiple of the function
 $\| \lambda_1 - \lambda_2 \|_\sharp \bar \varphi_0 (r,t)$, where $\bar \varphi_0  = \mu_0^{1\over 2} t^{-{1\over 2}} \varphi_0 ({r\over \sqrt{t}})$, serves as supersolution for $\bar \psi$. This proves \eqref{voto1}.

 \medskip
 Let us now fix $\lambda$, and take $\phi_1$, $\phi_2$ satisfying \eqref{phibound}. Denote by $\phi_j^{in} = \eta_R \hat \phi_j$,
 and $\hat \phi_j (x,t) = \mu_0^{-{1\over 2}} \phi_j ({x \over \mu_0} , t)$, for $j=1,2$, as natural. Let $\bar \psi =
 \psi (\lambda , \phi_1) - \psi (\lambda , \phi_2) $. We have $ \bar \psi (r, t_0 ) = 0 $ and
 \begin{align*}
 \partial_t \bar \psi &= \Delta \bar \psi + V[\lambda ] \bar \psi  + (\psi_1+ \phi_1^{in} )^5 - (\psi_2 + \phi_1^{in} )^5 \\
 &+ [2\nabla \eta_{R} \nabla_x ( \hat \phi_1 - \hat \phi_2)  +  ( \hat \phi_1 - \hat \phi_2) (\Delta_x-\partial_t)\eta_{R} ] \\
 &+(\psi_2+ \phi_1^{in} )^5 - (\psi_2 + \phi_2^{in} )^5 - 5 U_2^4 (\phi_1^{in} - \phi_2^{in} ).
 \end{align*}
 Arguing as in \eqref{gg1}-\eqref{gg2}, we get
 \begin{align*}
 \left|[2\nabla \eta_{R} \nabla_x ( \hat \phi_1 - \hat \phi_2)  +  ( \hat \phi_1 - \hat \phi_2) (\Delta_x-\partial_t)\eta_{R} ]  \right| & \leq
 \mu_0^{1\over 2 } t^{-{3\over 2}} h_0 ({r\over \sqrt{t}} )  {\| \phi_1 - \phi_2 \|_{\nu , a} \over R^{1+a} }\\
 &\leq {\bf c} \mu_0^{1\over 2 } t^{-{3\over 2}} h_0 ({r\over \sqrt{t}} )  \| \phi_1 - \phi_2 \|_{\nu , a}
 \end{align*}
 and also
 \begin{align*}
 \left| (\psi_2+ \phi_1^{in} )^5 - (\psi_2 + \phi_2^{in} )^5 - 5 U_2^4 (\phi_1^{in} - \phi_2^{in} ) \right|& \leq
 \mu_0^{1\over 2 } t^{-{3\over 2}} h_0 ({r\over \sqrt{t}} )  {\| \phi_1 - \phi_2 \|_{\nu , a} \over R^{1+a} }\\
 &\leq {\bf c} \mu_0^{1\over 2 } t^{-{3\over 2}} h_0 ({r\over \sqrt{t}} )  \| \phi_1 - \phi_2 \|_{\nu , a}.
 \end{align*}
The constant $c_1$ in the last two formulas can be made arbitrarily small   provided $R$ is chosen large enough. This concludes the proof.
\end{proof}

\setcounter{equation}{0}
\section{Choice of $\lambda$: Part I}\label{secpar}

Let $\psi = \Psi [\lambda, \phi]$ be the solution to Problem \eqref{equpsi} predicted by Proposition \ref{propext}, and satisfying the properties described in  Proposition \ref{propext1}.
We substitute $\psi$ in equations \eqref{equ3331} and \eqref{defHH}, and we want to solve, in $\phi$, Problem \eqref{equ3331}, satisfying the initial condition \eqref{equ3332}. As we stated in Proposition \ref{prop0}, Problem \eqref{equ3331}-\eqref{equ3332} can be solved for functions $\phi$ satisfying \eqref{phibound}, provided that
\begin{equation}
\label{na1}
\int_{B_{2R}} H [ \psi , \lambda , \phi ] (y, t (\tau )) \, Z_0 (y) \, dy =0, \quad {\mbox {for all}} \quad t > t_0,
\end{equation}
where $H [ \psi , \lambda , \phi ]$ is defined in \eqref{defHH}.

\medskip
\noindent
Next Lemma states that \eqref{na1} is a non linear, non local equation in $\la$, at any fixed $\phi$.

\begin{lemma}\label{l1} Assume that $ \lambda $ satisfies \eqref{lambdabound}, and that the function $\phi$ satisfies the bound \eqref{phibound}.
Let $ \psi = \Psi [\la , \phi]$ be the solution to Problem \eqref{equpsi} predicted by Proposition \ref{propext}. Then Equation \eqref{na1} is equivalent to
\begin{equation}
\label{na2}
\left[ 1 + \mu_0 \mu_0' b(t) + q_1 (\la ) \right] \phi_0 (0,t) = g(t) + G[\la , \phi ] (t).
\end{equation}
Here $\phi_0 $ is the function defined in \eqref{defphi0} and also in \eqref{duha}, thus
\begin{equation}\label{duha11}
\phi_0 (0,t) =\int_{t_0 -1}^t {1\over (4\pi (t-s) )^{3\over 2} } \, \int_{\R^3} e^{-{|y|^2 \over 4 (t-s) }} \, {\bar \alpha (s) \over \mu + |y|} \, {\bf 1}_{\{ r< M\}} \,  dy \, ds.
\end{equation}
The function $b= b(t)$ is a smooth function in $(t_0 , \infty)$. With $q_1 (s)$ we denote a smooth function so that $q_1(0) = 0 $, and $q_1' (0) \not= 0$. Moreover,
\begin{equation}\label{queen1}
\| b \|_\infty <C , \quad
\| g \|_\flat
\leq C , \quad \| G [\la , \phi ] \|_\flat
\leq C.
\end{equation}
Furthermore, if the initial time $t_0 $ in Problem \eqref{p} is chosen large enough, there exists $R$ in the definition of the cut off function in \eqref{deftildephi} sufficiently large and there exist constant  ${\bf c} \in (0,1) $ so that, for any $\phi$,
\begin{equation}\label{queen2} \| G[\la_1 , \phi ] - G[\la_2 , \phi ] \|_\flat \leq {\bf c} \| \la_1 - \la_2 \|_\sharp
\end{equation}
and, for any $\la $,
\begin{equation}\label{queen3} \| G[\la , \phi_1 ] - G[\la , \phi_2 ] \|_\flat \leq {\bf c} \| \phi_1 - \phi_2 \|_{\nu, a}.
\end{equation}
The constants ${\bf c}$ in \eqref{queen2} and \eqref{queen3} can be made as small as one needs, provided that the initial time $t_0$ is chosen large enough.
We refer to \eqref{alphabound} and \eqref{normphi} for the definitions of $\| \cdot \|_\flat$ and $\| \cdot \|_{\nu , a}$ respectively.
\end{lemma}

\medskip
\begin{proof}
Throughout the proof,  we denote by $q_i = q_i (s)$, for any interegr $i$, a smooth real function, with the property that ${d \over (d s)^j} q_i(0) = 0$, for $j <i$, and ${d \over (d s)^i} q_i(0)  \not= 0$.

We decompose
\begin{align*}
\int_{B_{2R}} H[\psi ,  \lambda , \phi ] (y, t (\tau )) \, Z_0 (y) \, dy &=
\mu_0^{5\over 2} \int_{B_{2R}} {\mathcal E}_{22} (\mu_{0} y,t) \, Z_0 (y) \, dy \\
&+ 5 \int_{B_{2R}}  \,  { \mu_{0}^{\frac{1}2} \over (1+ \la)^2 } \, w^4 ( {y \over 1+ \la}  ) \psi(\mu_{0} y,t) \, Z_0 (y) \, dy
 \\
&+ \int_{B_{2R}} B[\phi ] Z_0 (y) \, dy  \int_{B_{2R}} B^0 [\phi ] Z_0 (y) \, dy \\
&= i_1+ i_2 + i_3 + i_4.
\end{align*}
For any $j=1, \ldots , 4$, $i_j$ is a function of $t$, and depends also on $\la$ and $\phi$. To emphasize this dependence, we write
$i_j = i_j [\la , \phi ] (t)$.

We claim that
\begin{align}\label{div1}
\mu_0^{-{1\over 2}} \, i_1 [\la , \phi ] (t)  &= \mu_0^{2} \mu^{-2}  \Biggl[ \left( 5 \int_{B_{2R} }
w^4 (y) Z_0 (y) \, dy \right) \, \phi_0 (0,t) \\
&+  \left( q_{1}  (\lambda ) + \mu_0 \mu_0'  q_0 (\lambda ) \right) \, \phi_0 (0,t)  + \mu_0^{ \sigma}  \alpha (t) b (t) \Biggl],
\nonumber
\end{align}
where $b(t)$ is a smooth function in $(t_0 , \infty)$, which is uniformly bounded as $t \to \infty$.

\medskip
Observe that $i_1$ does not depend on $\phi$.
From the equation \eqref{defphi0} satisfied by $\phi_0$, and
Lemma \ref{lemacalor}, we get the existence of a positive constant $c$ so that $|\phi_0 (\mu_0 y , t ) |\leq c \alpha (t) \mu_0 (t)$
for any $y \in B_{2R}$.
Thus, we Taylor expand $ {\mathcal E}_{22}  $ in the region $y \in B_{2R}$ as follows
\begin{align*}
{\mathcal E}_{22} (\mu_{0} y,t) &= 5 U_1^4 \phi_0  + 4 (U_1 + s \phi_0 )^3 \phi_0^2 = a+b
\end{align*}
for some $s\in (0,1)$. Let us first analyze $a$. We write
\begin{align*}
a&= 5 \mu^{-2} w^4 (y) \phi_0 (0,t) +\underbrace{ 5 [ U_1^4 (\mu_0 y ) - \mu^{-2} w^4 (y) ] \phi_0 (0,t)}_{:=a_1} +
\underbrace{5 U_1^4 [ \phi_0 (\mu_0 y , t) - \phi_0 (0,t) ]}_{:=a_2}
\end{align*}
Observe that, by definition of $U_1$ in \eqref{DefU}, and \eqref{maso0}, we have
\begin{align*}
U_1^4 (\mu_0 y )- \mu^{-2} w^4 (y) & = \left[ w_\mu (\mu_0 y ) + \mu_0' \mu^{1\over 2} \Phi_1 ({\mu_0 r \over \mu } )\right]^4 - \mu^{-2} w^4 (y) \\
&= \mu^{-2} \left[ w (y) + \left( w ({y \over (1+\Lambda )^2} ) - w(y) \right)+ \mu_0' \mu \Phi_1 ({\mu_0 r \over \mu } ) \right]^4 - \mu^{-2} w^4 (y)\\
&=4 \mu^{-2} \, w^3 (y) s \left[ \left( w ({y \over (1+\Lambda )^2} ) - w(y) \right)+ \mu_0' \mu \Phi_1 ({\mu_0 r \over \mu } )\right]
\end{align*}
for some $s \in (0,1)$. Observe that
\begin{equation}\label{gi1}
w ({y \over (1+\Lambda )^2} ) - w(y) = \nabla w (y) \cdot y + \nabla w (y) \cdot y z  [-2\Lambda -\Lambda^2]
\end{equation}
for some $z \in (0,1).$
Taking into account also the description of $\Phi_1$ in \eqref{expaPhi1}, we get that
\begin{equation}\label{mmee1}
\int_{B_{2R}} a_1 Z_0 \, dy = \mu^{-2} \left[ q_1 (\Lambda )  + \mu_0 \mu_0' q_0 (\Lambda )  \right] \phi_0 (0,t).
\end{equation}
We next claim that, for $y \in B_{2R}$, we have
\begin{equation}\label{mmee2}
 \phi_0 (\mu_0 y , t) - \phi_0 (0,t) = \alpha (t) |\mu_0 y |^\sigma \, \Pi (t) \, \Theta (|y|) ,
 \end{equation}
 for some $\sigma \in (0,1)$. We postpone the proof of \eqref{mmee2} to the Appendix. We thus get
 \begin{equation}\label{mmee3}
 \int_{B_{2R}} a_2 Z_0 \, dy = \mu^{-2} \mu_0^\sigma \alpha (t)  b (t).
 \end{equation}
Collecting estimates \eqref{mmee1}-\eqref{mmee3} we get \eqref{div1}.

\medskip
We claim that
\begin{equation}\label{div22}
\mu_0^{-{1\over 2}} i_2 [ \la , \phi ] (t) = g(t) + G [\la , \phi ] (t)
\end{equation}
with
$$
\| g \|_\flat
\leq c , \quad \| G [\la , \phi ] \|_\flat
\leq c
$$
for some constant $c$. We refer to \eqref{alphabound} for the definition of $\| \cdot \|_\flat$. Furthermore, we claim that $G$ satisfies estimates \eqref{queen2} and \eqref{queen3}, for some constant $c_1 \in (0,1)$.
To prove the above assertion, we write
\begin{align*}
 \mu_0^{-{1\over 2}} i_2 [\la , \phi ] (t) &=
5  \int_{B_{2R}}  \, \, w^4 (y) \psi  [0, 0] (\mu_{0} y,t) \, Z_0 (y) \, dy \\
&+ 5  \int_{B_{2R}}  \, \, w^4 (y) [ \psi [\lambda , 0 ]  \, - \psi \,[0, 0 ] ] (\mu_{0} y,t) Z_0 (y) \, dy \\
&+ 5  \int_{B_{2R}}  \, \, w^4 (y) [ \psi [\la , \phi  ] - \psi [\la , 0 ] ] (\mu_{0} y,t) Z_0 (y) \, dy \\
&+ 5  \int_{B_{2R}} \, [ w^4 ( {y \over (1+ \Lambda )^2}  ) - w^4 (y) ] \psi [\la , \phi ] (\mu_{0} y,t)  Z_0 (y) \, dy \\
&+ 5  [{1 \over (1+ \Lambda)^4 } - 1] \int_{B_{2R}} \,  w^4 ( {y \over (1+ \Lambda )^2 }  ) \psi  [\la , \phi ] (\mu_{0} y,t) \, Z_0 (y) \, dy \\
&= \sum_{j=1}^5 g_j.
\end{align*}
The first term,
$$
g_1 (t) = 5  \int_{B_{2R}}  \, \, w^4 (y) \psi(\mu_{0} y,t) \,[0, 0] Z_0 (y) \, dy ,
$$
is an explicit smooth function, globally defined in $(t_0 , \infty)$, which satisfies the bound
\begin{equation}\label{gi2}
\| g_1 \|_\flat
\leq c \left( 5\int_{B_{2R}} w^4 (y) |y| Z_0 (y) \, dy \right)
\end{equation}
for some constant $c>0$, as direct consequence of \eqref{marte2}. Let us analyze the term $g_5$. We see that $g_5 = g_5 [\la , \phi ] (t)$. Let us first assume that $\la$ and $\phi$ are fixed. From \eqref{marte2}, we get
\begin{align*}
\left| g_5 (t) \right| & \leq c q_1 (\la )  \int_{B_{2R}} \left| w^4 (y) \psi [\la , \phi ]  (\mu_0 y , t ) Z_0 (y) \, \right| \, dy \leq c \mu_0^{3\over 2} t^{-1} q_1 (\la ) \int {|y| \over (1+ |y|^5) } \, dy.
\end{align*}
Using again \eqref{marte2} and the assumptions on $\la$ and on $\phi$, we get
$
[ g_5 ]_{0,\sigma , [t, t+1] }\leq c \mu_0^{3\over 2} t^{-1},
$ from which we conclude that $\| g_5 \|_\flat \leq c$, for some constant $c>0$.
Let us now fix $\phi$ and take $\la_1$, $\la_2$ satisfying \eqref{lambdabound}. We write
\begin{align*}
g_5 [\la_1 , \phi ]  - g_5 [\la_2 , \phi]  &= 5  [{1 \over (1+ \Lambda_1)^4 } - {1 \over (1+ \Lambda_2)^4 } ] \int_{B_{2R}} \,  w^4 ( {y \over (1+ \Lambda_1 )^2 }  ) \psi  [\la_1 , \phi ] (\mu_{0} y,t) \, Z_0 (y) \, dy \\
&+ 5  [ {1 \over (1+ \Lambda_2)^4 } -1 ] \int_{B_{2R}} \,  [ w^4 ( {y \over (1+ \Lambda_1 )^2 }  ) -
w^4 ( {y \over (1+ \Lambda_2 )^2 }  ) ] \psi  [\la_1 , \phi ] (\mu_{0} y,t) \, Z_0 (y) \, dy \\
&+  5  [ {1 \over (1+ \Lambda_2)^4 } -1 ] \int_{B_{2R}} \,
w^4 ( {y \over (1+ \Lambda_2 )^2 }  ) [  \psi  [\la_1 , \phi ]- \psi  [\la_2 , \phi ] ]  (\mu_{0} y,t) \, Z_0 (y) \, dy \\
&= e_1 + e_2 + e_3.
 \end{align*}
 Thanks to \eqref{lambdabound}, and arguing as before, we see that
 \begin{align*}
 |e_1 (t) | &\leq c |\Lambda_1 (t) - \Lambda_2 (t) | \int_{B_{2R}} \left| w^4 (y) \psi [\la_1 , \phi ]  (\mu_0 y , t ) Z_0 (y) \, \right| \, dy \\
 &\leq c \mu_0(t)^{3\over 2} t^{-1} \left( \int_t^\infty s^{-1} \mu_0 (s) \, ds \right) \| \la_1 - \la_2 \|_\sharp \\
 &\leq [\mu_0 (t_0) ]  \mu_0(t)^{3\over 2} t^{-1} \| \la_1 - \la_2 \|_\sharp
  \leq c_1  \mu_0(t)^{3\over 2} t^{-1} \| \la_1 - \la_2 \|_\sharp
 \end{align*}
 where $c_1$ is a positive number, which can be chosen arbitrarily small, in particular $c_1 <1$, provided $t_0$ is chosen large enough.
 Similarly one can show that, thanks to \eqref{lambdabound},
 $$
 [ e_1 ]_{0,\sigma , [t, t+1]} \leq c_1 \mu_0(t)^{3\over 2} t^{-1} \| \la_1 - \la_2 \|_\sharp.
 $$
 We thus can conclude that there exists a positive small number $c_1 <1$ so that
 $$
 \| e_1 \|_\flat \leq c_1 \| \la_1 - \la_2 \|_\sharp.
 $$
 A similar argument allow us to say that also $\| e_2 \|_\flat \leq c_1 \| \la_1 - \la_2 \|_\sharp$. We next analyze $e_3$.
 From \eqref{voto1} we get that
 \begin{align*}
 |e_3 (t) | & \leq \mu_0^{3\over 2} t^{-1} \left( \int w^4 (y) {|y| \over 1+ |y|} \, dy \right) \| \psi [\la_1 , \phi ] -\psi [\la_2 , \phi] \|_{**} \\
 &\leq c_1 \mu_0^{3\over 2} t^{-1} \| \la_1 - \la_2 \|_\sharp,
 \end{align*}
 and also
 $$
 [e_3]_{0,\sigma , [t,t+1]} \leq c_1 \mu_0^{3\over 2} t^{-1} \| \la_1 - \la_2 \|_\sharp,
 $$
 for some constant $c_1 \in (0,1)$. We can conclude that
 $$
 \| g_5 [\la_1 , \phi ]  - g_5 [\la_2 , \phi]  \|_\flat \leq c_1 \| \la_1 - \la_2 \|_\sharp.
 $$
 The same estimate can be obtained for $g_4$, arguing in a similar way.

 Let us now consider $g_2$. This term does not depend on $\phi$, namely $g_2 [\la, \phi ](t) = g_2 [\la ] (t)$. From Proposition \ref{propext1}, we get
 \begin{align*}
 \left| g_2 (t) \right|&\leq \mu_0^{3\over 2} t^{-2} \left( \int w^4 {|y| \over 1+ |y| } \, dy \right) \| \la \|_\sharp \leq c \mu_0^{3\over 2} t^{-2} \| \la \|_\sharp,
 \end{align*}
 and similarly
 \begin{align*}
 \left[ g_2 (t) \right]_{0,\sigma , [t,t+1]} &\leq c \mu_0^{3\over 2} t^{-2} \| \la \|_\sharp.
 \end{align*}
 Furthermore, if $t_0$ is large enough, there exists $c_1 \in (0,1)$ so that
 \begin{align*}
 \left| g_2 [\la_1 ] (t) - g_2 [\la_2] (t) \right| &\leq 5 \int_{\R^3} w^4 (y) \left\| \left[ \psi [\la_1 , 0] - \psi [\la_2, 0] \right] (\mu_0 y,t)\right| Z_0 dy \\
 &\leq C t_0^{-1} \mu_0^{3\over 2} t^{-2} \| \la_1 - \la_2  \|_\sharp \leq c_1 \mu_0^{3\over 2} t^{-2} \| \la_1- \la_2 \|_\sharp
\end{align*}
and also
$$
\left[ g_1 [\la_2 ] - g_2 [\la_2] \right]_{0,\sigma , [t,t+1]} \leq c_1 \mu_0^{3\over 2} t^{-2} \| \la_1- \la_2 \|_\sharp
$$
thanks to the results of Proposition \ref{propext1}. Arguing in the same way, one gets similar estimates for $g_3$.

Collecting all the above arguments, we conclude that $\mu_0^{-{1\over 2}} i_2 [\la , \phi ] (t)$ can be written as in
\eqref{div22}, with $g$ and $G$ satisfying \eqref{queen1}, \eqref{queen2} and \eqref{queen3}.

\medskip
Next we claim that
\begin{equation} \label{div33}
\mu_0^{-{1\over 2}} i_j [\la , \phi ] (t) = G [\la, \phi ] (t), \quad j=3,4,
\end{equation}
and $G$ satisfies \eqref{queen1}, \eqref{queen2} and \eqref{queen3}.
We start with $j=3$. First, we see that $i_3$ does not depend on $\la$, and it is linear in $\phi$. Since we are assuming that
$\phi$ satisfies \eqref{phibound}, we have
$$
\left| \mu_0^{-{1\over 2} } i_3 (t) \right| \leq \left(\mu_0 \mu_0' R^{2-a} \right)  \mu_0^{3\over 2} (t) t^{-1} \| \phi \|_{\nu, a}
\leq c \mu_0^{3\over 2} (t) t^{-1} \| \phi \|_{\nu, a}
$$
and
$$
\left[ \mu_0^{-{1\over 2} } i_3 (t) \right]_{0,\sigma , [t,t+1]}
\leq c \mu_0^{3\over 2} (t) t^{-1} \| \phi \|_{\nu, a}
$$
for some constant $c>0$. Let us know take $\phi_1$, and $\phi_2$, and we get that, if $\mu_0 (t_0 ) \mu_0'  (t_0)  R^{2-a}$ is small enough,
$$
\left| \mu_0^{-{1\over 2} } \left( i_3 [\phi_1] -  i_3 [\phi_2] \right) (t) \right| \leq c_1 \mu_0^{3\over 2} (t) t^{-1} \| \phi_1 - \phi_2  \|_{\nu, a}
$$
and
$$
\left[ \mu_0^{-{1\over 2} } \left( i_3 [\phi_1] -  i_3 [\phi_2] \right) (t) \right]_{0,\sigma , [t,t+1]} \leq c_1 \mu_0^{3\over 2} (t) t^{-1} \| \phi_1 - \phi_2  \|_{\nu, a}
$$
for some $c_1 \in (0,1)$.
Estimate \eqref{div33} for $j=4$ can be proved in a very similar way. We leave the details to the interested reader.
Combining \eqref{div1}, \eqref{div22} and \eqref{div33}, we complete the proof of \eqref{na2}. This concludes the proof of the Lemma.

\end{proof}

\setcounter{equation}{0}

\section{Solving a non local linear problem}\label{nonlocal}

Let $\phi_0$ be the function introduced in \eqref{defphi0}. Later in our argument we will need to solve in $\lambda$,  a non local equation of the form
\begin{equation}
\label{jun}\phi_0 (0,t) = h(t), \quad t \in (t_0 , \infty)
\end{equation}
for a certain right hand side $h$. We see from \eqref{duha11} that $\phi_0 (0,t)$, defined as
$$
\phi_0 (0,t) =\int_{t_0 -1}^t  \int_{\R^3} {\bar \alpha (s) \over (4\pi (t-s) )^{3\over 2} } \,  { e^{-{|y|^2 \over 4 (t-s) }}   \over \mu + |y|} \, {\bf 1}_{\{ |y| < M\}} \,  dy \, ds,
$$
 defines a non-local non-linear operator in $\la$.
 For convenience we recall that
$$
\alpha (t)  = 3^{1\over 4}  \, \mu_0^{-{1\over 2}} \, \left( \mu_0 \Lambda \right)', \quad  \bar \alpha (t) = \left\{ \begin{matrix}
\alpha (t_0)  & \quad {\mbox {for}} \quad t< t_0 \\
\alpha (t)  & \quad {\mbox {for}} \quad t \geq t_0
\end{matrix}\right. , \quad \Lambda (t) = \ \int_t^\infty \la (s) \, ds.
$$
We write
\begin{equation}\label{duha12}
\phi_0 (0,t) = T [\la ] (t)  + \hat T [\la ] (t ),
\end{equation}
where $T$ is
\begin{equation}\label{defT}
T[\la ] (t) =
\int_{t_0-1 }^t \int_{\R^3} {\bar \alpha (s) \over  (4\pi (t-s) )^{3\over 2}} \, {e^{-{|y|^2 \over 4 (t-s) }} \over |y|}  {\bf 1}_{\{ |y|< M \}} \, dz \, ds.
\end{equation}
We shall see that $\hat T$ is a small perturbation of $T$, in some sense we will precise later. In this section, we start with the analysis of  Problem
\begin{equation}\label{freddy}
T[\la ] (t) = h(t), \quad t>t_0.
\end{equation}

\medskip
\noindent
Straightforward computations  give that
\begin{equation}
\label{jun1}
T[\la ] (t)  = - {\bar \omega_3 \over 4} \int_{t_0 -1}^t {\bar \alpha (s) \over \sqrt{t-s}} \, \left( 1- e^{-{M^2 \over  (t-s)}}  \right) ds.
\end{equation}
Indeed,
 letting $z= {y \over 2\sqrt{t-s}}$, one gets
\begin{align}\label{defphi00}
 T[\la ] (t) &= \int_{t_0-1 }^t \int_{\R^3} {\bar \alpha (s) \over 2 \sqrt{t-s}} \, {e^{-|z|^2} \over |z|}  {\bf 1}_{\{ |z|< {M \over \sqrt{t-s}}\}} \, dz \, ds \nonumber \\
&={\bar \omega_3 \over 2} \int_{t_0-1 }^t \int_0^\infty {\bar  \alpha (s) \over \sqrt{t-s}} e^{-\rho^2} \rho  {\bf 1}_{\{ \rho < {M \over \sqrt{t-s}} \}} \,  d\rho \, ds ={\bar \omega_3 \over 4} \int_{t_0-1}^t {\bar  \alpha (s) \over \sqrt{t-s}} \, \int_0^{M \over \sqrt{t-s}} e^{-\rho^2} 2\rho \, d\rho   \nonumber \\
&=- {\bar \omega_3 \over 4} \int_{t_0-1}^t {\bar  \alpha (s) \over \sqrt{t-s}} \, \left( 1- e^{-{M^2 \over  (t-s)}}  \right) ds.
\end{align}

\medskip
\noindent
Introduce the function $\beta = \beta(t)$ as
\begin{equation}\label{defbeta}
\beta (t) = {\bar \omega_3 \over 4} \int_t^\infty \bar \alpha (s)\, ds .
\end{equation}
If  $\beta = \beta (t)$ solves
\begin{equation}\label{mp1}
\int_{t_0-1 }^t {\beta'  (s) \over \sqrt{t-s}} \, \left( 1- e^{-{M^2 \over  (t-s)}}  \right) ds = h(t),
\end{equation}
then the function  $\Lambda (t) = \int_t^\infty \la (s) \, ds$, defined as
\begin{equation}
\label{Lambdasol} \bar \omega \Lambda (t) = \mu_0^{-{1\over 2} } (t) \beta (t) + {\mu_0^{-1} (t) \over 2} \, \int_t^\infty \beta (s) \mu_0^{-{1\over 2}} \, \mu_0' (s) \, ds, \quad \bar \omega =  {\bar \omega_3 \over 4} 3^{1\over 4},
\end{equation}
solves \eqref{freddy}.

Next Lemma constructs a solution to \eqref{mp1}. If we formally let $M \to \infty$ in \eqref{mp1}, we get that the left hand side of \eqref{mp1} is nothing but the ${1\over 2}$-Caputo derivative of $\beta$. This fact inspires the proof of the following

\medskip
\begin{lemma} \label{lemaalpha}
	Let $\sigma = {1\over 2} + \sigma'$, with $\sigma' >0$ small, be the number fixed in \eqref{lambdabound}, and  $h:(t_0, \infty) \to \R$ a smooth function  satisfying
	\begin{equation}\label{mp0}
	\sup_{t> t_0} \,  \mu_0^{-{3\over 2} } t \left[  \|h \|_{0, [t,t+1]} + [h ]_{0,\sigma, [t,t+1]} \, \right] \leq C,
	\end{equation}
	for some constant $C$. Then
	there exist a constant $C_1$ and a unique smooth function $\beta : (t_0 -1 , \infty ) \to \R$ which solves \eqref{mp1}, $\beta \in C^1$ and satisfies the bounds
	\begin{equation}\label{mp2}
	\sup_{t> t_0} \, \mu_0^{-{3\over 2} } t \left[  \|\beta'  \|_{0, [t,t+1]} + [\beta'  ]_{0,\sigma, [t,t+1]} \, \right] \leq C_1 \, M^{-1}.
	\end{equation}
We recall that $M^2 = t_0$, was first introduced in \eqref{defphi0}.
\end{lemma}

\medskip
Observe that a direct consequence of this Lemma, together with \eqref{Lambdasol} and \eqref{defaaa} is
 the invertibility theory for Problem \eqref{freddy} that will be used in next Section to solve \eqref{na1}. This is contained in the following

\begin{proposition}\label{freddyprop}

The function $T : X_\sharp \to X_\flat$, defined in \eqref{defT} is a linear, non-local, homeomorphism so that
\begin{equation}\label{au1}
\| T^{-1} (h) \|_\sharp \leq C M^{-1} \| h \|_\flat, \quad {\mbox {for any}} \quad h \in X_\flat
\end{equation}
for some fixed positive constant $C$. We refer to \eqref{lambdabound} and to \eqref{Xsharp} for the definition of the $\| \cdot \|_\sharp$-norm and of the set $X_\sharp$, and to \eqref{alphabound} and \eqref{Xflat}
for the definition of the norm $\| \cdot \|_\flat$ and of the space $X_\flat$.
\end{proposition}

\medskip
We devote the rest to the Section to the

\begin{proof}[Proof of Lemma \ref{lemaalpha}]
We start performing a change of variables, to transform Problem \eqref{mp1} into an equivalent one with simpler form: let
$$
s= t_0 + M^2 a, \quad t= t_0 + M^2 b, \quad \tilde \beta (a) = \beta (s), \quad \tilde h (b ) = h (t).
$$
After this change of variables, Problem \eqref{mp1} takes the form
\begin{equation}
\label{mp11}
\int_0^b {\tilde \beta' (a) \over \sqrt{b-a}} \, \left( 1- e^{-{1\over b-a}} \right) \, da = M\, \tilde h (b).
\end{equation}

\medskip
Let $K (\eta ) = {1 - e^{-{1\over \sqrt{\eta}}}\over \sqrt{\eta}}$ and take the Laplace transform of both sides in \eqref{mp11}, thus getting
$$
{\mathcal L} \left( {\tilde \beta}' \right) (\xi ) {\mathcal L} \left(K\right) (\xi ) =M  {\mathcal L} \left( {\tilde h} \right) (\xi).
$$
Since ${\mathcal L} \left( {\tilde \beta}' \right) = \xi {\mathcal L} \left(\tilde \beta \right) (\xi ) - \tilde \beta (0)$, we get
\begin{equation}
\label{lap00}
{\mathcal L} \left(\tilde \beta  \right) (\xi ) = {\tilde \beta (0) \over \xi} +M {{\mathcal L} \left( \tilde h \right) (\xi ) \over \xi {\mathcal L} \left(K\right) (\xi ) }
\end{equation}
Observe now that
$$
{\mathcal L} \left(K\right) (\xi ) = \int_0^\infty e^{-\xi \eta} \, \left( {1 - e^{-{1\over \eta}}\over \sqrt{\eta}} \right) \, d \eta =
{2 \over \sqrt{\xi}}  \int_0^\infty e^{-p^2} \,\left(  1-e^{-{\xi \over p^2}}  \right) \, dp.
$$
We readily get that
\begin{equation}
\label{lap0}
{\mathcal L} \left(K\right) (\xi ) = {1\over \sqrt{\xi}} \left( 2\int_0^\infty e^{-p^2} \, dp \right) \, \left( 1+ o(1) \right) , \quad {\mbox {as}} \quad \xi \to \infty.
\end{equation}
To describe the behavior of ${\mathcal L} \left( K \right) (\xi)$, for $\xi \to 0$, we first notice that
\begin{align*}
\int_0^{1\over \xi} e^{-\xi \eta} \, \left( {1 - e^{-{1\over \eta}}\over \sqrt{\eta}} \right) \, d \eta &= \int_0^\infty {1 - e^{-{1\over \eta}}\over \sqrt{\eta}}  \, d\eta \\
&-  \underbrace{ \int_{1\over \xi}^\infty {1 - e^{-{1\over \eta}}\over \sqrt{\eta}} \, d\eta -\xi \int_0^{1\over \xi} \eta \, {1 - e^{-{1\over \eta}}\over \sqrt{\eta}}  \, d\eta }_{O(\sqrt{\xi})} + O(\xi) \\
&= \int_0^\infty {1 - e^{-{1\over \eta}}\over \sqrt{\eta}}  \, d\eta + O(\sqrt{\xi}).
\end{align*}
On the other hand,
\begin{align*}
\int_{1\over \xi}^\infty e^{-\xi \eta} \, \left( {1 - e^{-{1\over \eta}}\over \sqrt{\eta}} \right) \, d \eta &= \int_{1\over \xi}^\infty e^{-\xi \eta} \, \left( {1 -{1\over \eta }-  e^{-{1\over \eta}}\over \sqrt{\eta}} \right) \, d \eta \\
&+ \int_{1\over \xi}^\infty \,  { e^{-\xi \eta}   \over \eta \sqrt{\eta}}  \, d \eta = O (\sqrt{\xi}).
\end{align*}
Thus we conclude that
\begin{equation}
\label{lap1} {\mathcal L} \left(K\right) (\xi ) = \int_0^\infty {1 - e^{-{1\over \eta}}\over \sqrt{\eta}}  \, d\eta + O (\sqrt{\xi}), \quad {\mbox {as}} \quad \xi \to 0.
\end{equation}
From \eqref{lap0} and \eqref{lap1}, we conclude that
$$
{1\over \xi {\mathcal L} \left(K\right) (\xi )} =
\left\{ \begin{matrix} {c_1 \over \xi} + {c_2 \over \sqrt{\xi}} + O(1)  & \quad {\mbox {if}} \quad \xi \to 0 \\
{c_3 \over \sqrt{\xi}} \, (1+ o(1))  & \quad {\mbox {if}} \quad \xi \to \infty.
\end{matrix} \right.
$$
Let now $G = G(t)$ be so that ${\mathcal L} \left(G \right) (\xi) = {1\over \xi {\mathcal L} \left( K\right) (\xi )}$.
Standard arguments on Laplace transformation imply that
$$
G(t)= \left\{ \begin{matrix} \tilde c_1  + {\tilde c_2 \over \sqrt{t}} + O({1\over t})  & \quad {\mbox {if}} \quad t \to \infty \\
{\tilde c_3 \over \sqrt{t}} \, (1+ o(1))  & \quad {\mbox {if}} \quad t \to 0.
\end{matrix} \right. ,
$$
for certain constants $\tilde c_1$, $\tilde c_2$ and $\tilde c_3$.
From \eqref{lap00}, taking the anti-Laplace transform of both sides, we get
\begin{align*}
\tilde \beta (b) &= \tilde \beta (0) + M \int_0^b \tilde h (a) G(b-a) \, da \\
&=\tilde \beta (0) +M  \tilde c_1 \int_0^\infty \tilde h(a) \, da
+M \tilde c_1 \int_b^\infty \tilde h(a) \, da + M \int_0^b \tilde h (a) \left[ G(b-a) - \tilde c_1 \right] \, da .
\end{align*}
We select the solution to Problem \eqref{mp11} so that
$$
\tilde \beta (0) + M \tilde c_1 \int_0^\infty \tilde h(a) \, da =0 .
$$
In the original variables, we thus obtain an explicit solution to \eqref{mp1}
\begin{equation}
\label{mp1sol}
\beta(t) = \underbrace{{\tilde c_1\over M} \int_t^\infty h(s) \, ds}_{:=\beta_1 (t)} + \underbrace{{1\over M} \int_{t_0}^t h(s) \, \left[ G \left( {t-s \over M^2} \right) -\tilde c_1 \right] \, ds}_{:=\beta_2 (t)}.
\end{equation}
Let us now check \eqref{mp2}. Since \eqref{mp0} holds, we easily get that
$$
\sup_{t>t_0} \mu_0^{-{3\over 2}}  \left| \beta_1  (t) \right| \lesssim M^{-1} .
$$
To control the second term in \eqref{mp1sol}, we change variable $t=M^2 \bar t$, $s = M^2 \bar s$, so that
$$
\beta_2 (t) = M \int_{t_0 \over M^2}^{\bar t} h(M^2 \bar s) \, \left[ G \left( \bar t- \bar s \right) -\tilde c_1 \right] \, d\bar s.
$$
Since $t_0 = M^2$ and since \eqref{mp0} holds, we get
\begin{align*}
\left| \beta_2 (t) \right| &\lesssim {1\over M} \int_1^{\bar t} {\mu_0^{3\over 2} (\bar s) \over {\bar s}} \, \left[ G \left( \bar t- \bar s \right) -\tilde c_1 \right] \, d\bar s \lesssim {1\over M} \mu_0^{3\over 2} (\bar t) \lesssim M^{-1} \mu_0^{3\over 2}  (t) ,
\end{align*}
from which we get the validity of \eqref{mp2}.

\medskip
The assumption that $\mu_0^{-{3\over 2}}t [h]_{0, \sigma , [t,t+1]}$  is bounded guarantees that the function $\beta$ defined in \eqref{mp1sol} is differentiable.  Indeed, trivially one has $\beta_1' (t) = -{\tilde c_1 \over M} h (t)$. Let us write $\beta_2$ in the following way
$$
\beta_2 (t) = {1\over M} \int_{t_0}^t ( h(s) - h(t) )  \left[G \left( {t-s \over M^2} \right) -\tilde c_1 \right] \, ds +
{h(t) \over M} \int_{t_0}^t  \, \left[ G \left( {t-s \over M^2} \right) -\tilde c_1 \right] \, ds.
$$
Thus we have
\begin{eqnarray*}
\beta_2' (t) &=& \underbrace{{1\over M} \lim_{s \to t} \left[ ( h(s) - h(t) )  [G \left( {t-s \over M^2} \right) -\tilde c_1 ] \right]}_{=0}
+ {1\over M^3} \int_{t_0}^t ( h(s) - h(t) )  G' \left( {t-s \over M^2} \right)  \, ds \\
&+&\underbrace{{h'(t) \over M} \int_{t_0}^t  \, \left[ G \left( {t-s \over M^2} \right) -\tilde c_1 \right] \, ds -{h'(t) \over M} \int_{t_0}^t  \, \left[ G \left( {t-s \over M^2} \right) -\tilde c_1 \right] \, ds}_{=0}\\
&+& {h(t) \over M} {d \over dt} \left( \int_{t_0}^t  \, \left[ G \left( {t-s \over M^2} \right) -\tilde c_1 \right] \, ds \right) \\
&=& {1\over M^3} \int_{t_0}^t ( h(s) - h(t) )  G' \left( {t-s \over M^2} \right)  \, ds + {h(t) \over M} {d \over dt} \left( \int_{t_0}^t  \, \left[ G \left( {t-s \over M^2} \right) -\tilde c_1 \right] \, ds \right).
\end{eqnarray*}
Both the last two integrals are well defined, as consequence of the behavior of $G (\eta) $, as $\eta \to 0$, and the assumption that
$\mu_0^{-{3\over 2}}t [h]_{0, \sigma , [t,t+1]}$  is bounded.
Since $G(\eta ) \sim \eta^{-{1\over 2}}$, as $\eta \to 0$,
direct computations give the bounds in \eqref{mp2} for $\beta' (t)$.
This concludes the proof of the Lemma.
\end{proof}

\setcounter{equation}{0}
\section{Choice of $\lambda$: Part II}\label{secpar1}

This Section is devoted to solve in $\la$ Equation \eqref{na1}, for fixed $\phi$ satisfying \eqref{phibound}.
We have the validity of the following

\begin{proposition}\label{proplambdafinal}
For any $\phi $ satisfying \eqref{phibound}, there exists $L>0$ and a unique solution $\la = \la [\phi]$ to Equation \eqref{na1}, with
\begin{equation}\label{sunday-1}
\| \la \|_\sharp \leq L M^{-1}
\end{equation}
where $M = \sqrt{t_0}$, provided the initial time $t_0$ in Problem \eqref{p} is chosen large enough. Furthermore, there exists a constant ${\bf c} \in (0,1)$ such that, for any $\phi_1$, $\phi_2$ satisfying \eqref{phibound}, we have
\begin{equation}\label{sunday-2}
\| \la [\phi_1 ] - \la [\phi_2] \|_\sharp \leq {\bf c} \| \phi_1 - \phi_2 \|_{\nu , a}.
\end{equation}
\end{proposition}

\medskip
\begin{proof}[Proof of Proposition \ref{proplambdafinal}.]
Lemma \ref{l1} states that solving Equation \eqref{na1} is equivalent to solve \eqref{na2}. We write \eqref{na2} as follows
\begin{equation}\label{na3}
T[\la] (t) + \hat T [\la ] (t) = \left( 1+ \mu_0 \mu_0' b(t) + q_1 (\la ) \right)^{-1} \, \left[ g(t) + G[\la , \phi ] (t) \right],
\end{equation}
where $T$ and $\hat T$ are defined in \eqref{duha12}  and \eqref{defT}, while $b$, $g$ and $G$ satisfy the bounds in \eqref{queen1},\eqref{queen2} and \eqref{queen3}. Here $q_1 =q_1 (s ) $ denotes a smooth function such that $q_1 (0) = 0$ and $q_1' (0) \not= 0$.
We observe first that
$$
\left( 1+ \mu_0 \mu_0' b(t) + q_1 (\la ) \right)^{-1} \, \left[ g(t) + G[\la , \phi ] (t) \right] = g_1 (t) + G_1 [\la , \phi ] (t),
$$
for some new functions $g_1$ and $G_1$ that also satisfy \eqref{queen1}, \eqref{queen2}, and \eqref{queen3}.

\medskip
Thanks to the result of Proposition \ref{freddyprop}, solving in $\la$ Equation \eqref{na3} reduces to find the fixed point problem
\begin{equation}\label{na4}
\la (t) = {\mathcal F} (\la) (t) , \quad {\mathcal F} (\la )  := T^{-1} \left( g_1  + G_1 [\la , \phi]  - \hat T [\la ]  \right)
\end{equation}
where $T^{-1}$ is the operator introduced in Proposition \ref{freddyprop}.

\medskip
{\it Step 1}. \ \ First we show that, for any fixed $\phi $ satisfying \eqref{phibound}, there exists a unique fixed point $\la = \la [\phi]$ of contraction type for ${\mathcal F}$ in the set
$$
B=\{ \la \in X_\sharp \, : \, \| \la \|_\sharp \leq L M^{-1} \}
$$
for some $L>0$ large.

\medskip
 In order to prove this fact, we claim that, if the initial time $t_0$ in Problem \eqref{p} is large enough, there are positive constants $\bar c_1$, $\bar c_2 \in (0,1)$ so that, for any $\la \in B$,
\begin{equation}
\label{magic1}
\| \hat T [\la ] \|_\flat \leq \bar c_1 M \| \la \|_\sharp, \quad {\mbox {with}} \quad \bar c_1 C < 1
\end{equation}
and
\begin{equation}
\label{magic2}
\| \hat T[\la_1] - \hat T [\la_2 ] \|_\flat \leq \bar c_2 \| \la_1 - \la_2 \|_\sharp \quad {\mbox {with}} \quad C M^{-1} ({\bf c} +\bar c_2 ) < 1,
\end{equation}
for any $\la_1 $, $\la_2 \in B$. The constant $C$ is the constant appearing in \eqref{au1}, Proposition \ref{freddyprop}, while ${\bf c}$ is the constant is the one appearing in \eqref{queen2}.

\medskip
Assume for the moment the validity  of \eqref{magic1} and \eqref{magic2}. For any $\la \in B$, we have
\begin{align*}
\| {\mathcal F} (\la ) \|_\sharp &\leq C M^{-1} \| g_1 + G_1 [\la , \phi] - \hat T [\la ] \|_\flat \leq C M^{-1} \left( \| g_1 \|_\flat + \| G_1 [\la , \phi] \|_\flat + \| \hat T [\la ] \|_\flat \right) \\
& \leq C M^{-1} \left( 2c + \bar c_1  L \right) \leq L M^{-1}
\end{align*}
provided $L > {2cC \over 1-\bar c_1 C}$, where $C$ is the constant in \eqref{au1}, $c$ are the constants in \eqref{queen1}, and $\bar c_1$ is the constant in \eqref{magic1}, which satisfies $\bar c_1 C <1$.

\medskip
Let us take now $\la_1 $, $\la_2 \in B$. We have
\begin{align*}
\| {\mathcal F} (\la_1 ) - {\mathcal F} (\la_2 ) \|_\sharp &=
\| T^{-1} ( G_1 [\la_1 , \phi ] - G_1 [\la_2 , \phi] ) - T^{-1} (\hat T[\la_1 ] -\hat T [\la_2] ) \|_\sharp \\
&\leq C M^{-1}\left( \| G_1 [\la_1 , \phi] - G_1 [\la_2 , \phi] \|_\flat + \| \hat T [\la_1 ] - \hat T [\la_2 ] \|_\flat \right]\\
&\leq C M^{-1} (c_1 + \bar c_2 ) \| \la_1 - \la_2 \|_\sharp < \ve \| \la_1 - \la_2 \|_\sharp,
\end{align*}
for some $\ve <1$, thanks to the choice of $\bar c_2$ in \eqref{magic2}.

A direct application of Banach fixed point gives the existence and uniqueness of a solution $\la $ to Equation \eqref{na1}, satisfying \eqref{sunday-1}. We complete the first part of the proof of the Proposition with the proofs of \eqref{magic1} and \eqref{magic2}.

\medskip
\noindent
{\it Proof of \eqref{magic1}.}
Let $\la \in B$. From \eqref{duha12} and \eqref{defT}, we get
\begin{align*}
\hat T [\la ] (t) & =
 -
\int_{t_0-1 }^t \int_{\R^3} {\bar \alpha (s)  \over  (4\pi (t-s) )^{3\over 2}} \, \,   {e^{-{|y|^2 \over 4 (t-s) }}  \over |y| }  { \mu (s) \over \mu (s) + |y|}  {\bf 1}_{\{ |y|< M \}} \, dz \, ds \\
&= \bar c  \int_{t_0 -1}^t {\bar \alpha (s) \mu(s)  \over \sqrt{t-s}} \int_0^{M\over \sqrt{t-s}} e^{-\rho^2} {\rho \over \mu + \rho} \, d\rho  ds,
\end{align*}
for some explicit constant $\bar c$.
Since $\left| \int_0^{M\over \sqrt{t-s}} e^{-\rho^2} {\rho \over \mu + \rho} \, d\rho \right| \leq c {M\over \sqrt{t}}$, for any $t$ large, we
observe that
\begin{equation}\label{sunday0}
\left| \hat T [\la ](t) \right| \leq A  {M \over \sqrt{t}} \, \left| \int_{t_0 -1}^t {\bar \alpha (s) \mu (s)  \over \sqrt{t-s}} \, ds \right|,
\end{equation}
for some fixed constant $A$.
We claim that
\begin{equation}
\label{sunday}
 \int_{t_0 -1}^t {\bar \alpha (s) \mu (s) \over \sqrt{t-s}} \, ds =\bar  \alpha (t) \mu (t)  \sqrt{t - t_0 +1} \, \Pi (t), \quad t > t_0,
 \end{equation}
 for some smooth and uniformly bounded function $\Pi (t)$.
Indeed, we write, for $\beta_* (s) = \bar \alpha (s) \mu (s) $,
\begin{equation}\label{sunday1}
 \int_{t_0 -1}^t {\beta_* (s) \over \sqrt{t-s}} \, ds =  \int_{t_0 -1}^t {\beta_* (s) -\beta_* (t)  \over \sqrt{t-s}} \, ds
 + 2 \beta_* (t) \, \sqrt{t-t_0+1}   =i+  2 \beta_* (t) \, \sqrt{t-t_0+1} .
 \end{equation}
Use the change of variables $x= \sqrt{t-s}$
$$
i= -2 \int_{0}^{\sqrt{t-t_0+1} } \left[ \beta_* (t) - \beta_* (t-x^2) \right] \, dx
=-2 \beta_* (t) \int_{0}^{\sqrt{t-t_0+1} } {\left[ \beta_* (t) - \beta_* (t-x^2) \right] \over \beta_* (t) } \, dx.
$$
We now observe that the function $x \to {\left[ \beta_* (t) - \beta_* (t-x^2) \right] \over \beta_* (t) }$ is uniformly bounded in
$x\in [0, \sqrt{t-t_0+1} ]$, since
$$
{\left[ \beta_* (t) - \beta_* (t-x^2) \right] \over \beta_* (t) } =
\left\{ \begin{matrix}
1- (1-{x^2 \over t} )^{-1} (1- {x^2 \over t} )^{-{3\over 2} \bar \gamma -1}  & \quad {\mbox {for}} \quad \gamma \not=2 \\
1- (1-{x^2 \over t} )^{-{5\over 2}} [1+ \log (1- {x^2 \over t} )]^3 & \quad {\mbox {for}} \quad \gamma=2
\end{matrix}\right.
$$
where $\bar \gamma = 1$ if $\gamma >2$, and $\bar \gamma = \gamma -1$ if $1< \gamma <2$. With this in mind, we conclude that
\begin{equation}\label{sunday2}
i= \beta_* (t) \sqrt{t-t_0+1} \, \, \Pi (t)
\end{equation}
for some smooth and bounded function $\Pi$. Inserting \eqref{sunday2} into \eqref{sunday1}, we get \eqref{sunday}.

\medskip
Using \eqref{sunday} in \eqref{sunday0}, we conclude that
$$
\left| \hat T [\la ](t) \right| \leq A \mu_0 (t) M \| \la \|_\sharp \left[ \mu_0^{3\over 2} (t) t^{-1} \right],
$$
for some fixed constant $A$, independent of $t$ and of $M$. Thus, for $t$ large, if we choose $t_0$ sufficiently large, there exists a constant $c_1 \in (0,1)$ such that
$$
\left| \hat T [\la ](t) \right| \leq c_1 \, M \, \| \la \|_\sharp \left[ \mu_0^{3\over 2} (t) t^{-1} \right].
$$
Let now consider $t_1$ and $t_2 \in [t, t+1]$. We write
\begin{align*}
\hat T[\la] (t_1) & - \hat T[\la ] (t_2) =
\bar c   \int_{t_0 -1}^{t_1}  \left[ {\bar \alpha (s) \over \sqrt{t_1 -s}} - {\bar \alpha (s) \over \sqrt{t_2 -s}} \right]\int_0^{M\over \sqrt{t_1-s}} e^{-\rho^2} {\rho \mu \over \mu  + \rho} \, d\rho ds \\
&- \bar c  \int_{t_0 -1}^{t_1}  {\bar \alpha (s) \over \sqrt{t_2 -s}} \int_{M\over \sqrt{t_2-s}}^{M\over \sqrt{t_1-s}} e^{-\rho^2} {\rho \mu \over \mu + \rho} \, d\rho ds \\
&- \bar c \int_{t_1}^{t_2}  {\bar \alpha (s) \over \sqrt{t_2 -s}}  \int_0^{M\over \sqrt{t_2-s}} e^{-\rho^2} {\rho \mu \over \mu  + \rho} \, d\rho ds
= \sum_{j=1}^3 i_j
\end{align*}
Observe that, for $t_1$, $t_2 \in [t, t+1]$, for $t$ large, we have
\begin{eqnarray}\label{sunday--}
\sup_{t_1 , t_2 \in [t, t+1] } {|\mu(t_1 ) - \mu(t_2)| \over |t_1 - t_2 |^\sigma} & \leq C \mu_0 (t)  \sup_{t_1 , t_2 \in [t, t+1] } { |\Lambda (t_1) - \Lambda (t_2) | \over |t_1 - t_2 |^\sigma}  \nonumber \\
&\leq
 C M^{-1} \mu_0 (t) \left( \mu_0^{3\over 2} (t) t^{-1} \right)
 \end{eqnarray}
for some constant $C$. With this, we can estimate $i_1$ and $i_5$, as follows
$$
[i_j  ]_{0,\sigma , [t,t+1]} \leq C M^{-1} \mu_0 (t) \left( \mu_0^{3\over 2} (t) t^{-1} \right) , \quad {\mbox {for}} \quad j=1,5.
$$
Straightforward computation gives
$$
[i_j]_{0,\sigma , [t, t+1]} \leq C M^{-1} \mu_0 (t) t^{-\sigma} \, \| \la \|_\sharp \left( \mu_0^{3\over 2} (t) t^{-1} \right), \quad {\mbox {for}} \quad j=1,2,3.
$$
These estimates, together with the ones we obtained before, constitute the proof of \eqref{magic1}.

\medskip
\medskip
\noindent
{\it Proof of \eqref{magic2}.}
Let $\la_1$, $\la_2 \in B$. From \eqref{duha12} and \eqref{defT},
\begin{align*}
\hat T[\la_1] (t)  & - \hat T[\la_2 ] (t) =
\bar c   \int_{t_0 -1}^{t} {\bar \alpha (s) \over \sqrt{t -s}} \int_0^{M\over \sqrt{t-s}} e^{-\rho^2} \left[ {\rho \mu [\la_1]  \over \mu [\la_1 ] + \rho} - {\rho \mu [\la_2]  \over \mu [\la_2 ] + \rho} \right] \, d\rho ds .
\end{align*}
Observe that
\begin{align*}
\left| \left( \mu [\la_1 ] - \mu [\la_2] \right) (s) \right| &\leq A \mu_0 (s) \left| \Lambda_1 (s) - \Lambda_2 (s) \right| \\
&\leq A \mu_0 (s) \int_s^\infty |\la_1 -\la_2| (x) \, dx \leq A \mu_0^2 (s) \| \la_1 -\la_2 \|_\sharp
\end{align*}
for some constant $A$, whose value may change from one line to the other, and which is independent of $t$ and $t_0$.
A Taylor expansion gives
\begin{align*}
| \hat T[\la_1] (t)   - \hat T[\la_2 ] (t) | \leq
   \int_{t_0 -1}^{t} {|\bar \alpha (s)| \over \sqrt{t -s}} \int_0^{M\over \sqrt{t-s}} e^{-\rho^2}  {\rho    \over (\tilde \mu  + \rho )^2}  \left| \mu [\la_1] (s) - \mu [\la_2 ] (s) \right| \, d\rho ds
\end{align*}
for some $\tilde \mu$ between $\mu [\la_1]$ and $\mu [\la_2] $. Thus we get
$$
| \hat T[\la_1] (t)   - \hat T[\la_2 ] (t) | \leq A \mu_0^2 (t) M \, [\mu_0^{3\over 2} (t) t^{-1} ] \, \| \la_1 - \la_2 \|_\sharp,
$$
where $A$ is a constant independent of $t_0$ and $t$. Using again \eqref{sunday--}, we can show that
$$
\left| \hat T[\la_1]   - \hat T[\la_2 ] \right|_{0,\sigma ,[t,t+1]} \leq A \mu_0^2 (t) M \, [\mu_0^{3\over 2} (t) t^{-1} ] \, \| \la_1 - \la_2 \|_\sharp,
$$
where $A$ is a constant independent of $t_0$ and $t$. Choosing $t_0$ large enough, we can find $\bar c_2$ small enough so that \eqref{magic2} holds true.

\medskip
\medskip
{\it Step 2}. \ \ In the second part of the proof, we show the validity of \eqref{sunday-2}. For this purpose, we fix $\phi_1 $ and $\phi_2$ satisfying \eqref{phibound}, and we let $\la_j = \la [\phi_j]$ , $j=1,2$. If $\bar \la = \la_1 - \la_2$, then we see that $\bar \la$ solves
\begin{align*}
\bar \la &= T^{-1} \left( G_1 [\la_1 , \phi_1 ] - G_1 [\la_2 , \phi_2] \right) \\
&= T^{-1} \left( G_1 [\bar \la_1 , \phi_1 ] - G[\bar \la_1, \phi_2] \right) + T^{-1} \left( G_1 [\la_1 , \phi_2] - G[\la_2 , \phi_2] \right).
\end{align*}
Thus
\begin{align*}
\| \bar \la \|_\sharp &\leq C M^{-1} \left( \| G_1 [\bar \la_1 , \phi_1 ] - G[\bar \la_1, \phi_2] \|_\flat +
\| G_1 [\la_1 , \phi_2] - G[\la_2 , \phi_2] \|_\flat \right) \\
&\leq C M^{-1} \left( {\bf c} \| \phi_1 - \phi_2 \|_{\nu , a} + {\bf c} \| \la_1 - \la_2 \|_\sharp \right),
\end{align*}
where $C$ is the constant in \eqref{au1}, $M^2= t_0$, ${\bf c}$ are the constants defined respectively in \eqref{queen2} and \eqref{queen3}.
We now observe that the proof of Lemma \ref{l1} also gives that the constants ${\bf c}$ in \eqref{queen2} and \eqref{queen3} can be such that $CM^{-1} {\bf c}<1$. Thus the proof of \eqref{sunday-2} readily follows.

\medskip
This concludes the proof of the Proposition.
\end{proof}

\medskip

\begin{remark}\label{rmk2}
Recall that the function $ \psi = \bar \Psi [\psi_0]$ solution to Problem \eqref{equpsi}
depends smoothly on the initial condition $\psi_0$, provided $\psi_0$ belongs to a small
neighborhood of $0$ in the Banach space $L^\infty (\Omega)$ equipped with the norm defined in \eqref{gg2}, as observed in Remark \ref{rmk1}. This fact implies that also $\la = \la[\psi_0] $ solution to \eqref{na1} depends  on $\psi_0$. A closer look at the definitions of $\la = \la[\psi_0] $ gives that
$$
\| \la [\psi_0^{(1)} ] - \la [\psi_0^{(2)} ] \|_\sharp \lesssim
\| e^{b|y|} [ \psi_0^{(1)}  - \psi_0^{(1)} ] \|_{L^\infty (\R^3) } + \| e^{b|y|} [ \nabla \psi_0^{(1)}  - \nabla \psi_0^{(1)} ] \|_{L^\infty (\R^3) }.
$$
This fact will be useful in the final argument of finding $\phi$ solution to \eqref{equ3332}.

\end{remark}

\setcounter{equation}{0}
\section{Final argument: solving \eqref{equ3}} \label{final}

We are constructing a global unbounded solution to Problem \eqref{p}-\eqref{IC} of the form \eqref{solll}
$$
u= U_2 [\la ] (r,t) + \tilde \phi.
$$
The function $U_2$ is defined in \eqref{defU2}, while $\tilde \phi$ is given in \eqref{deftildephi}. The function $\psi$ which enters in the definition of $\tilde \phi$ solves the {\it outer problem} \eqref{equpsi}, and its properties are contained in Proposition \ref{propext} and \ref{propext1}.
The parameter $\la = \la(t)$ belongs to the space $X_\sharp$, \eqref{Xsharp}, and has been chosen to solve Equation \eqref{na1}. The properties of this $\la = \la (t)$ are collected in Proposition \ref{proplambdafinal}. What is left is to solve in $\phi$ the {\it inner problem} \eqref{equ3}. Thanks to the choice of $\la = \la (t)$, the orthogonality condition \eqref{ortio} is satisfied, so that we can use the result of Proposition \ref{prop0} to solve in $\phi$ Problem \eqref{equ3}.

In other words, we want to find $\phi$, with its $\| \phi \|_{\nu, a}$-bounded, solution to Problem \eqref{equ3}. The function $\psi = \Psi [\la [\phi] , \phi]$ solves \eqref{equpsi}, while $\la = \la[\phi]$ solves Equation \eqref{na1}.

\medskip
At this point, we fix $a$ in the definition od the $\| \star \|_{\nu , a}$ to be equal to $1$.
Proposition \ref{prop0} defines a linear operator  $\phi = {\mathcal T} (h)$, where $\phi $ is the solution to \eqref{p110nuovo} so that
$$
\| \phi \|_{\nu , 1 } \leq C_0  \, R^4 \, \| h \|_{\nu , 3}
$$
for some fixed constant $C_0$. We refer to \eqref{minchia} for  $\| h \|_{\nu , 2+a}$ and to \eqref{normphi} for  $\| \phi \|_{\nu , a }$, for $a=1$. Thus we can say that $\phi$ solves \eqref{equ3331}-\eqref{equ3332} if and only if $\phi$ is a fixed point for the Problem
\begin{equation}\label{vio}
\phi = {\mathcal T} \left( {\bf H} [\phi ] \right), \quad {\mbox {where}} \quad {\bf H} [\phi] = H ( \psi [\phi] , \la [\phi] , \phi ),
\end{equation}
and $H$ is defined in \eqref{defHH}.
Choose the number $R$ in the cut off function $\eta_R$, defined in \eqref{cut1} and appearing in the ansatz \eqref{deftildephi}, to be sufficiently large in terms of $t_0$, say  $R^{6} \mu_0^{1\over 2} (t_0  ) = 1$.  We claim that there exists a unique $\phi$ solution to \eqref{vio} in the set
$$
B_1 = \{ \phi \, : \, \| \phi \|_{\nu, 1} \leq L_1 \}
$$
for some $L_1 >0$, fixed.

\medskip
From \eqref{cut1} and \eqref{marte2}, we see that
$$
\left|  \mu_{0}^{\frac{5}2}  {\mathcal E}_{22} (\mu_{0} y,t) \right| \lesssim 
 \mu_0^{1\over 2} {\mu_0^{3\over 2} t^{-1} \over (1+ |y|^2)^2}, \quad 
  \left|  5 \,  { \mu_{0}^{\frac{1}2} \over (1+ \Lambda)^4 } \, w^4 ( {y \over (1+ \Lambda)^2 }  ) \psi(\mu_{0} y,t)
\right| \lesssim {\mu_0^2 (t) t^{-1} \over (1+ |y|^3)}.
$$
Furthermore,
$$
\left| B[\phi] (t) \right| \leq C R^2 \mu_0 \mu_0' {\mu_0^{3\over 2} t^{-1} \over (1+ |y|^{2+a} )}, \quad
\left| B^0 [\phi] (t) \right| \leq C \Lambda (t) { \mu_0^{3\over 2} t^{-1} \over (1+ |y|^{4+a} )}.
$$
In fact, one can prove that
$$
\| {\bf H} [\phi ] \|_{\nu, 2+a} \leq C_1 \, R^{-4}
$$
for some fixed number $C_1$, independent from $t$ and of $t_0$. This implies that, if $\phi \in B_1$, then ${\mathcal T} (\phi ) \in B_1 $
provided $L_1$ is chosen large. Furthermore, combining \eqref{marte0001new}, the result of Proposition \ref{propext1}, and the result of Proposition \ref{proplambdafinal}, we get the existence of a number ${\bf c} \in (0,1 )$, so that
$$
\| {\mathcal T} [\phi_1 ] - {\mathcal T} [\phi_2] \|_{\nu, a} \leq {\bf c} \| \phi_1 - \phi_2 \|_{\nu , a}
$$
for any $\phi_1$ and $\phi_2 \in B_1$. We apply Banach fixed point theorem to get the existence of a unique solution to \eqref{vio} with $\| \cdot \|_{\nu , a}$-bounded.

This concludes the proof of the existence of the solution to Problem \eqref{p}-\eqref{IC}, or equivalently Problem \eqref{p001}-\eqref{gam}, as predicted by Theorem \ref{teo1}.  \qed

\setcounter{equation}{0}
\section{Basic linear theory for the inner problem}\label{inner0}

Let  $R >0$ be a fixed large number. This section is devoted to construct a solution to the initial value problem
\be \label{cf1}
\phi_\tau  =
\Delta \phi +5w^4  \phi + h(y,\tau )  \inn B_{2R} \times (\tau_0, \infty ) , \quad \phi(y,\tau_0) = e_0 Z(y)  \inn B_{2R},
\ee
for any given function $h$ with $\|h\|_{\nu, 2+a} <+\infty$, not necessarily radial in the $y$ variable. We refer to \eqref{minchia} for the explicit definition of the
$\| \cdot \|_{\nu , 2+a}$-norm. The corresponding problem in dimension $n\geq 5$ has already been treated in \cite{CDM}, Section 7. We follow the same strategy in the procedure to construct the solution to \eqref{cf1}, but in dimension $3$ we get  a decay estimate for the solution different from the one valid for dimensions $n\geq 5$.

\medskip
We recall that the operator $L_0 (\phi ) = \Delta \phi + 5 w^4 \phi$ has an $4$ dimensional kernel generated by the bounded functions $Z_0$ defined in \eqref{Z0def} and also by
\begin{equation}\label{cf2}
Z_i (y) = {\partial w \over \partial y_i} , \quad i = 1, 2 , 3.
\end{equation}
In the class of radially symmetric functions, the only element in the kernel of $L_0$ is $Z_0$.
To describe our construction, we
 consider an orthonormal basis $\vartheta_m$, $m=0,1,\ldots,$ in  $L^2(S^2)$  of spherical harmonics, namely eigenfunctions of the problem
$$
\Delta_{S^2} \vartheta_m + \la_m \vartheta_m = 0 \inn S^2
$$
so that $0=\la_0 < \la_1 =\ldots= \la_3 = 2< \la_4 \le \ldots $.
Let $h (\cdot , \tau )\in L^2(B_{2R} )$, for any $\tau \in [\tau_0 , \infty)$. We decompose it into the form
$$
h(y,\tau) =     \sum_{j=0}^\infty h_j(r,\tau)\vartheta_j (y/r), \quad r=|y|, \quad h_j(r,\tau) = \int_{S^2} h (r\theta , \tau ) \vartheta_j(\theta) \, d\theta.
$$
In addition, we write $h = h^0 + h^1 + h^\perp$ where
$$
h^0 =  h_0(r,\tau), \quad h^1 = \sum_{j=1}^3 h_j(r,\tau)\vartheta_j, \quad h^\perp = \sum_{j=4}^\infty h_j(r,\tau)\vartheta_j.
$$
Observe that $h^1=h^\perp =0$ if $h$ is radially symmetric in the $y$ variable.
Consider also the analogous decomposition for $\phi$ into $\phi = \phi^0 + \phi^1 + \phi^\perp$. We build the solution $\phi$ of Problem \equ{cf1} by doing so separately for the pairs
$(\phi^0,h^0)$, $(\phi^1, h^1)$ and $(\phi^\perp, h^\perp)$.

\medskip
Our main result in this section is the following proposition.

\begin{proposition} \label{prop0cf}
Let $\nu,a$ be given positive numbers with $0<a <1$. Then, for all sufficiently large $R>0$ and any  $h=h(y,\tau)$  with  $\|h\|_{\nu, 2+a} <+\infty$
that satisfies for all $j=0, 1,\ldots, 3$
\be
 \int_{B_{2R}} h(y ,\tau)\, Z_{j} (y) \, dy\ =\ 0  \foral \tau\in (\tau_0, \infty)
\label{ortion}\ee
there exist  $\phi = \phi [h]$  and $e_0 = e_0 [h]$ which solve Problem $\equ{cf1}$. They define linear operators of $h$
that satisfy the estimates
\be
  |\phi(y,\tau) |  \ \lesssim   \  \tau^{-\nu} \Big [\, \frac {R^{4-a}} { 1+ |y|^3} \,   \|h^0\|_{\nu, 2+a}   +  \frac {R^{4-a}} { 1+ |y|^4} \,   \|h^1\|_{\nu, 2+a} +
  \frac {\|h \|_{\nu, 2+a}}{1+ |y|^a} \Big ] ,
\label{cta1cf}\ee
\be
  |\nabla_y \phi(y,\tau) |  \ \lesssim   \  \tau^{-\nu} \Big [\, \frac {R^{4-a}} { 1+ |y|^4} \,   \|h^0\|_{\nu, 2+a}   +  \frac {R^{4-a}} { 1+ |y|^5} \,   \|h^1\|_{\nu, 2+a} +
  \frac {\|h \|_{\nu, 2+a}}{1+ |y|^{a+1}} \Big ] ,
\label{cta2cf}\ee
and
\be
 | e_0[h]| \, \lesssim \,  \|h\|_{ \nu, 2+a}.
\label{cot2n}\ee
\end{proposition}

\medskip
Proposition \ref{prop0} is a direct consequence of Proposition \ref{prop0cf}. Indeed, if $h$ is radially symmetric in the $y$ variable, \eqref{ortion} is authomatically satified for $j=1, \ldots , 3$, and $h \equiv h^0$.

\medskip
The result contained in Proposition \ref{prop0cf} follows from next Proposition, which refers to the following problem
\be \label{p11}
\phi_\tau  =
\Delta \phi + 5w(y)^4 \phi + h(y,\tau )-c(\tau) Z \inn B_{2R} \times (\tau_0, \infty ), \quad \phi(y,\tau_0) = 0 \inn B_{2R}.
\ee

\medskip

\begin{proposition} \label{prop1}
Let $\nu,a $ be given positive numbers with $0< a < 1 $. Then, for all sufficiently large $R>0$ and any $h$ with  $\|h\|_{\nu, 2+a} <+\infty$ and satisfying the orthogonality conditions \eqref{ortio},
there exist  $\phi = \phi[h]$ and $c=  c[h]$ which solve Problem $\equ{p11}$,  and define linear operators of $h$. The function
$\phi[h]$ satisfies estimate $\equ{cta1cf}$, \eqref{cta2cf}
and for some $\Gamma >0$
\be
\left |  c(\tau) -  \int_{B_{2R}} hZ \right | \, \lesssim \, \tau^{-\nu}\, \Big [\,    R^{2-a} \left \| h -  Z \int_{B_{2R}} hZ\, \right \|_{ \nu, 2+a } +     e^{-\Gamma R} \|h\|_{ \nu, 2+\alpha } \Big ].
\label{cta2}\ee
\end{proposition}

\medskip
Assuming the validity of Proposition \ref{prop1}, we proceed with
\subsubsection*{\bf Proof of Proposition \ref{prop0cf}.} Let $\phi_1$ be the solution of Problem \eqref{p11} predicted by Proposition \ref{prop1}.
Let us write
\be
\phi(y,\tau) = \phi_1(y,\tau) +  e(\tau) Z(y).
\label{vv}\ee
for some $e\in C^1 \left( [\tau_0,\infty) \right)$. We find
$$
\partial_\tau \phi   =   \Delta \phi  + 5w^{4} \phi   + h(y,\tau ) +   \left[ \, e'(\tau)-\la_0 e(\tau) - c(\tau) \, \right] \, Z(y).
$$
We choose $e(\tau)$ to be the unique bounded solution of the equation
$$
e'(\tau)\,-\, \la_0 e(\tau)\, =\,  c(\tau), \quad \tau \in (\tau_0,\infty)
$$
which is explicitly given by
$$
e(\tau )  =   \int_\tau^\infty \exp({ \sqrt{\la_0} (\tau - s)})\, c(s)\, ds  \, .
$$
The function $e$ depends linearly on $h$. Besides, we clearly have from \equ{cta2},
$
|e(\tau)|  \ \lesssim \  \tau^{-\nu} \|h\|_{ \nu, 2+a}.
$
and thus, from the fact that $\phi_1$ satisfies estimates \equ{cta1cf}, \eqref{cta2cf},  so does $\phi$ given by \equ{vv}. Thus
$\phi$ satisfies Problem \equ{cf1} with initial condition $\phi(y,\tau_0) =  e(\tau_0 ) Z(y)$.
The proof is concluded. \qed

\medskip
\noindent
The rest of the Section is devoted to the

\medskip

\begin{proof}[{\bf Proof of Proposition \ref{prop1}}] The proof is divided in two steps. In the first step, we construct a solution to \eqref{p11}
which has value zero on the boundary $\partial B_{2R}$, at any time $\tau$, for a right hand side $h$ not necessarily satisfying the orthogonality conditions \eqref{ortion}. In the second step, we make use of this construction to solve \eqref{p11}, for a right hand side satisfying \eqref{ortion}, and to obtain estimates \eqref{cta1cf}, \eqref{cta2cf} and \eqref{cot2n}.

\medskip
\noindent
{\it {\bf Step $1$}}. \ \ We claim that
for all sufficiently large $R>0$ and any $H$ with  $\|H\|_{\nu, a} <+\infty$
there exists  $\phi = \phi (y, \tau )$ and $c= c(\tau )$ which solve Problem
\be\label{modo0}
 \phi_\tau = \Delta \phi   + 5w^4 \phi  + H(y,\tau) - c(\tau) Z (y)\inn B_{2R}\times (\tau_0,\infty)
 \ee
$$
\phi = 0 \onn  \pp B_{2R} \times (\tau_0,\infty)  ,\quad \phi(\cdot, \tau_0 ) = 0 \inn B_{2R}.
$$
The functions $\phi$ and $c$ are linear operators of $h$
and satisfy the estimates
$$
 (1+ |y| ) \, | \nabla \phi (y, \tau )| + |\phi(y,\tau)|
   \ \lesssim   \
   $$
   \be
   \tau^{-\nu} \Big [
     \frac {R^{4-a}   \| H^0\|_{\nu,2+a} }{ 1+ |y| }     +   \frac {R^{4-a} \| H^1\|_{\nu,2+a} }{ 1+ |y|^{2} }    +   R^2 \frac{ \| H \|_{\nu, a}}{1+|y|^a} \Big ] \label{cota11}\ee
and for some $\Gamma >0$
\be
\left |  c(\tau) -  \int_{B_{2R}} H Z \right | \, \lesssim \, \tau^{-\nu}\, \Big [\,  R^2 \left \| H -  Z \int_{B_{2R}} H Z\, \right \|_{ \nu, a} +     e^{-\Gamma R} \|H\|_{ \nu, a} \Big ].
\label{cota22}\ee

\medskip
We construct the solution $\phi$ mode by mode, considering first mode $0$, then modes $1,2,3$ and finally modes greater or equal to $4$. For each mode, we get the corresponding estimates.

\bigskip
{\it Construction at mode $0$.}
Consider
 Problem \equ{modo0} for a right hand side $H=H_0(r, \tau)$ radially symmetric.
 Let $\eta (s)$ be the smooth cut-off function in \eqref{defeta}, and consider
 $\eta_\ell (y) = \eta ( |y| -\ell )$, for
a large but fixed number $\ell $ independently of $R$. By standard parabolic theory, there exists a unique solution
 $\phi_*[\bar h_0]$ to
  \be\label{a1}
 \phi_\tau = \Delta \phi   + 5w(r)^{4}(1-\eta_\ell ) \phi  + \bar H_0 (y,\tau)  \inn B_{2R}\times (\tau_0,\infty)
\ee
 $$
\phi = 0 \quad {\mbox {on}} \quad  \partial B_{2R} \times (\tau_0,\infty)  ,\quad \phi(\cdot, \tau_0) = 0 \inn B_{2R},
$$
where
$$
\bar H_0 = H_0 -  c_0(\tau) Z , \quad      c_0(\tau) =  \int_{B_{2R}} H_0 (y, \tau ) Z (y) \, dy .
$$
The function $\phi_* [\bar h_0]$ is radial and satisfies the bound
$$
\left| \phi_* [\bar H_0 ] \right| \lesssim  \tau^{-\nu} \, R^{2-a} \,   \| H \|_{ \nu , a} .
$$
This can be proven with the use of a special super solution, arguing as in Lemma 7.3 in \cite{CDM}.
Setting $\phi =  \phi_*[\bar H_0] + \tilde \phi $ and $c(\tau) = c_0(\tau) + \tilde c(\tau)$,
Problem \equ{modo0} gets reduced to
 \be\label{modo01}
 \tilde \phi_\tau = \Delta \tilde \phi   + 5w(r)^{4} \tilde \phi  + \tilde H_0(r,\tau) - \tilde c(\tau) Z\inn B_{2R}\times (\tau_0,\infty)
 \ee
$$
\tilde \phi = 0 \quad {\mbox {on}} \quad   \partial B_{2R} \times (\tau_0,\infty)  ,\quad \tilde \phi(\cdot, \tau_0) = 0 \inn B_{2R} .
$$
where
$
\tilde H_0 = 5 w^{4}\eta_\ell   \phi_*[\bar H_0 ] .
$
Observe that $\tilde H_0$  is radial, it is compactly supported and with size controlled by that of $\bar H_0$. In particular we have that for any $m>0$,
\be \label{tete}
|  \tilde H_0 (r, \tau) |  \lesssim\  { \tau^{-\nu} \over 1+r^m} \, \left[ \sup_{\tau > \tau_0} \tau^{\nu} \| \phi_*[\bar H_0] (\cdot , \tau)  \|_{L^\infty}  \right] \lesssim  { \tau^{-\nu} \over 1+r^m} \, R^{2-a} \,   \| H \|_{ \nu , a}.
\ee
We shall next solve Problem \equ{modo01} under the additional orthogonality constraint
\be \label{modo02}
\int_{B_{2R}}  \tilde \phi (\cdot , \tau )\, Z  \ =\ 0 \foral \tau\in (\tau_0, \infty ) .
\ee
Problem
\equ{modo01}-\equ{modo02} is equivalent to solving just \equ{modo01} for $\tilde c$ given by the explicit linear functional
$\tilde c := \tilde c[\tilde \phi, \tilde H_0] $  determined  by the relation
\be \label{cc}
\tilde c(\tau) \int_{B_{2R}} Z^2    =  \int_{B_{2R}} \tilde  H_0 (\cdot, \tau)  Z     +     \int_{\partial B_{2R}} \partial_r \tilde \phi(\cdot, \tau)  Z .
\ee
If the function $\tilde c= \tilde c  (\tau)$ defined by \eqref{cc}
were independent of $\phi$, standard linear parabolic theory would give the existence of a unique solution. On the other hand, a close look to \eqref{cc} shows that
the dependence of  $\tilde c= \tilde c(\tau )$  on $\phi$ is  small in an $L^\infty$-$C^{1+\alpha, \frac {1+\alpha} 2} $ setting, since $Z( R) = O(e^{-\Gamma R})$ for some $\Gamma>0$.  A contraction argument applies to yield existence of a unique solution to \equ{modo01}-\equ{modo02} defined at all times. To get the estimates, we assume smoothness of the data  so that integrations by parts and differentiations can be carried over, and then arguing by approximations.
Testing \equ{modo01}-\equ{modo02} against $\tilde \phi$ and integrating in space, we obtain the relation
$$
\partial_\tau \int_{B_{2R}} \tilde \phi^2     + Q(\tilde \phi,\tilde \phi)  =  \int_{B_{2R}} g\tilde \phi, \quad g= \tilde H_0 - \tilde c(\tau) Z_0 ,
$$
where $Q$ is the quadratic form defined by
\be\label{Q}
Q(\phi,\phi) := \int \left[  |\nabla \phi|^2 - 5w^4 |\phi|^2 \right]  .
\ee
Since dimension is $3$, there exists $ \beta >0$ such that, for any $\phi$ with $ \int \phi Z = 0$, the following inequality holds
$$
Q(\phi , \phi ) \geq {\beta \over R^2 } \int \phi^2.
$$
The proof of this inequality is a slight modification of the proof for the corresponding inequality in dimensions $n\geq 5$ that can be found in Lemma 7.2 \cite{CDM}, considering that $\int_{B_R} Z_0^2 = O(R)$, as $R\to \infty$, when dimension is $3$.
Thus we have, for some $\beta' >0$,
\be\label{en}
\partial_\tau \int_{B_{2R}} \tilde \phi^2     + \frac{\beta' }{R^2} \int_{B_{2R}} \tilde \phi^2  \lesssim   R^2\int_{B_{2R}} g^2 .
\ee
We observe that from \equ{cc} and \equ{tete} for $m=0$  we get that
$$
|\tilde c(\tau)| \le \tau^{-\nu } K , \quad K:=
\left[ \sup_{\tau > \tau_0} \tau^{\nu} \| \phi_*[\bar H_0] (\cdot , \tau)  \|_{L^\infty}  \right] + e^{-\Gamma R} \left[ \sup_{\tau > \tau_0} \tau^{\nu} \| \nabla \phi_*[\bar H_0] (\cdot , \tau)  \|_{L^\infty}  \right].
$$
Besides,  using again estimate \equ{tete} for a sufficiently large $m$, we get
$$
\int_{B_{2R}} g^2 \ \lesssim\     \tau^{-2\nu } K^2  .
$$
Using that $\tilde\phi(\cdot, \tau_0) =0$ and Gronwall's inequality, we readily get from \equ{en} the $L^2$-estimate
\be\label{K}
\|\tilde \phi (\cdot , \tau ) \|_{L^2(B_{2R})} \  \lesssim \ \tau^{-\nu } R^2  K,
\ee
for all $\tau > \tau_0$. Now, using standard parabolic estimates in the equation satisfied by $\tilde \phi$ we obtain then that on any large fixed radius $\ell>0$,
$$
\|\tilde \phi (\cdot , \tau ) \|_{L^\infty(B_M)} \  \lesssim \ \tau^{-\nu } R^2 \,  K \foral \tau > \tau_0.
$$
Since the right hand side has a fast decay at infinity and taking into account that we are in dimension $3$,  outside $B_\ell $ we can dominate the solution by a barrier of the order $\tau^{-\nu} |y|^{-1} $.
As a conclusion, also using local parabolic estimates for the gradient, we find that
\be \label{i1}
(1+ |y| ) \, |\nabla_y \tilde \phi (y , \tau )| + |\tilde \phi (y , \tau )| \ \lesssim  \tau^{-\nu } \frac{ R^2} {1+ |y|} \left[ \sup_{\tau > \tau_0} \tau^{\nu} \| \phi_*[\bar H_0] (\cdot , \tau)  \|_{L^\infty}  \right]    .
\ee
It clearly follows from this estimate and inequality \equ{tete}  that
the function
\be\label{p0} \phi_0 [h_0] :=  \tilde \phi + \phi_*[\bar H_0] \ee
solves Problem \equ{modo0}  for $H=H_0$ and satisfies
$$
(1+ |y| ) \, |\nabla_y \tilde \phi_0 (y , \tau )|+  |\phi_0(y,\tau)|
   \ \lesssim   \    \tau^{-\nu} \frac {R^{4-a} } { 1+ |y| }   \| H\|_{\nu, a}  \quad
$$
Finally, from \equ{cc} we see that
we have that  $$ c(\tau) =  \int_{B_{2R}}  H Z   +  \int_{B_{2R}}  5w^4 \eta_\ell  \phi_*[\bar H_0 ] \,Z   + O  (e^{-\Gamma R} ) \|H\|_{\nu , a}. $$
From here we find the validity of estimate
$$
\left |  c(\tau) -  \int_{B_{2R}} H_0Z \right | \, \lesssim \, \tau^{-\nu}\, \Big [\,  R^2  \left \| H_0 -  Z \int_{B_{2R}} H_0Z\, \right \|_{  \nu, a} +     e^{-\Gamma R} \|H_0\|_{\nu, a} \Big ].
$$
Hence estimates \equ{cota11} and \equ{cota22} hold. The construction of the solution at mode 0 is concluded.

\medskip
\bigskip
{\it Construction at modes $1$ to $3$.}
Here we consider the case $H=H^1$ where
$
H^1(y,\tau)  =  \sum_{j=1}^3 H_j(r,\tau) \vartheta_j.
$
The function
\be\label{p1}
\phi^1[H^1] :=  \sum_{j=1}^n\phi_j(r,\tau) \vartheta_j.
\ee
solves the initial-boundary value problem
\be\label{modo1}
 \phi_\tau = \Delta \phi   + 5 w^4 \phi  + H^1(y,\tau)  \inn B_{2R}\times (\tau_0,\infty)
 \ee
$$
\phi = 0 \onn  \pp B_{2R} \times (\tau_0,\infty)  ,\quad \phi(\cdot, \tau_0 ) = 0 \inn B_{2R},
$$
if the functions $\phi_j(r,\tau)$ solves
\be\label{modo11}
 \pp_\tau \phi_j  = \mathcal L_1 [\phi_j]  + H_j(r,\tau)   \inn (0,2R)\times (\tau_0,\infty)
 \ee
$$
\partial_r \phi_j(0,\tau) = 0 = \phi_j(R,\tau)  \foral \tau\in (\tau_0,\infty)  ,\quad \phi_j(r, \tau_0 ) = 0 \foral r\in (0,R),
$$
where
\be\label{mathcalL1}
\mathcal L_1 [\phi_j] := \pp_{rr}\phi_j  + 2\frac{ \pp_{r}\phi_j } r -2\frac{\phi_j } {r^2}  + 5w^4 \phi_j.
\ee
Let us consider
the solution of the stationary problem
$
  \mathcal L_1 [\phi]  +  (1+r)^{-a} =0
$
given by the variation of parameters formula
$$
\bar \phi(r) =  Z(r) \int_r^{2R} \frac 1{\rho^2 Z(\rho)^2} \int_0^\rho (1+s)^{-a} Z(s)s^2\, ds
$$
where $Z(r) = w_r(r).$ Since $w_r(r)\sim r^{-2}$ for large $r$, we find the estimate
$
|\bar \phi(r) | \ \lesssim \ \frac { R^{4-a}   } {1 + r^{2}}  .
$
Then, provided that $\tau_0$ was  chosen sufficiently large, the function  $2\|H_j\|_{\nu , a} \tau^{-\nu} \bar \phi (r) $ is a positive super-solution of Problem  \equ{modo11}
 and thus  we find
$
|\phi_j(r,\tau) |   \ \lesssim \ \tau^{-\nu}  \frac { R^{4 -a}     } {1 + r^2}\| H_j\|_{\nu , a} .
$
Hence $\phi^1[H^1]$ given by \equ{p1} satisfies
$$
|\phi^1[H^1] (y,\tau) |   \ \lesssim \ \frac { R^{4-a}  } {1 + |y|^{2}}\| H^1\|_{\nu  , a} .
$$
A corresponding estimate for the gradient follows.

\medskip
\bigskip
{\it Construction at higher modes.}
We consider now the case  of higher modes,
\be\label{modok}
\phi_\tau = \Delta \phi   + 5w^4 \phi  + H^\perp  \inn B_{2R}\times (\tau_0,\infty)
\ee
$$
\phi = 0 \quad {\mbox {on}} \quad  \partial B_{2R} \times (\tau_0,\infty)  ,\quad \phi(\cdot, \tau_0 ) = 0 \inn B_{2R},
$$
where
$H= H^\perp = \sum_{j=4}^\infty H_j(r) \Theta_j$
whose solution has the form
$\phi^\perp = \sum_{j=4}^\infty \phi_j(r, \tau) \Theta_j$.
Given the quadratic form in \eqref{Q},  for $\phi^\perp\in H_0^1(B_{2R})$
\be\label{co}
\int_{B_{2R}} \frac {|\phi^\perp|^2} {r^2} \lesssim Q(\phi^\perp ,\phi^\perp ).
\ee
The proof of this fact is elementary. The interested reader can find it in \cite{CDM}.
Let $\phi_*[H^\perp] $ be the solution to
$$
\phi_\tau = \Delta \phi   + 5w(r)^{4}(1-\eta_\ell )\phi  + \bar H^\perp (y,\tau)  \inn B_{2R}\times (\tau_0,\infty)
$$
$$
\phi = 0 \quad {\mbox {on}} \quad  \partial B_{2R} \times (\tau_0,\infty)  ,\quad \phi(\cdot, 0) = 0 \inn B_{2R},
$$
where
$
\bar H^\perp = H^\perp -  c^\perp Z $, and $c^\perp = \int_{B_{2R}} H^\perp Z $.
By writing $\phi =\phi_*[H^\perp] +\tilde \phi$, Problem \equ{modok}   reduces to solving
\be
 \tilde\phi_\tau = \Delta \tilde\phi   + 5w(y)^4 \tilde \phi  + \tilde H  \inn B_{2R}\times (\tau_0,\infty)
 \ee
$$
\tilde \phi = 0 \quad {\mbox {on}} \quad  \partial B_{2R} \times (\tau_0,\infty)  ,\quad \tilde \phi(\cdot, \tau_0 ) = 0 \inn B_{2R},
$$
where
$
\tilde H = 5w(y)^{4}\eta_\ell  \phi_*[H^\perp] ,
$
for a sufficiently large $\ell $.
Arguing as in \equ{en} we now get
\be\label{en1}
\partial_\tau \int_{B_{2R}} \tilde \phi^2     + c\int_{B_{2R}} \frac {|\tilde \phi|^2}{|y|^2}   \lesssim   \int_{B_{2R}} |y|^2 |\tilde H|^2  .
\ee
Similarly to \equ{K} we get

\be\label{K1}
\|\, |y|^{-1} \tilde \phi (\cdot , \tau ) \|_{L^2(B_{2R})} \  \lesssim \tau^{-\nu}  R^{2-a}  \| H  \|_{\nu , a}
\ee
From elliptic estimates we then get that
$$
\|\tilde \phi (\cdot , \tau ) \|_{L^\infty(B_{2R})} \  \lesssim \ \tau^{-\nu }  R^{2-a}  \| H^\perp  \|_{\nu , a} .  \foral \tau > \tau_0,
$$
so that with the aid of a barrier we obtain

$$
 |\tilde \phi (y , \tau )|   \ \lesssim\    \tau^{-\nu }  R^{2-a}  \| H^\perp  \|_{\nu , a} \,(1+ |y|)^{-1} .
 $$
  It follows that the function
\be \label{pperp}
 \phi^\perp[H^\perp]   : =  \tilde \phi + \phi_*[ H^\perp ] \ee  satisfies
$$
|\phi^\perp [H^\perp] (y , \tau )|    \ \lesssim\   \tau^{-\nu } \, R^2 \, \left [ \,(1+ |y|)^{-1}    +  (1+|y|)^{-a}  \right ]\, \| H^\perp  \|_{\nu , a}\inn B_{2R}.
$$
Similar estimates for the gradient follow.
Conclusion:  let
$$
\phi[h] := \phi^0[h^0] + \phi^1[h^1] + \phi^\perp[h^\perp]
$$
for the functions defined in \equ{p0},  \eqref{p1}, \equ{pperp}. By construction, $\phi[h]$
solves Equation \equ{modo0}. It defines a linear operator of $h$ and satisfies \eqref{cota11}.
The proof of {\it Step 1} is concluded.

\medskip
\noindent
{\it {\bf Step $2$}.} \ \ To complete the proof of Proposition \ref{prop1}, we decompose the right hand side $h$ in \eqref{p11} in modes,
$h= h^0+ h^1 +h^\perp$ as before,
and define separately associated solutions of \equ{p11} in a decomposition $\phi= \phi^0 + \phi^1 +\phi^\perp$.

\medskip

\medskip
\bigskip
{\it Construction at mode $0$.}
For a bounded radial $h=h(|y|)$  defined in $B_{2R}$ with $\int_{B_{2R}} hZ_0=0$, let $\tilde h$ designate the extension of $h$ as zero outside $B_{2R}$.
The equation
$$
\Delta H + 5w^4 (y) H + \tilde h(|y|) = 0 \inn \R^3 , \quad H(y) \to  0 \quad {\mbox {as}} \quad |y|\to \infty
$$
has a  solution $H =: L_0^{-1}[h] $ represented by the variation of parameters formula
\be  H(r) =  \tilde Z(r)\int_r^{\infty}  \tilde h(s)\,Z_0 (s)\, s^{2}\,ds  +   Z_0(r) \int_r^{\infty}   \tilde h(s)\,\tilde Z(s)\, s^{2}\,ds
\label{ll0}\ee
where $\tilde Z(r)$ is a suitable second radial solution of $ L_0[\tilde Z] = 0 $, linearly independent  with $Z_0$.
Mode $0$ function $h_0 = h_0(|y|,\tau)$ is defined in $B_{2R}$, and satisfies $\|h_0\|_{\nu, 2+a} <+\infty $  and $\int_{B_{2R}} h_0 Z_0=0$ for all $\tau$. Then
$H_0:= L_0^{-1}[h_0(\cdot, \tau) ]$ satisfies the estimate
$$
|H_0 (r, \tau ) |\lesssim {\tau^{-\nu} \over (1+ r)^a} \, \| h_0\|_{ \nu , 2+a} .
$$
Let  $\Phi_0[h_0 ]$ be the radial solution  in $B_{3R}$ to
\be\label{modo00}
 \Phi_\tau = \Delta \Phi   + 5w^4 (y)  \Phi  + H_0(|y|,\tau) - c_0(\tau) Z\inn B_{3R}\times (\tau_0,\infty)
 \ee
$$
\Phi = 0 \quad {\mbox {on}} \quad  \partial B_{3R} \times (\tau_0,\infty)  ,\quad \Phi(\cdot, \tau_0 ) = 0 \inn B_{3R},
$$
that we discussed in Step 1.  $\Phi_0[h_0 ]$ defines a linear operator of $h_0$ and satisfies the estimates
\be
  |\Phi_0(y,\tau) |  \ \lesssim   \  \frac{\tau^{-\nu} R^{4-a} } { (1+ |y|) }    \| H_0 \|_{ \nu , a} , \quad
\label{cota111}\ee
where for some $\Gamma >0$
\be\left |  c_0(\tau) -  \int_{B_{2R}} H_0Z \right | \, \lesssim \, \tau^{-\nu}\, \Big [\,  R^{2} \left \| H_0 -  Z \int_{B_{2R}} H_0Z \, \right \|_{ \nu, a} +     e^{-\Gamma R} \|H_0\|_{ \nu, a} \Big ].
\label{cota222}\ee
Since $L_0[Z ] = \la_0 Z$ then
$$
\la_0 \int_{B_{2R}} H_0Z = \int_{B_{2R}} H_0 L_0[ Z]  =    \int_{B_{2R}} L_0[H_0]\, Z  + \int_{\partial B_{2R}}  (Z \partial_\nu H_0 -  H_0 \partial_\nu Z),
$$
and hence
$$
 \int_{B_{2R}} H_0Z    =   \la_0^{-1} \int_{B_{2R}} h_0\, Z  +  O(e^{-\Gamma R}) \tau^{-\nu} \| h_0\|_{\nu, 2+ a} .
$$
Also, from the definition of the operator $L_0^{-1}$ we see that $  Z  =  \la_0 L_0^{-1}[Z] $.
Thus
$$
\left \| H_0 -  Z \int_{B_{2R}} H_0Z\, \right \|_{ \nu,a} = \left \|L_0^{-1}\Big [ h_0 -  \la_0 Z \int_{B_{2R}} H_0Z  \Big ] \, \right \|_{ \nu, a}
\ \lesssim \
\left \| h_0 -  Z \int_{B_{2R}} h_0Z\, \right \|_{ \nu, 2+a } +     e^{-\Gamma R} \|h_0\|_{ \nu, 2+a}.
$$
Next, we discuss estimates on the first and second derivatives of $\Phi_0$.
Let us fix now a  vector $e$ with $|e|=1$, a large number $\rho>0$ with $\rho \le 2R$ and a number $\tau_1 \ge \tau_0$. Consider the change of variables
$$
\Phi_\rho (z,t) := \Phi_0 ( \rho e +  \rho z , \tau_1 +  \rho^2 t ), \quad H_\rho (z,t) := \rho^2 [ H_0 ( \rho e +  \rho z , \tau_1 +  \rho^2 t ) - c_0( \tau_1 +  \rho^2 t) Z(\rho e +
\rho z)\, ].
$$

Then $\Phi_\rho( z, t) $ satisfies an equation of the form
$$
 \partial_t \Phi_\rho  = \Delta_z \Phi_\rho   +  B_\rho(z,t)  \Phi_\rho   +     H_\rho  (z,t)  \inn B_1(0) \times (0,2). 
$$
where  $B_\rho = O(\rho^{-2})$ uniformly in $B_2(0) \times (0,\infty)$.
Standard parabolic estimates yield that for any $0<\alpha < 1$
$$
 \|\nabla_z \Phi_\rho \|_{L^\infty  (B_{\frac 12 }(0) \times (1,2))  } \lesssim  \|\Phi_\rho \|_{L^\infty ( B_{1 }(0) \times (0,2)) } + \| H_\rho \|_{L^\infty  (B_{1 }(0) \times (0,2)) }  .
$$
Moreover
$$
\| H_\rho \|_{L^\infty  (B_{1 }(0) \times (0,2)) }\ \lesssim\  \rho^{2-a}\tau_1^{-\nu}\| H_0 \|_{\nu , a} ,\quad
\|\Phi_\rho \|_{L^\infty  (B_{1 }(0) \times (0,2)) } \lesssim  \tau_1^{-1} K(\rho)
$$
where
\be\label{Knew}
K(\rho) =
\frac{ R^{2-a}} {  1+\rho}  R^{2}  \| h^0\|_{\nu,2+a}
\ee
This yields in particular that
$$
\rho |\nabla_y \Phi(\rho e , \tau_1 + \rho^2 )|  \ = \  |\nabla\tilde \phi(0, 1)| \lesssim   \tau_1^{-\nu} K(\rho) .
$$
Hence if we choose $\tau_0 \ge R^2$,
we get that for any $\tau > 2\tau_0$ and $|y|\le  3R$
\be \label{coco}
(1+ |y|)\, |\nabla_y \Phi(y , \tau  )|  \ \lesssim  \   \tau^{-\nu} K(|y|)
\ee
We obtain that these bounds are as well valid for $\tau < 2\tau_0$   by the use of similar parabolic estimates up to the initial time (with condition 0).

\medskip
Now, we observe that the function $H_0$ is of class $C^1$ in the variable $y$ and  $\|\nabla_y H_0\|_{ \nu , 1+a} \le  \|h^0\|_{\nu, 2+a} $. It follows that we have the estimate
$$
(1+ |y|^2)\, | D^2_y \Phi(y,\tau) |\ \lesssim \  \tau^{-\nu} K(|y|)
$$
for all $\tau> \tau_0$, $|y|\le  2R  $.
where $K$ is the function in \equ{Knew}.
The proof follows simply by differentiating the equation satisfied by $\Phi$, rescaling in the same way we did to get the gradient estimate, and apply the bound already proven for
 $\nabla_y \Phi$. Thus we have in $B_{2R}$
$$
 (1+|y|^2) |D^2\Phi(y,\tau)| +  (1+|y|) | \nabla \Phi(y,\tau)| \  + \  |\Phi(y,\tau)|
 \, \lesssim\,     \tau^{-\nu} \|h^0\|_{\nu, 2+a } \frac{R^{4-a}} {1+ |y| }    .
$$
This yields in particular
$$
| L_0 [\Phi](\cdot, \tau ) |\ \lesssim\  \tau^{-\nu} \|h^0\|_{\nu, 2+a } \frac{R^{4-a}} {1+ |y|^3 } \inn B_{2R}
$$
We define
$$
\phi^0[h_0] := L_0 [\Phi]\, \Big |_{B_{2R}}.
$$
Then  $\phi^0[h_0]$ solves Problem \equ{p11}
with
\be\label{ccnew}
c(\tau) := \la_0 c_0(\tau).
\ee
$\phi^0[h_0]$ satisfies the estimate
\be\label{pp0}
| \phi^0[h_0] (y, \tau) | \ \lesssim \ \tau^{-\nu} \|h_0\|_{\nu, 2+a } \frac{R^{4-a}} {1+ |y|^3 }  \inn B_{2R}.
\ee
and from \equ{cota222},
estimate \equ{cta2} holds too.

\medskip

\medskip
\bigskip
{\it Construction for modes $1$ to $3$.}
We consider now
$h^1(y,\tau) =  \sum_{j=1}^3 h_j (r , \tau ) \vartheta_j $  with  $\|h^1\|_{\nu,2+a } <+\infty$
that satisfies for all $i=1,\ldots, 3$
$
\int_{B_{2R}} h^1 Z_i = 0 \foral \tau\in (\tau_0, \infty).
$
We will show that
there is  a solution $$ \phi^1[h^1]  = \sum_{j=1}^3 \phi_j (r, \tau ) \vartheta_j({y \over r}) $$
to Problem $\equ{p11}$ for $h= h^1$, which define a linear operator of $h^1$
and satisfies the estimate
\be
  |\phi^1(y,\tau)|  \ \lesssim   \   \frac {R^4}{ 1+|y|^{4} }R^{-a}     \|h\|_{\nu, 2+a} .
\label{cta4}\ee
Let us fix $1\le j\le 3$.
For a function  $h = h_j(r  ) \vartheta_j  ({y \over r})  $ defined in $B_{2R}$, we
let $H= L_0^{-1}[h] := H_j(r) \vartheta_j ({y \over r}) $ be the solution of the equation
$$
\Delta H + pU^{p-1}H + \ttt h_j \vartheta_j  = 0 \inn \R^n , \quad H(y) \to  0 \ass |y|\to \infty
$$
where $\ttt h_j$ designates the extension of $h_j$ as zero outside $B_{2R}$, represented by the variation of parameters formula
$$
H_j(r) =  w_r(r) \int_r^{2R} \frac 1{\rho^{n-1} w_r(\rho)^2} \int_\rho^\infty \ttt h_j(s)\, w_r(s)s^{n-1}\, ds
$$
If we consider a function $h^j = h_j(r,\tau)\vartheta_j$ defined in $B_{2R}$ with $\|h^j\|_{2+a, \nu} <+\infty $  and $\int_{B_{2R}} h^jZ_{j}=0$ for all $\tau$, then
$H_j= L_0^{-1}[h^j(\cdot, \tau) ]$ satisfies the estimate
$
\| H_j\|_{\nu, a } \lesssim \| h_j \|_{\nu, 2+ a} .
$
Let us consider the boundary value problem in $B_{3R}$
\be\label{modo001}
 \Phi_\tau = \Delta \Phi   + pU(y)^{p-1} \Phi  + H_j(r)\vartheta_j (y) \inn B_{3R}\times (\tau_0,\infty)
 \ee
$$
\Phi = 0 \onn  \pp B_{3R} \times (\tau_0,\infty)  ,\quad \Phi(\cdot, \tau_0 ) = 0 \inn B_{3R}.
$$
As consequence of Step 1, we find a solution $\Phi_j[h]$ to this problem, which defines a linear operator of $h_j$ and satisfies the estimates
\be
  |\Phi_j(y,\tau) |  \ \lesssim   \  \frac{\tau^{-\nu} R^{3-a}} { 1+ |y|^{2}}  R^{1}  \| h_j\|_{\nu, 2+ a} , \quad
\label{cota112}\ee
Arguing by scaling and parabolic estimates, we find as in the construction for mode 0,
$$
| L[\Phi_j](\cdot, \tau ) |\ \lesssim\  \tau^{-\nu} \|h\|_{\nu, 2+a} \frac{R^{4-a}} {1+ |y|^4} \inn B_{2R}.
$$
We define
$
\phi_j[h_j] := L[\Phi_j]\, \Big |_{B_{2R}}.
$
and
$
\phi^1[h^1] := \sum_{j=1}^3 \phi_j[h_j]\vartheta_j .
$
This function solves \equ{p11} for $h= h^1$ and satisfies
\be\label{pp1}
| \phi^1[h^1] (y, \tau) | \ \lesssim \ \tau^{-\nu} \|h_j\|_{2+a,\nu} \frac{R^{4-a}} {1+ |y|^4 }  \inn B_{2R}.
\ee

\medskip

\medskip
\bigskip
{\it Construction at higher modes.}
In order to deal with the higher modes, for
$h= h^\perp = \sum_{j=4}^\infty h_j(r) \Theta_j$
we let $\phi^\perp[ h^\perp]$ be just the unique solution of the problem
\be\label{modoknew}
 \phi_\tau = \Delta \phi   + pU(y)^{p-1} \phi  + h^\perp  \inn B_{2R}\times (\tau_0,\infty)
 \ee
$$
\phi = 0 \quad {\mbox {on}} \quad  \partial B_{2R} \times (\tau_0,\infty)  ,\quad \phi(\cdot, \tau_0 ) = 0 \inn B_{2R},
$$
which 
is estimated as
\be\label{pport}
|\phi^\perp [h^\perp] (y , \tau )|    \ \lesssim\   \tau^{-\nu }\frac {\| h^\perp  \|_{\nu , 2+a}}{ 1+|y|^{a}}\, \inn B_{2R}.
\ee
We just let
$$
\phi [h]:= \phi^0[h^0] + \phi^1[h^1]  + \phi^\perp [h^\perp]
$$
be the functions constructed above. According to estimates \equ{pp0} and \equ{pport} we find that
this function solves Problem  \equ{p11} for $c(\tau)$ given by \equ{cc}, with bounds \eqref{cta1cf}, \eqref{cta2cf}, \equ{cta2} as required. The proof is concluded.

\end{proof}

\setcounter{equation}{0}
\section{Non radially symmetric case}\label{nr}

In this section, we discuss the existence of solutions for Problem \eqref{p} when  the initial condition is  not radially symmetric, and we discuss the co-dimension $1$ stability. Let
$\bar v_0$ be a positive, uniformly bounded smooth function, not radially symmetric and define
\be \label{ICnew}
v_0 (x) = {\bar v_0 (x) \over |x|^\kappa }
, \quad {\mbox {with}} \quad \kappa > \max \{{\gamma + 3 \over 2} , \gamma \}.
\ee
We construct a solution to the initial value Problem
\begin{equation}
\label{pnew}
\left\{\begin{array}{l} u_t = \Delta u + u^5 , \quad {\mbox {in}} \quad \R^3 \times (t_0 ,\infty) , \\
\quad  u (x, t_0 ) = u_0 (|x| ) + v_0 (x)
\end{array}
\right.
\end{equation}
where $u_0$ is radial and satisfies the decay condition \eqref{IC}, while  $v_0$ is a non radial function of the form \eqref{ICnew}.

Since the strategy of the proof is similar to the one already performed in details for $\bar v_0 (x) \equiv 0$, we shall indicate the changes in the argument that are required when the initial condition is not radially symmetric.

We start with a slightly different first approximation.
Let $p = p (t) :[t_0 , \infty ) \to  \R^3$ be a smooth function so that
$$
 p(t_0) = {\bf 0} , \quad p(t) = \int_{t_0}^t P(s) \, ds, \quad {\mbox {where}} \quad P \quad {\mbox {satisfies}}
 $$
 \begin{equation}\label{boundP}
 \| P \|_\diamondsuit := \sup_{t > t_0 } \,  \mu_0 (t)^{-{1 \over 2}} t^{\kappa -1} \, \left[ \| P(s)  \|_{\infty , [t, t+1]}   + [P ]_{0, \sigma , [t,t+1]} \right] \,  \leq \ell,
 \end{equation}
 with $\sigma$ the number fixed in \eqref{lambdabound}, and $\ell $ a positive fixed number. Observe that, under these assumptions, and the bound on $\kappa $ in \eqref{ICnew}, we have
 $
{|p(t)| \over \mu_0 (t) } \to 0$ as $t \to \infty$.
 Define
 \begin{equation}\label{maggio1}
 U [\la , P] (x,t) = \hat U_2 (x,t)  + U_3  (x,t), \quad \hat U_2 (x,t) := U_2 (|x-p (t) | , t),
 \end{equation}
 where $U_2$ is given by \eqref{defU2} and
 \begin{equation}
 \label{maggio2}
 U_3 (x,t) = \left( 1- \eta ({|x| \over t} ) \right) v_0 (x).
 \end{equation}
If we call ${\mathcal E} [\la, P] (x,t) := \Delta U + U^5 - U_t$, we can write
\begin{align*}
{\mathcal E} [\la , P] (x,t) &= {\mathcal E}_2 [\la] (|x-p|, t) - \nabla U_2 (|x-p| , t) \, \cdot \, \dot p (t) \\
&+ \underbrace{\Delta U_3 - {\partial U_3 \over \partial t} + (\hat U_2 + U_3 )^5 - (\hat U_2)^5}_{:= {\mathcal E}_3}.
\end{align*}
Define
\begin{align}\label{maggioe}
\bar {\mathcal E} (x,t) &= {\mathcal E}_{21} (|x-p|,t) \\
&+  \left(1-\eta_R ({|x| \over R \mu_0 } ) \right) \left[ {\mathcal E}_{22} (|x-p|,t)    - \nabla U_2 (|x-p| , t) \, \cdot \, \dot p (t) \right] \nonumber
\end{align}
where $\eta_R$ is defined in \eqref{cut1}.
We have that
\begin{equation}\label{maggio4}
\left| \bar {\mathcal E} (x,t) \right| \leq C \mu_0^{1\over 2} t^{-{3\over 2}} h_0 ({|x| \over \sqrt{t} } ), \quad
\left| {\mathcal E}_3 (x,t) \right| \leq C t^{-\kappa + {1\over 2} } h_0 ({|x| \over \sqrt{t} } ).
\end{equation}
A solution to \eqref{pnew} does exist and has the form
\be\label{solnew}
u= U [\lambda , P ] (r,t) + \tilde \phi, \quad t > t_0
\ee
where $U$ is defined in \eqref{maggio1}, while $\tilde \phi (x,t )$ is  given as in \eqref{deftildephi}
$$
\tilde \phi(x,t) =  \psi (x,t) +    \phi^{in} (x,t)
\quad
{\mbox {where}} \quad
 \phi^{in}(x,t) : =   \eta_{R} (x,t) \hat \phi  (x,t)
$$
and
$
\hat \phi (x,t) := \mu_{0}^{-\frac{1} 2} \phi \left (\frac {x} {\mu_{0}} , t     \right )
.
$
 For any
$ \psi_0 \in C^2 (\R^3 )$ so that
\begin{equation}\label{ultimo}
 |y| \, | \psi_0 (y) | + |y| \, |\nabla  \psi_0 (y) | \leq t_0^{-a} e^{-b |y|} , \quad
\end{equation}
 for some positive constants $a$ and $b$,   the function  $\psi$ is the solution to
 \begin{align} \label{equpsinew}
\partial_t \psi &=  \Delta \psi +  V  \psi  +
  [2\nabla \eta_{R} \nabla_x \hat \phi +  \hat \phi (\Delta_x-\partial_t)\eta_{R} ]\nonumber \\
&+   N [\lambda ]( \tilde \phi  ) +  \bar {\mathcal E}_  + {\mathcal E}_{3}  \inn \R^3 \times [t_0,\infty),
\\
\psi (x, t_0 ) &= \psi_0  , \nonumber
\end{align}
where $V$ is defined as in \eqref{defVmu} with $U$ instead of $U_2$, and
$
 N(\tilde \phi )  =  (U + \tilde \phi)^5 - U^5 - 5\, U^4 \, \tilde\phi.
$
This solution $\psi$ can be described as follows
\begin{equation}\label{maggio24}
\psi (x,t) = \psi_r (x,t) + \psi_{nr} (x,t),
\end{equation}
where $\psi_r $ is a radial function in $|x-p(t)|$, for any $t$, and
\begin{equation}\label{maggio24est}
\left| \psi_r (x,t) \right|\leq C \mu_0^{1\over 2} t^{-{1\over 2}} \, \varphi_0 ({|x| \over \sqrt{t}} ), \quad
\left| \psi_{nr} (x,t) \right|\leq C t^{-\eta + {3\over 2}}\, \varphi_0 ({|x| \over \sqrt{t}} ).
\end{equation}
We refer to \eqref{defvarphi0} for the definition of $\varphi_0$.

On the other hand, the function $\hat \phi$ satisfies
$$
\partial_t \hat \phi =  \Delta \hat \phi+ 5  w_\mu^4 \hat \phi + 5 w_\mu^4 \psi +  {\mathcal E}_{22} (|x-p (t)|,t)    - \nabla U_2 (|x-p (t)| , t) \, \cdot \, \dot p (t) \inn  B_{2R\mu_{0}} (0)  \times [t_0,\infty),
$$
with $\hat \phi (x,t_0) = \mu_0^{-{1\over 2}} (t_0 ) e_0 Z ({x\over \mu_0 (t_0) })$.
In terms of $\phi$, this equation becomes
\begin{align} \label{equ3new}
\mu_{0}^2 \partial_t  \phi  =&  \Delta_y \phi +  5 w^4 \phi
+
  f ( y,t)
 \inn  B_{2R} (0)  \times [t_0,\infty)    \\
\phi (y,t_0 )  &= e_0 Z(y) \nonumber
\end{align}
where
\begin{align*}
f (y,t) &= \mu_{0}^{\frac{5}2}  {\mathcal E}_{22} (|\mu_0 y -p (t)|,t)    - \nabla U_2 (| \mu_0 y -p (t)| , t) \, \cdot \, \dot p (t)  \\
&+ 5 \,  { \mu_{0}^{\frac{1}2} \over (1+\Lambda )^4 } \, w^4 ( {y \over (1+ \Lambda )^2}  ) \psi(\mu_{0} y,t) +  B[\phi] + B^0 [\phi].
\end{align*}
In the above expression, $\psi $ is the solution to \eqref{equpsinew}, while
$B$ and $B^0$ are defined respectively in \eqref{gajardo1} and \eqref{gajardo2}.
The solution $\phi$ exists in the class of functions with $\| \cdot \|_{\nu, a}$-norm bounded (see \eqref{phibound}), as consequence of Proposition \ref{prop0cf}, and a contraction type argument, provided the parameter functions $\la $ and $P$ can be chosen so that
\begin{equation}\label{maxx}
\int_{B_R} f(y,t) Z_j (y) \, dy = 0, \quad {\mbox {for all}} \quad t>t_0, \quad j=0, 1, \ldots , n.
\end{equation}
The system of $(n+1)$ non linear, non local equations in $\la $ and $P$ is solvable for $\la $ and $P$ satisfying \eqref{lambdabound} and \eqref{boundP}. Indeed, equation \eqref{maxx}, for $j=0$, can be treated as we did for equation \eqref{na1} in Sections \ref{secpar}, \ref{nonlocal}, \ref{secpar1}. On the other hand, when $j=1, \ldots , n$, equations \eqref{maxx} are perturbations of
$$
\dot p (t) = \mu_0^{1\over 2} t^{-\kappa +1} \bar u
$$
for some fixed vector $\bar u \in \R^3$. Thus it can be solved for parameters $p(t) = \int_{t_0}^t P(s) \, ds$ satisfying \eqref{boundP}.
This concludes the proof of existence of a positive global solution to \eqref{pnew}.

\medskip
Next we discuss the co-dimension $1$ stability. Let us  observe that the construction of $\phi $, and
$e_0 $ solution to \eqref{equ3new} is possible for any initial condition $\psi_0$ to the outer Problem \eqref{equpsinew}. We have the validity of Lipschitz dependence of  $\phi = \phi [\psi_0 ] $, and $e_0 = e_0 [\psi_0 ]$ in the $C^1$-topology described in \eqref{ultimo}. As a
consequence of the Implicit Function Theorem the maps $\phi [\psi_0 ] $, and $e_0 [\psi_0 ]$ depends in $C^1$-sense on $\psi_0$ in our $C^1$-topology
\eqref{ultimo}, thanks to the corresponding dependence for  $\psi$, $\la$ and $p$.

Let us consider the following map defined  in a small neighborhood of 0 in  $X= C^1(\bar\Omega)$.
$$
F (\psi_0) =    \psi_0 -    (e_{0} [\psi_0]- e_{0}[t_0])  Z_{0}
$$
so that $F[0]=0$, $F$ is differentiable and
$$
D_{\psi_0} F (0) [h]  =  h -    \langle  D_{\psi_0} e_{0} [0], h \rangle  Z_{0}, \quad h\in X.
$$
We have a solution which blows-up as $t\to +\infty$  provided that
\be\label{pq}
u(\cdot,t_0)   =   u^*(\cdot ,t_0) -    e_{0} [0]  Z_{0}    + g
\ee
where $u^* $ is the solution corresponding to $\psi_0 = 0$, and $g= F[\psi_0] $ for any small $\psi_0$.

The vector space of the functionals in $X$ given by  $D_{\psi_0} e_{0} [0]$ has dimension $1$.
We write $W:= {\rm  Ker}\, (D_{\psi_0} e_{0} [0])$ is a space with codimension $1$. Indeed,
we can find a non zero function $u$ such that
$$
X =     W\oplus  < u >.
$$
We consider the operator in a neighborhood of $0$ in  $X$ given  by
$$
G\big( w+   \alpha u \big ) \ =\   \alpha u+  F(w) , \quad \alpha_j\in \R,\quad w\in W.
$$
Then $G$ is of class $C^1$ near the origin, $G(0)=0$  and  $D_{\psi_0} G(0) [h] = h$. By the local inverse theorem, $G$ defines a local $C^1$ diffeormorphism onto a neighborhood of the origin. For all small $g$ we can find smooth functions $\alpha(g)$, $w(g)$ with
$$ \alpha(g) u +  F(w(g)) = g. $$ Thus the set  $\mathcal M$ of functions $F[w]$, $w\in W$ can be described in a neighborhood of $0$ exactly as those $g\in X$
such that
$$
\alpha (g) = 0 .
$$
 This says precisely that $\mathcal M$  is locally
a codimension $1$ $C^1$-manifold, such that if $g$ in \equ{pq} is selected there, then the desired phenomenon takes place.  The proof is concluded. \qed

\setcounter{equation}{0}
\section{Appendix A}\label{appeA}

\begin{proof}[Proof of Lemma \ref{uno}]
We denote by $y_2 (s)$ the solution to \eqref{app0}
 with $\lim_{s \to \infty } s^{2\nu}  y_2 (s) = 1$, and by $y_1(s)$ another solution, linearly independent from $y_2$, defined explicitly by
 \begin{equation}
 \label{app}
 y_1(s) = c \, y_2 (s) \, \int_s^\infty {e^{-{z^2 \over 4} } \over y_2(z)^2 \, z^2 } \, dz,
\end{equation}
for some positive constant $c$ we fix later.
The function $y_1 (s)$ decays fast at infinity, since $y_1 (s) = c_1 e^{-{s^2 \over 4}} \, s^{4\nu - 3} \, \left(1+ o(s^{-1} ) \right)
$, as $s \to \infty$, for some positive constant $c_1$, as a direct consequence  from \eqref{app}.
The function $y_2 (s)$ is  definite for any $s \in (0,\infty)$, and it is positive. Indeed, we first observe that the operator $L_\nu$ satisfies the maximum principe. This is consequence of the fact that the positive function $g_0 (s) = {e^{-{s^2 \over 4}} \over s} $, which solves $L_1 (g_0 ) = 0$, satisfies $L_\nu (g_0 ) <0$ in $(0,\infty )$. With this is mind, we define $\bar g_0 (s) =\int_s^\infty {e^{-{z^2 \over 4}} \over z^2} \, dz$. This is a positive function, which satisfies $L_\nu (\bar g_0 ) = \nu \bar g_0 >0$ in $(0, \infty)$. Thus $\bar g_0$ is a sub solution. Moreover, it is easy to see that $\bar g_0 (R) < y_2 (R)$ for any $R$ large enough. A standard application of the maximum principle thus gives that $y_2$ is positive in $(0, \infty)$.

\medskip
We now claim that $\lim_{s \to 0^+} s \, y_1 (s)$ exists and it is positive.  Write $y_1 (s) = \phi ({s^2 \over 4} )$, $x= {s^2 \over 4}$, from which we get that
$$
x \phi'' + ({3\over 2} + x ) \phi' + \nu \phi = 0 , \quad x \in (0, \infty).
$$
Performing the further change of variables $\phi (x ) = e^{-x} \varphi (x)$, we get that $\varphi$ satisfies
\begin{equation}
\label{manu}
x \varphi'' + ({3\over 2} - x) \varphi' - ({3\over 2} - \nu) \varphi = 0, \quad x \in (0,\infty).
\end{equation}
In \cite{FHV}, Appendix A, it is proven that \eqref{manu} admits polynomial solutions if and only if
${3\over 2} - \nu = -k$, $k=0,1,2,\ldots $. Since ${1\over 2} <\nu < 1 $, this never happens, thus $\varphi$ can not be bounded, as
$x \to 0^+$. On the other hand,  the behavior of the solutions to \eqref{manu}, as $x \to 0^+$, are determined by $x \varphi'' + {3\over 2} \varphi' = 0$, which implies that
the solutions to \eqref{manu} are bounded around $x=0$, or they behave like
 $x^{-{1\over 2}}$ as $x\to 0^+$.
 Combining all the above information, we showed that, for a proper choice of the constant $c$ in \eqref{app}, we get that
 $$
 y_1 (s ) = {1\over s} (1+ o(1) ) , \quad {\mbox {as}} \quad s\to 0.
 $$
To understand further the behavior of $y_1$ around $s=0$, we write $s y_1 (s) = f(s)$, so that
\begin{equation}\label{app1}
f'' + {s\over 2} f' + (\nu  -{1 \over 2} ) f = 0 , \quad s \in (0, \infty).
\end{equation}
Integrating \eqref{app1} between $0$ and $\infty$, and using the fast decay of $y_1$ to $0$ as $s \to \infty$, we compute
\begin{equation}\label{app4}
f' (0) = (\nu -1)  \int_0^\infty f(s) \, ds <0, \quad f'' (0) = {1\over 2} -\nu.
\end{equation}
With this information, we get the estimates \eqref{app11} and \eqref{app12} for $y_1 (s)$.

Since the Wronskian associated to Problem \eqref{app0} is given by a multiple of ${e^{-{s^2 \over 4}} \over s^2}$, we conclude that, since $y_1$ is unbounded as $s\to 0^+$, we have that $y_2 (s)$ is bounded, as $s \to 0^+$. This concludes the proof of the Lemma.

\end{proof}

\begin{lemma}\label{lemacalor}
Let $ h = h(s)$ be a smooth function defined for $s \geq 0$ so that
$$
h(s ) =\left\{ \begin{matrix}
{1\over s} & \quad {\mbox {for}} \quad s\to 0 \\
{1\over s^3} & \quad {\mbox {for}} \quad s\to \infty
\end{matrix}\right.
$$
Then there exists a solution to
\begin{equation}\label{cal1}
\partial_t \psi = \Delta \psi + t^{-\beta } h ({r\over \sqrt{t}} ),
\end{equation}
of the form
\begin{equation}\label{cal2}
\psi (r,t) = t^{-\beta +1} \varphi ({r\over \sqrt{t}} ), \quad {\mbox {with}} \quad
\varphi (s) = \left\{ \begin{matrix}
s & \quad {\mbox {for}} \quad s\to 0 \\
{1\over s^3} & \quad {\mbox {for}} \quad s\to \infty
\end{matrix}\right. .
\end{equation}
\end{lemma}

\begin{proof}
We look for a solution to \eqref{cal1} of the form  $\psi (r,t) = t^{-(\beta -1)} \varphi({r\over \sqrt{t}} )$. Thus $\varphi$ satisfies
$$
\varphi '' + \left({2\over s} + {s\over 2} \right) \varphi' + (\beta -1) \varphi + h(s ) = 0.
$$
We look for a solution of the above equation of the form
$$
\varphi (s) = z(s) \, y_1 (s)
$$
where $y_1$ solves $y_1'' + \left({2\over 2} + {s\over 2} \right) y_1' + (\beta -1) y_1 =0$, and
$y_1 (s ) \sim \left\{ \begin{matrix} {1\over s} & \quad {\mbox {as}} \quad s \to 0\\
e^{-{s^2 \over 4}} \, s^{4(\beta -1) -3} & \quad {\mbox {as}} \quad s \to \infty
\end{matrix} \right. $. The existence of $y_1$ is consequence of Lemma \ref{uno}. A direct computation gives
$$
z(s ) = -\int_0^s {e^{-{\eta^2 \over 4}} \over y_1 (\eta)^2 \eta^2 } \, \left( \int_0^\eta h(x) y_1 (x) \, x^2 \, e^{x^2 \over 4} \, dx \right) \, d\eta .
$$
One can easily see that
$$
z(s) \sim \left\{ \begin{matrix} s^2 & \quad {\mbox {as}} \quad s \to 0 \\
e^{s^2 \over 4} s^{-4 (\beta -1)} & \quad {\mbox {as}} \quad s\to \infty
\end{matrix}.
\right.
$$
This fact gives \eqref{cal2}, and concludes the proof of the Lemma.
\end{proof}

\medskip
\noindent
\begin{proof}[Proof of \eqref{mmee2}.]

For $x \in B_{2R}$, we shall prove
\begin{equation}\label{mmee2n}
 \phi_0 (\mu_0 x , t) - \phi_0 (0,t) = \alpha (t) |\mu_0 x |^\sigma \, \Pi (t) \, \Theta (|x|) ,
 \end{equation}
 for some $\sigma \in (0,1)$. Here $\Pi = \Pi(t)$ denotes a smooth and bounded function of $t$, and $\Theta$ a smooth and bounded function of $x$.

We have
\begin{align*}
\phi_0 (\mu_0 x,t) - \phi_0 (0,t) &=\int_{t_0}^t {1\over (4\pi (t-s) )^{3\over 2} } \, \int_{\R^3} \left[ e^{-{|x-y|^2 \over 4 (t-s) }}  - e^{-|y|^2 \over 4 (t-s)} \right] \, {\alpha (s) \over |y|} \, {\bf 1}_{\{ r< M\}} \,  dy \, ds \\
&= {1\over 2} \int_{t_0}^t \int {\beta' (s) \over (t-s)^{1\over 2}} \, \left[e^{-|z-{\mu_0 x \over 2\sqrt{t-s}}|^2} - e^{-|z|^2} \right] \, {1 \over |z|} \, {\bf 1}_{\{ |z|< {M \over 2\sqrt{t-s}}\}} \,  dy \, ds \\
&= I + II
\end{align*}
where
$$
I= \int_{t_0}^{t- ({\mu_0 x \over 2m})} \int {\beta' (s) \over (t-s)^{1\over 2}} \, \left[e^{-|z-{\mu_0 x \over 2\sqrt{t-s}}|^2} - e^{-|z|^2} \right] \, {1 \over |z|} \, {\bf 1}_{\{ |z|< {M \over 2\sqrt{t-s}}\}} \,  dy \, ds.
$$
We start estimating $II$. We observe that, if $t-  ({\mu_0 x \over 2m}) <s<t$, then ${\mu_0 |x| \over 2\sqrt{t-s} } >m$.
We write
$$
II = II_1 + II_2 + II_3
$$
where
$$
II_j = \int_{t- ({\mu_0 x \over 2m})}^t \int_{D_j } {\beta' (s) \over (t-s)^{1\over 2}} \, \left[e^{-|z-{\mu_0 x \over 2\sqrt{t-s}}|^2} - e^{-|z|^2} \right] \, {1 \over |z|} \, {\bf 1}_{\{ |z|< {M \over 2\sqrt{t-s}}\}} \,  dy \, ds.
$$
with
$$
D_1 = \{ z \, : \, |z-{\mu_0 x \over 2\sqrt{t-s}}| <{1\over 4} {\mu_0 |y| \over 2\sqrt{t-s}} \}, \quad
D_2 = \{ z \, : \, |z| <{1\over 4} {\mu_0 |y| \over 2\sqrt{t-s}} \}
$$
and $D_3$ the complement of the two above regions.

We start estimating $II_1$. We see that
\begin{align*}
\int_{D_1 } \,  e^{-|z-{\mu_0 x \over 2\sqrt{t-s}}|^2}  \, {1 \over |z|} \, {\bf 1}_{\{ |z|< {M \over 2\sqrt{t-s}}\}} \,  dy
= \int e^{-|\bar z|} {1\over |\bar z +{\mu_0 x \over 2 \sqrt{t-s}}|} \, d\bar z = c {2\sqrt{t-s} \over \mu_0 |x| },
\end{align*}
for some constant $c$, as a direct application of Dominated Convergence Theorem. Thus
\begin{align*}
\int_{t- ({\mu_0 x \over 2m})}^t \int_{D_1 } \,  e^{-|z-{\mu_0 x \over 2\sqrt{t-s}}|^2}  \, {1 \over |z|} \, {\bf 1}_{\{ |z|< {M \over 2\sqrt{t-s}}\}} \,  dy\, ds
= {2c \over \mu_0 |x|} \int_{t- ({\mu_0 x \over 2m})}^t \, \sqrt{t-s} ds = c' (\mu_0 |x|)^{1\over 2}.
\end{align*}
On the other hand, for any $z$ in $D_1$, one has $|z| > {1\over 4} {\mu_0 |x| \over 2\sqrt{t-s}}$, and hence we can bound
$$
\left| \int_{D_1} e^{-|z|^2 } {1\over |z|} \, dz \right| \leq c \left[ {\sqrt{t-s} \over \mu_0 |x| } \right]^\sigma,
$$
for any $\sigma >0$. We take $\sigma >1$, so that
\begin{align*}
|\int_{t_0}^{t- ({\mu_0 x \over 2m})} & \int_{D_1}  e^{-|z|^2}  \, {1 \over |z|} \, {\bf 1}_{\{ |z|< {M \over 2\sqrt{t-s}\}}} \,  dy\, ds  | \leq
{1\over (\mu_0 |x| )^\sigma }\left| \int_{t_0}^{t- ({\mu_0 x \over 2m})}  (t-s)^{{\sigma \over 2} - {1\over 2}}\, ds \right| \leq c' \mu_0 |x|
\end{align*}
Thus we conclude that
$$
\left| II_1 \right| \lesssim \beta' (t) (\mu_0 |x| )^{1\over 2}.
$$
Arguing in a similar way, one finds the same type of estimate for $II_2$. In the third region $D_3$, we have that
$$
|z| > {1\over 4} {\mu_0 |x| \over 2\sqrt{t-2}} , \quad |z-{\mu_0 x \over 2\sqrt{t-s}} | > {1\over 4} {\mu_0 |x| \over 2\sqrt{t-s}},
$$
so that again one gets the estimate
$$
\left| II_3 \right| \lesssim \beta' (t) \mu_0 |x|.
$$
Let us now consider the interval of time $t_0 < s < t- \left( {\mu_0 |x| \over 2m \sqrt{t-s}} \right)^2 $, region where one has ${\mu_0 |x| \over 2\sqrt{t-s}} <m$.
We decompose
$$
I= III + IV
$$
where
$$
III= \int_{t_0}^{t-1}  \int {\beta' (s) \over (t-s)^{1\over 2}} \, \left[e^{-|z-{\mu_0 x \over 2\sqrt{t-s}}|^2} - e^{-|z|^2} \right] \, {1 \over |z|} \, {\bf 1}_{\{ |z|< {M \over 2\sqrt{t-s}\}}} \,  dy \, ds
$$
We start with $IV$, where we expand in Taylor
\begin{align*}
IV & = \int_{t-1 }^{t- \left( {\mu_0 |x| \over 2m \sqrt{t-s}} \right)^2}  \int {\beta' (s) \over (t-s)^{1\over 2}} \, \left[e^{-|z-{\mu_0 x \over 2\sqrt{t-s}}|^2} - e^{-|z|^2} \right] \, {1 \over |z|} \, {\bf 1}_{\{ |z|< {M \over 2\sqrt{t-s}\}}} \,  dy \, ds  \\
&= \beta' (t) \int_{t-1 }^{t- \left( {\mu_0 |x| \over 2m \sqrt{t-s}} \right)^2} {\mu_0 |x| \over t-s } \left( \int {e^{-|z|^2 } \over |z| } \, dz \right) \, ds = \beta' (t) \log \left ( {t- \left( {\mu_0 |x| \over 2m \sqrt{t-s}} \right)^2 \over t  } \right) \mu_0 |x| \\
&= \beta' (t) \mu_0 |x| [\log (\mu_0 |x| )] = \beta' (t) (\mu_0 |x|)^{\sigma},
\end{align*}
for some positive $\sigma <1$.
Finally, we consider $III$. Again, after a Taylor expansion, we have
$$
III = \mu_0 |x| \, \int_{t_0}^{t-1} {\beta' (s) \over (t-s )} \, ds
= \mu_0 |x| \int_{t_0}^{t-1} {\beta' (s) \over t-s} \, ds.
$$
Collecting the previous estimates, we conclude with the validity of \eqref{mmee2n}.
\end{proof}

\bigskip

\setcounter{equation}{0}
\section{Appendix B}\label{appeB}

\medskip
\begin{proof}[Proof of Lemma \ref{des1}]

Throughout the proof of the Lemma,  we denote by $q_i = q_i (s)$, for any interegr $i$, a smooth real function, with the property that ${d \over (d s)^j} q_i(0) = 0$, for $j <i$, and ${d \over (d s)^i} q_i(0)  \not= 0$.
With  $\Theta = \Theta (r)$ we intend a smooth function of the space variable, which is uniformly bounded.
Also, $\Pi = \Pi (t) $ stands for a smooth function of the time variable, which is uniformly bounded in $t \in (0, \infty)$. The explicit expressions of these functions change from line to line, and also within the same line.

\medskip
Let $R_0 = r_0 \sqrt{t}$. A simple computation gives the explicit expression of the error ${\mathcal E}_1$ in \eqref{error1}
\begin{eqnarray}\label{e1}
{\mathcal E}_1 (r,t) &=& {\mathcal E}_{{\mbox {in}}}^1 \eta ({r\over R_0} )
+{\mathcal E}_{{\mbox {out}}}^1 \left( 1- \eta ({r\over R_0} )
\right) \\
&+& \underbrace{ R_0^{-2}\left( \uin -  \uout \right) \Delta \eta ({r\over R_0} ) + 2
R_0^{-1} \nabla \left( \uin -  \uout \right) \cdot \nabla \eta ({r\over R_0} ) }_{:= \bar {\mathcal E}_1} \nonumber \\
&+& \underbrace{ \left( \uin -  \uout \right) \, {R_0' \over R_0^2} \,   \eta' ({r\over R_0} ) }_{:= \hat {\mathcal E}_1} \nonumber
\end{eqnarray}
where
\begin{equation}\label{Einoutdef}
{\mathcal E}_{{\mbox {in}}}^1= \Delta \uin + \uin^5 - \partial_t \uin
, \quad  {\mbox {and}} \quad {\mathcal E}_{{\mbox {out}}}^1= \Delta \uout + \uout^5 - \partial_t \uout.
\end{equation}
We start analyzing $ {\mathcal E}_{{\mbox {in}}}^1$, getting
\begin{eqnarray}\label{Ein1}
{\mathcal E}_{\mbox {in}}^1 (r,t ) &=& \mu_0' \left[ \Delta \psi_1 + 5 w_\mu^4 \psi_1 \right] -
\mu' {\partial w_\mu \over \partial \mu} \nonumber \\
&+& (w_\mu + \mu_0'\psi_1 )^5 - w_\mu^4 - 5 w_\mu^4 \mu_0' \psi_1
-\mu_0 '' \psi_1 - \mu_0'\mu' {\partial \psi_1 \over \partial \mu} \nonumber \\
&=& \left( \mu' - \mu_0'\right) \, \mu^{-{3\over 2}} Z_0 ({r \over \mu} )+
\left[
 (w_\mu + \mu_0'\psi_1 )^5 - w_\mu^5 - 5 w_\mu^4 \mu_0' \psi_1 \right] \nonumber \\
&-& \mu_0 '' \psi_1 - \mu_0'\mu' {\partial \psi_1 \over \partial \mu}.
\end{eqnarray}
Now we write
\begin{eqnarray*}
\left( \mu' - \mu_0'\right) \, \mu^{-{3\over 2}} Z_0 ({r \over \mu} )&=& \left[ 2 (\mu^{1\over 2} - \mu_0^{1\over 2} )' +
(\mu^{1\over 2} - \mu_0^{1\over 2} ) \mu_0^{-1} \mu_0' \right] \, \mu^{-1} Z_0 ({r\over \mu} ) \\
&-& { (\mu^{1\over 2} - \mu_0^{1\over 2} )^2 \over \mu^{1\over 2}} \mu_0^{-1} \mu_0' \,  \mu^{-1} Z_0 ({r\over \mu} ).
\end{eqnarray*}
Taking into account that $Z_0 (s) = {3^{1\over 4} \over 2} \, {1\over s} + O({1\over s^3}) $, as $s\to \infty$,
 it is convenient to write
\begin{eqnarray*}
\left[ 2 (\mu^{1\over 2} - \mu_0^{1\over 2} )' +
(\mu^{1\over 2} - \mu_0^{1\over 2} ) \mu_0^{-1} \mu_0' \right] & & \mu^{-1} Z_0 ({r\over \mu} )\,  =
 {\alpha (t)   \over \mu + r}
\\
&+ & \left[ 2 (\mu^{1\over 2} - \mu_0^{1\over 2} )' +
(\mu^{1\over 2} - \mu_0^{1\over 2} ) \mu_0^{-1} \mu_0' \right] \mu^{-1} \left[ Z_0 ({r \over \mu} ) - {3^{1\over 4} \over 2}  {\mu \over \mu+ r} \right]  ,
\end{eqnarray*}
where $\alpha $ is defined in \eqref{defaaa}.
We  decompose \eqref{Ein1} as
\begin{equation}\label{Ein11}
{\mathcal E}_{\mbox {in}}^1 (r,t ) =   {\alpha (t) \over \mu + r} \, + \bar {\mathcal E}_{\mbox {in}}^1 (r,t ), \end{equation}
where $\bar {\mathcal E}_{\mbox {in}}^1$ is explicitly given by
\begin{eqnarray}\label{barEin1}
\bar {\mathcal E}_{\mbox {in}}^1 (r,t )
&=& - { (\mu^{1\over 2} - \mu_0^{1\over 2} )^2 \over \mu^{1\over 2}} \mu_0^{-1} \mu_0'\,   \mu^{-1} Z_0 ({r\over \mu} ) - \mu_0 '' \psi_1
+
\left[
 (w_\mu + \mu_0'\psi_1 )^5 - w_\mu^5 - 5 w_\mu^4 \mu_0' \psi_1 \right]  \nonumber \\
&+&  \left[ 2 (\mu^{1\over 2} - \mu_0^{1\over 2} )' +
(\mu^{1\over 2} - \mu_0^{1\over 2} ) \mu_0^{-1} \mu_0' \right] \mu^{-1} \left[ Z_0 ({r \over \mu} ) - {3^{1\over 4} \over 2}  {\mu \over \mu + r} \right]-\mu_0'\mu' {\partial \psi_1 \over \partial \mu}   \, \nonumber \\
&=& \sum_{j=1}^5 e_j.
\end{eqnarray}
We observe now that $(e_1 + e_2 +e_3 ) \eta ({r\over R_0}) $  can be described as sum of functions of the form
\begin{equation}\label{uffa}
{\mu_0^{1\over 2} t^{-2} R_0^2 \over \mu_0 + r} \, q_0 (\Lambda ) \, \Pi (t) \,  \Theta (r) , \quad
{\mu_0^{-{1\over 2}} t^{-1} \over \mu_0 + r} \, q_2 (\Lambda ) \, \Pi (t) \, \Theta (r),
\end{equation}
where $q_0$ is a smooth function with $q(0) \not= 0$, while $q_2$ is a smooth function with
$q_2(0)= q_2' (0) = 0 $, and $q_2'' (0) \not= 0$.
On the other hand,   we see that
\begin{equation}\label{uffaaa}
e_4 = { \alpha (t) \mu_0^2 \over \mu_0^3 + r^3} \, \Pi(t) \, \Theta (r),
\end{equation}
and $e_5$
\begin{equation}\label{uffaaaa}
 {\mu_0^{1\over 2} t^{-1} \over \mu_0 + r} \left[ R_0^2 \Lambda'+ R_0^2 t^{-1} q_1 (\Lambda ) \right] \, \Pi (t ) \, \Theta (r),
\end{equation}
where $q_1$ is a smooth function with $q_1 (0) = 0 $, $q_1' (0) \not= 0$.
Under assumption \eqref{lambdabound} and  combining \eqref{Ein11}-\eqref{uffa}-\eqref{uffaaa}-\eqref{uffaaaa}, we find that
\begin{align*}
\left| \bar {\mathcal E}_{\mbox {in}}^1 \eta  \right|_{\infty , B(x,1)\times [t,t+1]}   \lesssim \mu_0^{1\over 2} t^{-{3\over 2} } h_0 ({r\over \sqrt{t}} ) , \quad r=|x|.
\end{align*}
Since \eqref{defaaa}, we observe that
$$
\left| {\alpha (t)   \over \mu + r} \left( 1- \eta ({r\over R_0} ) \right) \right| \lesssim \mu_0^{3\over 2} t^{-{3\over 2} } h_0 ({r\over \sqrt{t}} ) , \quad r=|x|.
$$
Let us fix $\lambda_1$ and $\lambda_2$ satisfying \eqref{lambdabound}. We write, for some $\bar \lambda = s \lambda_1 + (1-s) \lambda_2 $, $s \in (0,1)$,
$$
\left( \bar {\mathcal E}_{\mbox {in}}^1  [\lambda_1] - \bar {\mathcal E}_{\mbox {in}}^1  [\lambda_2] \right)  \eta ({r\over R_0})  =\left(  D_{\lambda} \bar {\mathcal E}_{\mbox {in}}^1  [\bar \lambda]  [\lambda_1 - \lambda_2] \right) \eta ({r\over R_0}), \quad {\mbox {with}} \quad
D_{\lambda} \bar {\mathcal E}_{\mbox {in}}^1    [\bar \lambda] = \sum_{j=1}^5 ( D_{\lambda } e_j  ) [\bar \lambda],
$$
where the $e_j$ are defined in \eqref{barEin1}. Let us consider $e_1$. We have that
$$
(D_{\lambda } e_1) [\bar \lambda ] = 2 \mu_0 (1+ \Lambda ) D_\mu (e_1) [\bar \lambda ].
$$
Direct computation give that
$$
\left| D_\mu (e_1) [\bar \lambda ]  (r,t) \right| \lesssim {\mu_0^{-{1\over 2} } t^{-1} \over \mu_0 + r} q_0 (\bar \lambda ) \, \Pi (t) \,  \Theta (r).
 $$
 We combine the above estimates to get
 \begin{align*}
 \left| e_1 [\la_1] - e_1 [\la_2] \right| \eta ({r \over R_0} ) & \leq \mu_0 {\mu_0^{-{1\over 2}} t^{-1} \over \mu_0 + r } \, |\la_1 - \la_2 | \eta ({r\over R_0} ) \\
 &\leq C \left( \mu_0 (t) t^{-1} \right) \, \mu_0^{3\over 2} (t) t^{-{3\over 2}} h_0 ({r\over \sqrt{t}} ) \| \la_1 - \la_2 \|_\sharp \\
 &\leq C \left( \mu_0 (t_0) t_0^{-1} \right) \, \mu_0^{3\over 2} (t) t^{-{3\over 2}} h_0 ({r\over \sqrt{t}} ) \| \la_1 - \la_2 \|_\sharp.
 \end{align*}
 Choosing $t_0$ large if necessary, we get $C \left( \mu_0 (t_0) t_0^{-1} \right) <1$.
 Similar estimates can be obtained for the other terms $e_2$, $ \ldots $, $e_5$.
 Thus we get
 $$
 \left| \left( \bar {\mathcal E}_{\mbox {in}}^1  [\lambda_1] - \bar {\mathcal E}_{\mbox {in}}^1  [\lambda_2] \right)  \chi   \right|_{\infty , B(x,1)\times [t,t+1]}
 \leq c_1^o \mu_0^{1\over 2} t^{-{3\over 2}} h_0 ({r\over \sqrt{t}} ) \, \|\lambda_1 - \lambda_2 \|_\sharp,
 $$
 for some constant $c_1^o$ which can be made arbitrarily small, if $t_0$ is chosen large. Also,
we have
$$
[ \bar {\mathcal E}_{\mbox {in}}^1  [\lambda_1]  - \bar {\mathcal E}_{\mbox {in}}^1  [\lambda_2] ]_{0,\sigma, [t,t+1] } \leq c_1^o \mu_0^{1\over 2} t^{-{3\over 2}} h_0 ({r\over \sqrt{t}} ) \,  [\lambda_1 - \lambda_2 ]_{0, \sigma , [t,t+1]}  .
$$

\medskip
Let us now describe ${\mathcal E}_{\mbox {out}}^1$.  A first observation is that, for any value of $\gamma$, we immediately see that ${\mathcal E}_{\mbox {out}}^1$  does not depend on $\lambda$.
On the other hand, if $1< \gamma \leq 2$ the expression for ${\mathcal E}_{\mbox {out}}^1$ becomes
$$
{\mathcal E}_{\mbox {out}}^1 (r,t) = \uout^5 ,
$$
so that we directly get
\begin{equation}\label{uffa02}
\left| {\mathcal E}_{\mbox {out}}^1 \, (1-\chi ({r\over R_0})) \right|\leq C   {\mu_0^{5\over 2} \over r^5} \, {\bf 1}_{\{ r> R_0^{-1}\} }\, .
\end{equation}
Let us consider now  $\gamma >2$. In this case,  the expression of ${\mathcal E}_{\mbox {out}}^1$ is a bit more involved
\begin{eqnarray}\label{ee1}
{\mathcal E}_{\mbox {out}}^1 (r,t) &=&  \eta ({r\over t} ) (\uout^1 )^5
+ \left( 1- \eta ({r\over t} ) \right) \, A \, \left[ {\gamma (\gamma-1) \over r^{\gamma +2} } + {A^4 \over r^{5\gamma}} \right]
 \\
&+& \underbrace{ t^{-2} \left( \uout^1 -  \uout^2 \right) \Delta \eta ({r\over t} ) + 2
t^{-} \nabla \left( \uout^1 -  \uout^2 \right) \cdot \nabla \eta ({r\over t} ) }_{:= \bar {\mathcal E}_1^{out}} \nonumber \\
&+& \underbrace{ \left( \uout^1 -  \uout^2 \right) \,  t^{-2}   \,   \eta' ({r\over t} ) }_{:= \hat {\mathcal E}_1^{out}} . \nonumber
\end{eqnarray}
A close analysis of each one of the terms appearing in \eqref{ee1} gives that
\begin{eqnarray}\label{uffa2}
\left| {\mathcal E}_{\mbox {out}}^1 \, (1-\eta ({r\over R_0})) \right| & \leq & C \Biggl\{
{t^{-(\gamma -1)} \over r^3 } \, {\bf 1}_{\{ r > t \}} \nonumber \\
&+& {t^{-2} \mu_0^{1\over 2} \over r} \,  {\bf 1}_{\{ t < r < 2 t \}}
+ {t^{-{5\over 2}} \over r^5} \, {\bf 1}_{\{r_0 \sqrt{t} < r < t \}} \Biggl\}.
  \end{eqnarray}
From \eqref{uffa02}-\eqref{ee1} and \eqref{uffa2}, we obtain that
$$
\left| {\mathcal E}_{\mbox {out}}^1 \left(1 - \chi ({r\over R_0} ) \right) \right| \lesssim \left\{ \begin{matrix} \mu_0^{1\over 2} t^{-{3\over 2}} h_0 ({r\over \sqrt{t}} ) & \quad {\mbox {if}} \quad 1<\gamma \leq 2 \\
t^{-2} h_0 ({r\over \sqrt{t}} )  & \quad {\mbox {if}} \quad \gamma > 2.
\end{matrix} \right.
$$

\medskip
Going back to \eqref{e1}, we are left with the description of
 $\bar {\mathcal E}_1 = \bar {\mathcal E}_1 [\lambda]$ and $\hat {\mathcal E}_1 [\lambda] $. Directly we check
\begin{equation}\label{estbarE1}
\left| \bar {\mathcal E}_1 (r,t) \right| , \, \left| \hat {\mathcal E}_1 (r,t)  \right|\leq C R_0^{-2} \,  \,  {\mu_0^{1\over 2} \over r} \, {\bf 1}_{\{R_0 < r < 2R_0\}},
\end{equation}
for some positive constant $C$. This gives right away
$$
\left| \bar {\mathcal E}_1 + \hat {\mathcal E}_1 \right|\lesssim \mu_0^{1\over 2} t^{-{3\over 2} } h_0 ({r\over \sqrt{t}} ) .
$$
Let us fix $\lambda_1$ and $\lambda_2$ satisfying \eqref{lambdabound}. We write, for some $\bar \lambda = s \lambda_1 + (1-s) \lambda_2 $, $s \in (0,1)$,
$$
\bar {\mathcal E}_1 [\lambda_1] (r,t) - \bar {\mathcal E}_1 [\lambda_2] (r,t) = D_{\lambda} \bar {\mathcal E}_1 [\bar \lambda]  [\lambda_1 - \lambda_2] (r,t),
$$
where
$$
D_{\lambda} \bar {\mathcal E}_1 [\bar \lambda] = R_0^{-2} (\partial_{\lambda} \uin [\bar \lambda] )  \Delta \eta ({r\over R_0}) +
2R_0^{-1} \nabla \left(( \partial_{\lambda} \uin )[\bar \lambda] \right) \cdot \nabla \eta ({r \over R_0} ).
$$
Since in the region we are considering
$$
\partial_{\lambda} \uin [\bar \lambda] = 2 \mu_0 (1+ \Lambda ) (\partial_{\mu} \uin ) [\bar \lambda], \quad | (\partial_{\mu} \uin ) | \leq c {\mu_0^{-{1\over 2}} \over r},
$$
we have
\begin{align*}
\left| \bar {\mathcal E}_1 [\lambda_1] (r,t) - \bar {\mathcal E}_1 [\lambda_2]  \right|_{\infty , B(x,1)\times [t,t+1]}  & \leq \left( \mu_0 (t_0) t_0^{-1} \right)  \mu_0^{1\over 2} t^{-{3\over 2}} h_0 ({r\over \sqrt{t}} ) \| \lambda_1 -\lambda_2 \|_\sharp \\
&\leq c_1 \mu_0^{1\over 2} t^{-{3\over 2}} h_0 ({r\over \sqrt{t}} ) \| \lambda_1 -\lambda_2 \|_\sharp ,
\end{align*}
for some constant $c_1 \in (0,1)$, provided $t_0$ is large enough.
Furthermore, we also have, for any $t >t_0$,
$$
[ \bar {\mathcal E}_1 [\lambda_1]  - \bar {\mathcal E}_1 [\lambda_2] ]_{0,\sigma, [t,t+1] } \leq c_1 \mu_0^{1\over 2} t^{-{3\over 2}} h_0 ({r\over \sqrt{t}} ) \, \left(  [\lambda_1 - \lambda_2 ]_{0, \sigma , [t,t+1]} \right) ,
$$
with again $c_1 \in (0,1)$.
Collecting all the previous estimates, we get the proof of the Lemma.
\end{proof}

\medskip
\begin{remark}\label{rrr1} From the proof of the result, we also get that the constants ${\bf c}$ in \eqref{marte0001} and \eqref{marte0002} can be made as small as one needs, provided that the initial time $t_0$ is chosen large enough.
\end{remark}.

\bigskip

\setcounter{equation}{0}
\section{Appendix C}\label{appeC}

\medskip
\begin{proof}[Proof of Lemma \ref{des2}]

Under the assumptions \eqref{lambdabound} on  $\lambda$, we get that, for any $r>0$ and $t>t_0$,
\begin{equation}\label{marte111}
\left| {\mathcal E}_{2,1} (r,t) \right| +  \left[ {\mathcal E}_{2,1} (r,t) \right]_{0,\sigma , [t,t+1]}   \lesssim \mu_0^{1\over 2} t^{-{3\over 2}} h_0 ({r\over \sqrt{t}}),
\end{equation}
where $h_0$ is given by \eqref{defh}, and also estimates similar to \eqref{marte0001} and \eqref{marte0002} for $\partial_\lambda {\mathcal E}_{2,1}$.
These estimates follow from \eqref{marte0}-\eqref{marte0001}, \eqref{defaaa} and
from
$$
\left| {\alpha (t) \over \mu + r} \left( \eta ({r\over R_0} ) - {\bf 1}_{\{ r < 2M\} }\right) \right|\leq |\alpha (t) | t^{-{1\over 2}} h_0 ({r\over \sqrt{t}}).
$$
Here we use again $R_0 = r_0 \sqrt{t}$.
Furthermore, in the region where $\eta ({r\over R_0} ) - {\bf 1}_{\{ r < 2M\} } \not= 0$, the above function is regular enough to have
$$
 [{\alpha (t) \over \mu + r} \left( \eta ({r\over R_0} ) - {\bf 1}_{\{ r < 2M\} }\right) ]_{0, \sigma , B(x,1)\times [t, t+1] }\leq |\alpha (t) | t^{-{1\over 2}} h_0 ({r\over \sqrt{t}}), \quad r=|x|.
$$
Using \eqref{alphabound}, we get \eqref{marte111}. Let us consider now ${\mathcal E}_{22} (1-\eta_R ) (r,t )  $.
We claim that
\begin{equation}\label{marte01}
 \left\| {\mathcal E}_{22} (1-\eta_R ) (r,t )  \right\|_{* } \leq c_2.
\end{equation}
Given $d>1$, define
$
 h_* (s) =  \left\{ \begin{matrix}
{1\over s} & \quad {\mbox {for}} \quad s\to 0 \\
{1\over s^d} & \quad {\mbox {for}} \quad s\to \infty
\end{matrix}\right. .
$
Arguing as in the proof of Lemma \ref{lemacalor}, we get the existence of $\psi_*$ so that
$$
\partial_t \psi_* = \Delta \psi_* + \mu_0^{1\over 2} t^{-{3\over 2}} h_* \left( {r\over \sqrt t} \right) , \quad
{\mbox {with}} \quad \psi_* (r,t ) = \mu_0^{1\over 2} t^{-{1\over 2}} \varphi_* ({r\over \sqrt{t}} )
,\quad
\varphi (s) = \left\{ \begin{matrix}
s & \quad {\mbox {for}} \quad s\to 0 \\
{1\over s^d} & \quad {\mbox {for}} \quad s\to \infty
\end{matrix}\right. .
$$
Comparing the above equation and the equation satisfied by $\phi_0$, and using the maximum principle, we obtain that, in the region where $(1-\eta_R) \not= 0$,
\begin{equation}\label{maso1}
\left| \phi_0 (x,t) \right| \leq \| \lambda \|_\sharp \mu_0^{1\over 2} t^{-{1\over 2}} \varphi_* ({r\over \sqrt{t}} ).
\end{equation}
We proceed now with the estimate of $(1-\eta_R) {\mathcal E}_{22}$. A Taylor expansion gives the existence of $s^* \in (0,1)$, so that
$$
{\mathcal E}_{22} (r,t) = 5 (U_1 + s^* \phi_0 )^4 \, \phi_0.
$$
Let $\bar M$ be a large fixed number. From \eqref{DefU} and \eqref{maso0}, we see that, if $r < \bar M \sqrt{t}$,
$$
\left| (1-\eta_R ) {\mathcal E}_{22} \right| \lesssim w_\mu^4 \phi_0
\lesssim  R^{-2}  \mu_0^{1\over 2} t^{-{3\over 2}} h_0 ({r\over \sqrt{t}}) .$$
On the other hand, thanks to \eqref{maso1} we see that, for $r> \bar M \sqrt{t}$, we get
$$
\left| (1-\eta_R ) {\mathcal E}_{22} \right| \lesssim ( \phi_0 )^5 \lesssim \mu_0^{5\over 2} t^{-{5\over 2} } h_0 ({r\over \sqrt{t}})
 .$$
Thus we get the $L^\infty$ bound in estimate \eqref{marte01}. The control on the H\"{o}lder norm contained in \eqref{marte0001new} and \eqref{marte0002new} follows arguing as in the proof of \eqref{marte0001}-\eqref{marte0002} in the proof of Lemma \ref{des1}, and from the assumption on $\lambda$ in \eqref{lambdabound}. We leave the details to the reader.

\end{proof}

\end{document}